\documentclass[12pt, reqno]{amsart}
\usepackage{amsmath,amssymb,amsthm,amsxtra, setspace}
\usepackage{amssymb}
\usepackage{mathrsfs}
\usepackage{alltt}
\usepackage{graphicx,type1cm,xcolor}
\usepackage{mathtools, accents}
\usepackage{mathrsfs}
\usepackage{hyperref}
\usepackage{alltt}
\usepackage{fouridx}
\usepackage{dsfont}
\usepackage{cancel}
\usepackage[mathcal]{euscript}
\usepackage[margin=1in]{geometry}
\usepackage{shadethm}
\usepackage{float}
\usepackage{graphicx}
\usepackage{orcidlink}
\usepackage{xpatch}

\usepackage{todonotes}

\allowdisplaybreaks

\newcommand{\eps}{\varepsilon}

\numberwithin{equation}{section}

\newcommand{\wtilde}{\widetilde}
\newshadetheorem{Theorem}{Theorem}
\newshadetheorem{Proposition}{Proposition}
\newshadetheorem{Lemma}{Lemma}
\newshadetheorem{Assumption}{Assumption}
\newshadetheorem{Corollary}{Corollary}
\newshadetheorem{Definition}{Definition}
\newshadetheorem{Example}{Example}
\newshadetheorem{Remark}{Remark}
\newshadetheorem{Note}{Note}
\newshadetheorem{Hypothesis}{Hypothesis}
\newshadetheorem{Explanation}{Explanation}
\newtheorem{theorem}{Theorem}[section]
\newtheorem{lemma}[theorem]{Lemma}
\newtheorem{proposition}[theorem]{Proposition}

\newtheorem{corollary}[theorem]{Corollary}
\newtheorem{definition}[theorem]{Definition}
\newtheorem{example}[theorem]{Example}
\newtheorem{remark}[theorem]{Remark}

\newcommand{\embed}{\hookrightarrow}

\def\Ls{\mathcal{L}}

\def\A{\mathcal{A}}

\def\Fn{\mathcal{F}}
\def\En{\mathscr{E}}

\def\bA{\mathscr{A}}

\def\a{\mathrm{a}}

\def\C{\mathcal{C}}

\def\Bb{\mathbb{B}}
\def\0{\boldsymbol{0}}

\def\Y{\mathbb{Y}}

\def\N{\mathbb{N}}

\def\bM{\mathcal{M}}

\def\U{\mathcal{U}}

\def\bO{\mathcal{O}}

\def\Gn{\mathfrak{G}}
\def\Vn{\mathcal{V}}
\def\Hn{\mathcal{H}}
\def\Wn{\mathcal{W}}
\def\Xn{\mathcal{X}}

\def\gn{\mathfrak{g}}
\def\Jn{\mathcal{J}_{\mu}}

\def\Bb{\mathfrak{B}}

\newcommand{\norm}[1]{\left\Vert#1\right\Vert}
\newcommand{\abs}[1]{\left\vert#1\right\vert}

\newcommand{\R}{\mathbb{R}}

\newcommand{\loc}{\mathop{\mathrm{loc}}}

\newcommand{\Vs}{\mathcal{V}}
\renewcommand{\d}{\/\mathrm{d}\/}

\newcommand{\dist}{\mathrm{dist}}

\let\originalleft\left
\let\originalright\right
\renewcommand{\left}{\mathopen{}\mathclose\bgroup\originalleft}
\renewcommand{\right}{\aftergroup\egroup\originalright}
\newcommand{\Nn}{\mathcal{N}}

\newcommand{\ip}[2]{\fourIdx{}{0}{}{\!x}{\mathcal{ X}}}

\newcommand\dela[1]{}

\hypersetup{colorlinks=true,%
	citecolor=red,%
	filecolor=blue,%
	linkcolor=blue,%
}

\makeatletter   
\xpatchcmd{\@tocline}
{\hfil\hbox to\@pnumwidth{\@tocpagenum{#7}}\par}
{\ifnum#1<0\hfill\else\dotfill\fi\hbox to\@pnumwidth{\@tocpagenum{#7}}\par}
{}{}
\makeatother    
\makeatletter
\def\l@subsection{\@tocline{2}{0pt}{4pc}{6pc}{}}
\def\l@subsubsection{\@tocline{3}{0pt}{8pc}{8pc}{}}
\makeatother

\makeatletter
\def\l@section{\@tocline{1}{12pt}{0pt}{}{\bfseries}}% <- added
\makeatother
\usepackage{ragged2e}
\justifying
\usepackage[utf8]{inputenc}
\usepackage[T1]{fontenc}
\mathtoolsset{showonlyrefs}
\usepackage{nomencl}
\makenomenclature
\usepackage{amssymb}

\usepackage[symbol]{footmisc}
\usepackage{centernot}

\title[Well-posedness and asymptotic analysis of a nonlinear heat equation]{Well-posedness and the \L{}ojasiewicz-Simon inequality in the asymptotic analysis of a nonlinear heat equation with constraints of finite codimension}

\begin{document}
	\maketitle
	
	\begin{center}
		\author{Ashish Bawalia\footnote[4]{Department of Mathematics, Indian Institute of Technology Roorkee-IIT Roorkee, Haridwar Highway, Roorkee, Uttarakhand 247667, India.}\orcidlink{0009-0002-9141-2766}, Zdzis\l{}aw Brze\'zniak\footnotemark[2]$^\ast$\orcidlink{0000-0001-8731-6523}, Manil T. Mohan\footnotemark[4]\orcidlink{0000-0003-3197-1136} and Piotr Rybka\footnotemark[3]\orcidlink{0000-0002-0694-8201}}
		\footnotetext[2]{Department of Mathematics, University of York, Heslington, YO10 5DD, York, United Kingdom. 
		}
		\footnotetext[3]{Faculty of Mathematics, Informatics and Mechanics, University of Warsaw, ul. Banacha 2, 02-097 Warsaw, Poland.\\
			\textit{e-mail:} Ashish Bawalia: \email{ashish1@ma.iitr.ac.in; ashish1441chd@gmail.com}.\\
			\textit{e-mail:} Zdzis\l{}aw Brze\'zniak: \email{zdzislaw.brzezniak@york.ac.uk}.\\
			\textit{e-mail:} Manil T. Mohan: \email{maniltmohan@ma.iitr.ac.in; maniltmohan@gmail.com}.\\
			\textit{e-mail:} Piotr Rybka: \email{rybka@mimuw.edu.pl}\\
			$\hspace{2mm} ^\ast$Corresponding author.
			\textit{MSC}: 
			%Asymptotic behavior of solutions to PDEs
			35B40,
			%IVP for nonlinear higher order PDEs
			35G25, 
			% Heat equation
			35K05,
			%Heat and other parabolic equation methods for PDEs on manifolds
			58J35
			%Inequalities involving derivatives and differential and integral operators
			26D10
		}

	\end{center}

	\begin{abstract}
		We establish the global well-posedness of the $D(A)-$valued strong solution to a nonlinear heat equation with constraints on a  \textit{Poincar\'e domain} $\bO\subset \R^d$ whose boundary is of class $C^2$. Consider the following nonlinear heat equation
		\begin{align*}
			\frac{\partial u}{\partial t} - \Delta u + |u|^{p-2}u = 0,
		\end{align*}
		projected onto the tangent space $T_u\bM$, where
		$\mathcal{M}:=\left\{u\in  L^2(\bO):\|u\|_{L^2(\bO)}=1\right\}$ is a submanifold of $L^2(\bO)$. The  nonlinearity exponent satisfies $2\le p < \infty$ for $1\leq d\leq 4$ and  $2 \le p \le \frac{2d-4}{d-4}$ for $d \ge 5$. The solution is constrained to lie within $\mathcal{M}$ which encodes the norm-preserving constraint. By modifying the nonlinearity and exploiting the abstract theory for \textit{$m-$accretive }evolution equations, we prove the existence of a global strong solution. 
		Using {resolvent-idea }	and the \textit{Yosida approximation} method, we derive regularity results. In the asymptotic analysis, $\bO$ is restricted to bounded domains with even $p$
		and $1\le d \le 3$. For any initial data in $D(A) \cap \mathcal{M}$, we apply the \textit{\L{}ojasiewicz-Simon gradient inequality} on a Hilbert submanifold [F. Rupp, \textit{J. Funct. Anal.}, 279(8), 2020], to demonstrate that the unique global strong solution converges in $W^{2,q}(\bO) \cap W^{1,q}_0(\bO)$ to a stationary state, where $2 \le q < \frac{2d}{d + 4 - 4\beta}$ and $1 < \beta < \frac{3}{2}$.
		This work proposes an alternative method for establishing the global existence and analyzing  long-term behavior of the unique strong solution to an $L^2-$norm preserving nonlinear heat equation.
	\end{abstract}
	
	\tableofcontents 
	
	\section{Introduction}\label{Sec-Intro}
	
	In this manuscript, we address the well-posedness, regularity properties, and long-time behavior of a nonlinear heat equation featuring polynomial-type damping. It is worth noting that the well-posedness and asymptotic analysis of the unique global $L^p\cap H_0^1-$valued strong solution  have been studied in our recent work \cite{AB+ZB+MTM-25+}. Our focus, however, is to consider the problem described in \cite{AB+ZB+MTM-25+} to  Poincar\'e domains and study stronger solutions with higher regularity, specifically, $D(A)-$valued solutions, see below for the definition of $A$.

	For the domain $\mathcal{O}$, if there exists a positive constant $\lambda_{1}$ such that the following Poincar\'e inequality  is satisfied:
	\begin{align}\label{2.1}
		\lambda_{1}\int_{\mathcal{O}} |\psi(x)|^2 \d x \leq \int_{\mathcal{O}} |\nabla \psi(x)|^2 \d x,  \ \text{ for all } \  \psi \in {H}_0^1 (\mathcal{O}),
	\end{align}
	then, we call it as a \emph{Poincar\'e domain} and if  $\mathcal{O}$ is bounded in some direction, then the Poincar\'e inequality \eqref{2.1} holds. For example, one can consider  $\mathcal{O}=\R^{d-1}\times(-L,L),$ $L>0$. On bounded domains, $\lambda_1$ coincides with the first eigenvalue of the Dirichlet Laplacian. 
	Let $\bO \subset \R^d$, for any dimension $d \ge 1$, be a Poincar\'e domain with $C^2-$boundary $\partial\bO$. Given the parameter
	\begin{align}\label{Embed-D(A)-L^p}
		p \in \left\{
		\begin{aligned}
			[2,\infty), & \text{ when }\  1 \le d\le 4,\\ 
			\Big[2, \frac{2d}{d-4}\Big], & \text{ when }\  d \ge 5,
		\end{aligned}
		\right.
	\end{align} 
	we consider the following Cauchy problem:
	\begin{align}\label{eqn-main-heat}
		\left\{\begin{aligned}
			\frac{\partial u(t)}{\partial t} & = \Delta u(t) -|u(t)|^{p-2}u(t)+ \big( \norm{\nabla u(t)}_{L^2(\bO)}^2 + \norm{u(t)}_{L^p(\bO)}^p \big) u(t),\ \ t>0,\\
			u(t)|_{\partial\bO} & = 0,\\
			u(0) & = u_0,
		\end{aligned}\right.
	\end{align}
	such that the constraint $u(t) \in \bM$, for all $t\ge 0$ is satisfied, where $u:[0,\infty) \times\bO \to \R$ and 
	$$\bM:=\left\{v\in L^2(\bO):\|v\|_{L^2(\bO)} = 1 \right\}.$$
	For the precise formulation of the problem \eqref{eqn-main-heat}, we refer the reader to \cite[Introduction]{AB+ZB+MTM-25+}. The condition given in \eqref{Embed-D(A)-L^p} ensures that $D(A)\hookrightarrow L^p(\bO)$, where $A=-\Delta$ is the Dirichlet Laplacian with
	\begin{equation}\label{Def-D(A)}
		D(A):= H^2(\bO)\cap H_0^1(\bO).
	\end{equation}
	Observe that, when $u_0 \in L^p(\bO)\cap H_0^1(\bO)\cap \bM$ for $p \in [2,\infty)$, the uniqueness and the invariance in manifold $\bM$ of the global $L^p(\bO)\cap H_0^1(\bO)-$valued strong solution of the above-mentioned problem \eqref{eqn-main-heat} follows from our recent work \cite[Theorem 1.1]{AB+ZB+MTM-25+}. 
	In this work, for $u_0 \in D(A)\cap \bM$, we focus on the existence and regularity results of the global $D(A)-$valued unique strong solutions of the constrained problem \eqref{eqn-main-heat} on a \textit{Poincar\'e domain}. Additionaly, when $\bO\subset\mathbb{R}^d$ $1\leq d\leq 3$ is a bounded domain, we investigate the asymptotic behaviour of the global $D(A)-$valued unique strong solutions.  The dimension restriction is due to the Sobolev embedding $D(A)\embed C(\bO)$ and the boudnedness of the domain is needed for the compact embedding of $W^{2\beta, 2}(\bO) \embed W^{2,q}(\bO)\ \text{ for  }\ q < \frac{2d}{d+4-4\beta} \ \text{ and }\ \beta\in(1, \frac{3}{2})$. 
	In particular, this work builds upon and introduces an alternative approach to study the global existence and long-term dynamics of the unique strong solution to a nonlinear heat equation that preserves the $L^2-$norm, as investigated in \cite{PA+PC+BS-24, AB+ZB+MTM-25+}.
	
	\subsection{Previous works}
	
	\subsubsection{Constrained equations}
	Initially, Rybka \cite{PR-06} and later Caffarelli and Lin \cite{LC+FL-09} investigated the heat equation in the space $L^2(\Omega)$, where $\Omega$ is a bounded domain in $\R^2$, constrained to evolve on a manifold $M$, defined as
	\begin{align*}
		M = \left\{\varphi \in L^2(\Omega) \cap C(\Omega) : \int_\Omega \varphi^i(x) dx = C_i,\ 1\le i \le K\right\}.
	\end{align*} 
	For sufficiently regular initial condition, Rybka demonstrated that the nonlinear heat equation
	$$
	\frac{\partial u}{\partial t} = \Delta u - \sum_{i=1}^K \kappa_i(u) u^{i-1},
	$$
	subject to Neumann boundary conditions, has a global unique solution. The coefficients $\kappa_i(u)$ are chosen to enforce orthogonality of the time derivative $\frac{\partial u}{\partial t}$ to the span of the set $\{u^{i-1}\}$.
	Similarly, Caffarelli and Lin \cite{LC+FL-09} developed a theory ensuring the global existence and uniqueness of solutions that conserve energy in the context of the classical heat equation. Their analytical framework was broadened to encompass a wider class of singularly perturbed and non-local parabolic equations. The study demonstrated that solutions to these perturbed models exhibit strong convergence toward weak solutions of a constrained, non-local heat flow, potentially within a singular target space.
	In a different setting, Ma and Cheng \cite{LM+LC-09} examined two variants of non-local heat flows that maintain the $L^2-$norm on compact Riemannian manifolds. Their analysis established the global existence and stability of solutions, and  analyzed the long-time behavior of solutions, in addition to deriving gradient bounds for positive solutions. Later, in \cite{LM+LC-13}, they established the global existence of positive solutions to a porous-medium type non-local heat flow on compact Riemannian manifolds with strictly positive initial data. By employing Sobolev embeddings and Moser iteration techniques, they showed that the limit of the flow approaches a solution to the Laplace eigenvalue problem.
	Brze\'zniak and Hussain \cite{ZB+JH-24} investigated a nonlinear heat equation of gradient type, focusing on the existence and $L^2-$sphere invariance supporting unique global strong solutions. Their approach utilized semigroup methods and the application of fixed-point arguments. They investigated an evolution equation derived from the Laplace operator, projected onto the tangent space of the $L^2$-unit sphere, with a polynomial nonlinearity of degree $2p-1$, subject to Dirichlet boundary conditions.
	In a related study, Hussain \cite{JH-23} studied strong solutions of a constrained heat equation with values in a Hilbertian manifold and demonstrated their existence and uniqueness using the standard Faedo-Galerkin approximation along with compactness arguments.
	
	Antonelli et al. \cite{PA+PC+BS-24} analyzed a nonlinear heat equation that conserves the $L^2-$norm of the solution  and established the well-posedness both locally and globally in bounded domains and the full space $\R^d,$ with  $2\leq p<\infty$ for $d=1,2$ and $2\leq p\leq\frac{2d}{d-2}$ for $d\geq 3$, using semigroup techniques. Notably, in the case of an open ball, they showed that when starting with strictly positive initial conditions, the unique strong solution evolves over time toward the positive ground state.
	Recently, Shakarov \cite{BS-25} investigated a similar heat equation where solutions are constrained to evolve on an $L^2-$sphere via a nonlocal term. The existence and uniqueness of weak solutions, both local and global, were established for bounded domains with $C^2$ boundaries, as well as for the whole space $\R^d$. The methodology they employed consists the Schauder fixed point Theorem in bounded settings and the contraction mapping principle in unbounded ones. 
	In our recent work \cite{AB+ZB+MTM-25+}, we studied the global well-posedness of $L^p\cap H_0^1-$valued, for all $2\le p < \infty$, strong solutions to \eqref{eqn-main-heat} on any bounded smooth domain in arbitrary dimensions $d\ge 1$. In addition, our asymptotic analysis presented in \cite{AB+ZB+MTM-25+} generalizes the convergence result of Antonelli et al. \cite{PA+PC+BS-24}, from the specific setting of a ball to general bounded smooth domains. For numerical implementations of the approach developed by Caffarelli and Lin \cite{LC+FL-09}, and for practical applications across disciplines such as population dynamics, ecology, and material science, one may refer to \cite{QD+FL-09} and the references therein.
	
	In contrast, Brzeźniak et al. \cite{ZB+GD+MM-18}, by using fixed point techniques, established the well-posedness of global solutions that conserve energy for the incompressible Navier-Stokes equations along with a constrained forcing on $\mathbb{T}^2$ and $\R^2$. The authors also showed the convergence to the Euler equations as viscosity vanishes, assuming bounded initial vorticity. The first extension to stochastic case appeared in \cite{ZB+JH-20}, where Brze\'zniak and Hussain established unique mild solutions to a stochastic heat equation with constraint driven by a Stratonovich forcing in two-dimensional bounded domains. Later, in \cite{ZB+GD-21}, Brze\'zniak and Dhariwal proved the existence of martingale and strong solutions to two-dimensional stochastic Navier-Stokes equations with multiplicative noise. Brze\'zniak and Cerrai \cite{ZB+SC-25} analyzed stochastic damped wave equations constrained to the unit sphere in Hilbert spaces, showing the well-posedness and asymptotic convergence to a constrained parabolic equation in the vanishing mass limit. Most recently, Cerrai and Xie \cite{SC+MX-25} studied the small-mass limit (Smoluchowski-Kramers approximation) for stochastic damped wave equations constrained to the unit sphere in $L^2(0, L)$.

	\subsubsection{The \L{}ojasiewicz inequality} 
	The \L{}ojasiewicz inequality plays a key role in real algebraic and differential geometry, providing a detailed description of how an analytic function behaves in the vicinity of a critical point.
	Firstly, \L{}ojasiewicz established the inequality \eqref{eqn-Lojasiewicz-original} given below in his seminal work on semianalytic and subanalytic sets on standard Euclidean space of dimension $d$ \cite[Theorem 4]{SL-63}. In particular, the precise statement of the \L{}ojasiewicz inequality is as follows:
	\begin{theorem}[{\L{}ojasiewicz inequality, \cite[Theorem 4]{SL-63}}]
		Let $D\subset \R^d$ be open set. If $f: D \to \R$ is an analytic function and $\a\in D$ is a critical point of $f$, i.e., $\nabla f(\a) =0$, then there exist $C,\sigma >0$ and $\theta\in (0,\frac{1}{2}]$ such that
		\begin{align}\label{eqn-Lojasiewicz-original}
			\abs{f(u) - f(\a)}^{1-\theta} \leq C \|\nabla f(u)\|,\ \text{ for all }\ \|u - \a\| \leq \sigma,
		\end{align}
		where $\|\cdot\|$ denotes the Euclidean norm in $\mathbb{R}^d$.  
	\end{theorem}
	Subsequently, Simon \cite[Theorem 3]{LS-83} generalized this inequality to apply to specific energy functionals defined on infinite-dimensional Hilbert spaces setting by using {Lyapunov-Schmidt reduction}. In honor of these significant contributions, the inequality is now commonly referred to as the {\L{}ojasiewicz-Simon gradient inequality}. On other hand, {Kurdyka} \cite{KK-98} generalized the inequality to a broader class of function spaces. 
	In \cite{PF+MM-20}, under weaker assumptions, authors established several abstract versions of the  {\L{}ojasiewicz-Simon} gradient inequality for analytic functions on Banach spaces. They also determined the optimal exponent of the {\L{}ojasiewicz-Simon} gradient inequality when the function is {Morse-Bott}. 
	
	Over the past two decades, the \L{}ojasiewicz-Simon inequality has been widely used to analyze the long-time behavior of gradient flows. Roughly speaking, when the energy functional associated with an evolution equation satisfies this inequality near a stationary point $a = \lim_{m\to\infty} u(t_m)$, for some sequence $t_m\to\infty$, and the gradient flows $\{u(t):t\geq 0\}$ is a precompact solution of the gradient system
	\begin{align*}
		\left\{
		\begin{aligned}
			\frac{\partial u(t)}{\partial t} & = - \nabla f(u(t)),\ \text{ for all }\ t>0,\\
			u(0) & = u_0,
		\end{aligned}
		\right.
	\end{align*}
	the inequality provides crucial information about convergence and stability near $a$, see for example \cite{MAJ-98, PR+KHH-98, AH+MAJ-99, RC-03, RC+EF-05, RC+AF-06, HW+MG+SZ-07, AH+MAJ-07, MG+HW+SZ-08, RC+AH+MAJ-09, AH+MAJ-11, SI-21}. 
	Specifically, these results are often employed to analyze the behavior of solutions near equilibrium points in an appropriate norm. A primary application is the study of the long-term dynamics of trajectories, where the behavior of the energy functional near equilibria is crucial. For example, see \cite[for sub-gradient systems]{RC+SM-18}, \cite[fractional Cahn-Hilliard
	system]{GA+GS+AS-19}, \cite[Kurdyka-\L{}ojasiewicz-Simon inequality in metric spaces]{DH+JM-19}, \cite[coupled Yang-Mills energy functionals]{PMNF+MM-20}, among others. Additionally, numerical implementations of this framework are discussed in \cite{PAA+RM+BA-05,HA+JB+PR+AS-10}.
	
	A second line of research has focused on applying inequalities, particularly the \L{}ojasiewicz inequality, to analyze the structural properties of solutions to various evolutionary partial differential equations exhibiting a gradient flow structure. One of the earliest results in the context of constrained gradient flow problems was obtained by {Rybka} \cite{PR-06}, who used the \L{}ojasiewicz inequality to address the convergence of solutions to a heat equation with analytic nonlinearity in the space $W^{2,p}(D)$, where $D$ is a bounded domain in $\R^2$. In \cite{MC+LS+BV-20}, {Colombo} et al. studied parabolic variational inequalities arising from gradient flow problems. They developed a constrained \L{}ojasiewicz-Simon inequality for critical points within a given convex subset of $L^2(\bO)$ associated with an analytic integral functional. This framework was then applied to the parabolic obstacle and thin-obstacle problems. 
	
	Recently, {Rupp} \cite{FR-20} introduced a refined version of the \L{}ojasiewicz-Simon gradient inequality for gradient flow problems in Banach spaces with constraints; see \cite[Theorem 1.4]{FR-20} for details. This advancement has opened new avenues for studying the convergence of solutions to evolution equations with gradient-like structures subject to constraints. Applications of this approach include length-preserving elastic flows \cite{FR+AS-24}, evolving heterogeneous elastic wires \cite{AD+GJ+LL+FR-24}, elastic flows of curves \cite{MP-22}, heterogeneous elastic wires \cite{AD+LL+FR-24}, and Navier-Stokes-Cahn-Hilliard systems \cite{JH+HW-24}, among others. 
	
	\subsection{Highlights and novelties of this work}
	\subsubsection{Highlights}
	This note has two main objectives. First, we establish the global existence and regularity results of the unique $D(A)-$valued strong solution to the system \eqref{eqn-main-heat} on a Poincar\'e domain $\bO\subset\R^d$ with $C^2-$boundary. We adopt a classical strategy of cut-off, inspired by \cite[p. 291]{VB+SSS-01} (see also \cite[p. 1080]{SG+KK+MTM-24}, \cite[Section 3]{ZB+TZ-23}), together with the famous Yosida approximation. 
	Specifically, under the assumptions on $p$ and $d$ as in \eqref{Embed-D(A)-L^p}, we prove the following:
	\begin{itemize}
		\item[$(i)$] For fixed $K\in\N$, the nonlinear cut-off operator $\gn^K: D(A) \to L^2(\bO)$,	see \eqref{Def-gn}, satisfies the demicontinuity property as well as some monotone-type bounds, see Lemmas \ref{Lem-gn-mono} and \ref{gn-demi};
		\item[$(ii)$] Using the properties of the map $\gn^K$, along with the hemicontinuity and coercivity properties, we establish the $m-$accretivity of $\Gn^K + \Gamma I$, for some $\Gamma>0$ depending on $K$, see Theorem \ref{Thm-m-accretive}. Consequently, applying the abstract theory of $m-$accretive operators for evolution equations (\cite[Theorem 1.4--1.6]{VB-93}), we obtain the existence result for the modified problem \eqref{eqn-quant-heat}, see Proposition \ref{Prop-quant};
		\item[$(iii)$] Deriving uniform energy estimates and choosing $K > \frac{1}{2}\norm{\nabla u_0}_{L^2(\bO)}^2 + \frac{1}{p}\norm{u_0}_{L^p(\bO)}^p$, we first establish the global existence of the unique strong solution to the system  \eqref{eqn-main-heat} in $W^{1,\infty}([0,T];L^2(\bO))\cap L^\infty(0,T; D(A))$, see Subsection \ref{Subsec-Thm-I};
		\item[$(iv)$] By introducing the Yosida-approximated solution \eqref{Def-u_mu} and using the resolvent-identity strategy, we demonstrate a convergence result showing that the Yosida-approximated solution converges in $D(A)-$norm to the unique strong solution of \eqref{eqn-main-heat}, see Proposition \ref{Lem-mu->u-cgs}. Next, we derive two regularity results for $u_0\in D(A)$; see Propositions \ref{Prop-A^{3/2}} and \ref{Lem-u_t-is-C}, which completes the proof of Theorem \ref{Thm-Main-Exis}.
	\end{itemize}

	In the second part of this work, we restrict ourselves to bounded domains $\bO\subset \R^d$  and establish two results on uniform-in-time bounds to analyze the long-time behavior of the unique strong solution to \eqref{eqn-main-heat}, following the approach developed in \cite{MAJ-98, PR-06, FR-20}. To begin, let us choose and fix
	\begin{align}\label{Embed-H_0^1-L^p}
		p \in \left\{
		\begin{aligned}
			[2,\infty), & \text{ when }\  1\le d \le 4,\\ 
			\bigg[2, \frac{2d-6}{d-4}\bigg), & \text{ when }\  d \ge 5.
		\end{aligned}
		\right.
	\end{align}
	Then, we establish the following:
	\begin{itemize}
		\item[$(i)$] If $p< \frac{2d-4}{d-4}$ for $d\ge5$  in \eqref{Embed-D(A)-L^p}, and $u_0 \in D(A^\alpha)$ with $\alpha \in (\frac{1}{2},1)$, then the trajectory $\{u(t): t\geq 0\}$ is bounded in $D(A^\alpha)$;
		\item[$(ii$)] For $u_0 \in D(A),$ the orbit $\{u(t): t\geq 1\}$ is bounded in $D(A^\beta)$, for every $\beta \in (1,\frac{3}{2})$, which further implies that the omega-limit set $\omega(u)$ is non-empty, compact and connected in $W^{2,q}(\bO)\cap W^{1,q}_0(\bO)$, where $q\in \big[2, \frac{2d}{d+4-4\beta}\big)$, see Subsection \ref{Subsec-Other-reg};
		\item[$(iii)$] When $p \in \{2, 4, \dots\}$ in addition to $2\le p < \infty$ for $1\le d\le3$, the energy and constrained functionals, defined in \eqref{def-energy} and \eqref{def-constraint}, respectively, are analytic from $D(A)$ to $\R$ with their first order Fr\'echet derivatives take values in $L^2(\bO)$. We also show that the second order Fr\'echet derivatives of the energy and  constrained functionals,   are of Fredholm index zero and compact, respectively, for details see Subsection \ref{Subsec-LSI};
		\item[$(iv)$] Using the above properties of energy and constrained maps, we verify the \L{}ojasiewicz-Simon inequality on Hilbert spaces with constraints, in our settings and further prove that, if $$u_0 \in D(A)\cap\bM,\ \text{ and }\ q \in \left[2, \frac{2d}{d+4-4\beta}\right),$$
		then, the problem \eqref{eqn-main-heat} admits a unique strong solution that approaches a steady state $u^\infty$ in the $W^{2,q}(\bO)\cap W^{1,q}_0(\bO)-$norm as $t\to\infty$, where $u^\infty$ solves
		\begin{align}\label{eqn-stationary}
			\Delta v -|v|^{p-2}v + \big( \norm{\nabla v}_{L^2(\bO)}^2 + \norm{v}_{L^p(\bO)}^p \big) v = 0.
		\end{align}
	\end{itemize}
	
	\begin{remark}
		In the asymptotic analysis, it is important to emphasize that, due to the Dirichlet boundary condition in problem \eqref{eqn-main-heat}, the results of Antonelli et al. \cite{PA+PC+BS-24}, as well as their generalization to bounded domains in \cite{AB+ZB+MTM-25+},  can only be recovered by applying the refined \L{}ojasiewicz-Simon gradient inequality developed by Rupp \cite[Theorem 1.4]{FR-20}, for integer $p\in[2,\infty)$ with $1\le d \le 3$. This restriction arises because the nonlinearity $u^{p-1}$ is not analytic at zero unless $p\in\N$.
	\end{remark}	
	
	\subsubsection{Novelties}
	In this article, we highlight the following novel contributions:
	\begin{itemize}
		\item[$(i)$] To best of our knowledge, for any spatial dimension $d \ge 1$ and any exponent $p$ satisfying \eqref{Embed-D(A)-L^p}, this is the first work to establish the existence of $D(A)-$valued global strong solutions to a nonlinear heat equation defined on arbitrary Poincar\'e domains  subject to $L^2-$norm constraints  through \textit{$m-$accretive} techniques. Our work highlights the potential of the nonlinearity cut-off trick to broaden the applicability of the technique to other constrained problems.
		\item[$(ii)$] This work extends the existing well-posedness theory and establishes new regularity results, {by exploiting both resolvent and spectral analysis within the framework of the Yosida approximation,} for the nonlinear heat equation with constraints on arbitrary Poincar\'e domain with $C^2-$boundary; see, for instance, \cite[Theorem 1.3]{PA+PC+BS-24}, \cite[Theorem 1.10]{AB+ZB+MTM-25+}, \cite[Theorem 1.5]{ZB+JH-24}, and \cite[Theorem 2.2]{JH-23}.
		\item[$(iii)$] In the case of a damped heat equation defined on bounded domains with constraints, for any initial data in $D(A)\cap\bM$, we affirmatively analyze the long-time behaviour by incorporating a fundamentally different approach (cf. \cite[Theorem 1.7]{PA+PC+BS-24} and \cite[Theorem 1.10]{AB+ZB+MTM-25+}), i.e., by utilizing the refined \L{}ojasiewicz-Simon gradient inequality formulated for Hilbert spaces.
		\item [$(iv)$] The \L{}ojasiewicz-Simon gradient inequality allows us to establish local asymptotic stability around every stationary solution (without requiring any positivity assumption), while \cite{PA+PC+BS-24,AB+ZB+MTM-25+} focus on the asymptotic behavior towards a unique positive stationary solution.
		\item[$(v)$] It is worth emphasizing that {Rybka} \cite{PR-06} analyzed the asymptotic behavior of the heat equation with polynomial damping only in two spatial dimensions. In contrast, our results hold for any $1\le d \le 3$ and any even integer $2\leq p<\infty$.
	\end{itemize}

	\subsection{Statement of the main results}
	Consider $\bO \subset \R^d$ to be a Poincar\'e domain with dimension $d\ge 1$ and $C^2-$boundary denoted by $\partial\bO$. To begin, let us clarify the meaning of a strong solution in the context of problem \eqref{eqn-main-heat}.
	\begin{definition}\label{Def-strong-soln}
		Let us choose and fix $p$ as in \eqref{Embed-D(A)-L^p}, $T\in(0,\infty)$ and $u_0 \in D(A)\cap\bM$. A function
		\begin{align*}
			u \in W^{1,\infty}([0,T]; L^2(\bO))\cap C([0,T]; D(A)\cap\bM)\cap L^2(0,T;D(A^{\frac{3}{2}})),
		\end{align*}
		is called a \emph{strong solution} of the system \eqref{eqn-main-heat},  if
		the following two conditions are satisfied:
		\begin{itemize}
			\item[$(i)$] The equation
			\begin{align}\label{eqn-102}
				\displaystyle \frac{\partial u(t)}{\partial t} -\Delta u(t) + |u(t)|^{p-2}u(t) -  \big(\norm{\nabla u(t)}_{L^2(\bO)}^2  + \norm{u(t)}_{L^p(\bO)}^p \big) u(t) = 0,
			\end{align}
			is satisfied in $L^2(0,T;L^2(\bO))$, i.e., for all $\psi\in L^2(0,T; L^2(\bO))$
			\begin{equation}\label{eqn-test}
				\int_0^T\bigg(\frac{\partial u(t)}{\partial t} -\Delta u(t) + |u(t)|^{p-2}u(t) -  \big(\norm{\nabla u(t)}_{L^2(\bO)}^2  + \norm{u(t)}_{L^p(\bO)}^p \big) u(t)),\psi(t)\bigg)dt=0.
			\end{equation}
			\item[$(ii)$] The initial data is satisfied
			\begin{align}
				u(0) = u_0 \ \text{ in }\ L^2(\bO).
			\end{align}
		\end{itemize}
	\end{definition}
	\noindent
	Let us consider the energy functional
	\begin{align*}
		\En: D(A) \ni u \mapsto \En(u) \in \R
	\end{align*}
	defined by
	\begin{align}\label{Def-En}
		\En(u) = \frac{1}{2}\int_\bO \abs{\nabla u(x)}^2 dx + \frac{1}{p} \int_\bO \abs{u(x)}^p dx.
	\end{align}
	\noindent
	First, we demonstrate one of the main results on the global existence of a strong solution of the problem \eqref{eqn-main-heat}, as described in Definition \ref{Def-strong-soln}. Let us choose and fix $p$ satisfying  \eqref{Embed-D(A)-L^p}.
	
	\begin{theorem}\label{Thm-Main-Exis}
		Let us suppose $0<T<\infty$ and $u_0\in D(A)\cap \bM$. Then, there exists a unique strong solution 
		\begin{align}\label{eqn-D(A)}
			u\in W^{1,\infty}([0,T];L^2(\bO))\cap C([0,T]; D(A)\cap \bM)\cap L^2(0,T;D(A^{\frac{3}{2}})),
		\end{align}
		with $\frac{\partial u}{\partial t} \in L^2(0,T;H_0^1(\bO))$ solves the following Cauchy problem \eqref{eqn-main-heat} in $(0,T)\times \bO$ and it satisfies the energy equality
		\begin{align}\label{eqn-dec-ener}
			\En(u(t))+\int_0^t\bigg\|\frac{\partial u(s)}{\partial s}\bigg\|_{L^2(\bO)}^2ds= \En(u_0), \ \text{ for all }\ t\ge0,
		\end{align}
		where $\En$ is defined in \eqref{Def-En}.
	\end{theorem}
	
	\begin{remark}
		Let us highlight one of the key novelties of Theorem \ref{Thm-Main-Exis}. We establish the well-posedness of the nonlinear heat equation with constraints on {a Poincar\'e domain} with a $C^2-$boundary. In contrast, the results in \cite{PA+PC+BS-24} and our earlier work \cite{AB+ZB+MTM-25+} address the well-posedness only for bounded domains and the whole space, and bounded domains, respectively.
	\end{remark}
	
	For the detailed proof of Theorem \ref{Thm-Main-Exis}, we refer the reader to Section \ref{Sec-Thm-Exi}.
	
	\noindent
	Next, we restrict {our attention} to a bounded domain $\bO \subset \R^d$ with $1 \le d \le 3$, fix an exponent $2 \le p < \infty$ where $p \in \{2,4,\dots\}$, and select an auxiliary parameter $\beta \in \bigl(1, \tfrac{3}{2}\bigr)$ such that
	\begin{align*}
		q\in \left[2, \frac{2d}{d+4-4\beta}\right).
	\end{align*}
	
	Then, the following theorem is the main result on the asymptotic analysis of this work.
	
	\begin{theorem}\label{Thm-u(t)-cgs-u^infty-copy}
		Let $u_0 \in D(A)\cap\bM$.
		Assume that $u$ is the unique global strong solution to problem \eqref{eqn-main-heat}, whose existence is guaranteed by Theorem \ref{Thm-Main-Exis}. Then, there exists $u^\infty\in W^{2,q}(\bO)\cap W^{1,q}_0(\bO)\cap\bM,$ which is a stationary solution to \eqref{eqn-stationary} such that 
		\begin{align}\label{eq-con-to u^infty}
			\norm{u(t)- u^\infty}_{W^{2,q}(\bO)} \to 0 \ \text{ as }\ t \to \infty.
		\end{align}
		In particular, it holds that
		\begin{align*}
			\norm{u(t)- u^\infty}_{D(A)} \to 0 \ \text{ as }\ t \to \infty.
		\end{align*}
	\end{theorem}
	
	\begin{corollary}
		If $u_0 \in D(A)\cap\bM$ is a stationary solution to \eqref{eqn-stationary}, then $u_0\in W^{2,q}(\bO)\cap W^{1,q}_0(\bO)\cap\bM$.
	\end{corollary}
	\begin{proof}
		Let us choose and fix  $u_0 \in D(A)\cap\bM$ which is a stationary solution to \eqref{eqn-stationary}. Define $u(t)=u_0$ for $t\geq 0$. Then, $u$  is the unique global strong solution to problem \eqref{eqn-main-heat}.
		By Theorem \ref{Thm-u(t)-cgs-u^infty-copy}, there exists $u^\infty\in W^{2,q}(\bO)\cap W^{1,q}_0(\bO)\cap\bM$ such that \eqref{eq-con-to u^infty} holds. Hence $u_0=u^\infty$ and so the result follows.
	\end{proof}
	
	We demonstrate the proof of Theorem \ref{Thm-u(t)-cgs-u^infty-copy} in Subsection \ref{Subsec-Asymp}.
	
	\begin{remark}
		\noindent
		\begin{itemize}
			\item[$(i)$]  Let us first emphasize that Rupp \cite{FR-20} developed a refined \L{}ojasiewicz-Simon gradient inequality on Banach space settings with constaints and applied it to the Allen-Cahn equation. In our work, we establish a version of this refined inequality (cf. \cite[Theorem 1.4]{FR-20}) adapted to our framework and employ it to analyze the asymptotic behavior of the strong solution described in Theorem \ref{Thm-u(t)-cgs-u^infty-copy}.
			
			\item[$(ii)$] It is worth noting that Antonelli et al. \cite{PA+PC+BS-24} established Theorem \ref{Thm-u(t)-cgs-u^infty-copy} for non-negative initial data in $H_0^1(\bO)$. Their result holds on a ball for $2\le p <\infty$ when $d=1,2,$ and $2\leq p\leq \frac{2d}{d-2}$ when $d \ge 3$. In contrast, in our recent work \cite{AB+ZB+MTM-25+}, we lifted the restriction on $p$ for all dimensions $d \ge 1$, and considered the equation on arbitrary bounded domains with initial data in $L^p(\bO)\cap H_0^1(\bO)$.
			Moreover, our current result complements both the above mentioned results for $2\leq p<\infty$,
			with the additional assumptions that $p$ is even and the initial data belongs to $D(A)$.
		\end{itemize}
	\end{remark}
	
	\subsection{Plan of the manuscript}
	The organization of this manuscript is as follows.
	
	Section \ref{Sec-Preliminaries} begins with essential preliminaries, including the functional framework and auxiliary concepts such as the definitions of monotonicity and $m-$accretivity of operators, the notion of the $\omega-$limit set in complete metric spaces along with related results. Then, a linear operator, a nonlinear operator and their monotonicity properties are discussed. Further, we show that the first equation in the problem \eqref{eqn-main-heat} is of gradient structure. 
	
	Section \ref{Sec-Thm-Exi} is divided into two parts. First, we define a nonlinear operator $\gn^K$ (see \eqref{Def-gn}) and show that it satisfies a certain monotone-type bound and it is demicontinuous. Then, using these facts, we establish that the main cut-off operator $\Gn^K + \Gamma I$ (see \eqref{Def-Gn-K}) is \textit{$m-$accretive}, for sufficiently large $\Gamma>0$ depending on $K$. Moreover, by utilizing the abstract theory of \textit{$m-$accretive} operators for evolution equations, we obtain the existence results for the modified system, see Proposition \ref{Prop-quant}. Finally, by calculating the energy estimates and choosing indexed parameter $K$ sufficiently large, we prove the existence result pertaining to the original problem in the class $W^{1,\infty}([0,T];L^2(\bO))\cap L^\infty(0,T; D(A))$, which completes the first part of the proof of Theorem \ref{Thm-Main-Exis}, see Subsection \ref{Subsec-Thm-I}.
	Subsection \ref{Subsec-Thm-II} opens with a regularity result for the solution established in the previous section. Then, by virtue of a well-known method called \textit{Yosida approximation} technique and \textit{resolvent identity}, 
	%spectral measures, 
	when the initial condition is in $D(A)$, we prove two regularity results. The first result (see Proposition \ref{Prop-A^{3/2}}) provides the required regularity (see Remark \ref{Rmk-Regularity}), which completes the proof of Theorem \ref{Thm-Main-Exis}. The second result is instrumental in establishing uniform-in-time bounds in the following section.
	
	Next, in Section \ref{Sec-Asy-anal}, we show that if $u_0 \in D(A^\alpha)$ with $\alpha \in \big(\frac{1}{2},1\big)$, then the trajectory $\{u(t): t \ge 0\}$ remains bounded in $D(A^\alpha)$; similarly, if $u_0 \in D(A)$, then $\{u(t): t \ge 1\}$ is bounded in $D(A^\beta)$ for any $\beta \in \big(1,\tfrac{3}{2}\big)$.
	Then, by utilizing the Sobolev embedding Theorem, we produce that the orbit $\{u(t): t\geq1\}$ is precompact and the omega-limit set $\omega(u)$ is compact in $W^{2,q}(\bO)\cap W^{1,q}_0(\bO)$. We also present foundational aspects of the \L{}ojasiewicz-Simon inequality on Hilbert spaces, including analytic functions, the Fredholm index, and other relevant topics and show that the energy and constrained functionals are analytic. Moreover, we demonstrate some related lemmas that are used to establish a \L{}ojasiewicz-Simon inequality in our setting. 
	Further, we establish that each member of the $\omega(u)$ is a critical point of the restricted energy functional $\En|_\bM$ if and only if it solves the stationary problem \eqref{eqn-stationary}. By utilizing the \L{}ojasiewicz-Simon inequality (see \eqref{eqn-Lojasiewicz-Rupp}) along with the previously established results, we finally prove the strong convergence of the unique global strong solution in $W^{2,q}(\bO)\cap W^{1,q}_0(\bO)-$norm to a stationary state, as $t$ tends to $\infty$, see Theorem \ref{Thm-u(t)-cgs-u^infty}. 
	
	The manuscript concludes in Appendix \ref{Sec-Appendix} with the statement of {a Representation Theorem, a particular case of Lions-Magenes Lemma (with the proof)}, the Spectral Theorem for self-adjoint operators, {alternative proofs of the regularity results, demonstrated in Subsection \ref{Subsec-Yosida}, by spectral measures} and a proof of an elementary inequality used in one of the main results, Theorem \ref{Thm-u(t)-cgs-u^infty}.
	
	\section{Preliminaries}\label{Sec-Preliminaries}
	We outline the foundational material used in this work in the present section. This includes the functional settings, along with some elementary definitions and results on  linear and nonlinear operators, and the gradient flow structure of the first equation in \eqref{eqn-main-heat}. We fix $\bO\subset \R^d$, to be any Poincar\'e domain with a boundary of class $C^2$, throughout this section.
	
	\subsection{Functional setting}
	For any $1\leq p < \infty$, the space $L^p(\bO)$ consists of equivalence classes $[h]$ of   Lebesgue measurable functions $h : \bO \to\mathbb{R}$ that satisfy $\int_{\bO}|h(x)|^pdx<\infty.$ The $L^p-$norm of $h \in L^p(\bO)$ is defined by $\|h\|_{L^p(\bO)}:=\left(\int_{\bO}|h(x)|^pdx\right)^{1/p}$.  For $p=2$, the space $L^2(\bO)$ forms a Hilbert space, with the inner product denoted by $(\cdot,\cdot)$. Additionally, let $H_0^1(\bO)$, also written as $W_0^{1,2}(\bO)$, represents the Sobolev space consisting of equivalence classes of Lebesgue measurable functions $h \in L^2(\bO)$ whose weak partial derivatives $\frac{\partial h}{\partial x_i}$ belongs $L^2(\bO)$, and which vanish on the boundary of $\bO$ in the sense of trace. The norm on $H_0^1(\bO)$ is defined, via the Poincar\'e inequality \eqref{2.1}, by $$\norm{h}_{H_0^1(\bO)} := \left(\int_{\bO}|\nabla h(x)|^2dx\right)^{1/2}.$$ 
	We next introduce the dual space $H^{-1}(\bO):=(H_0^1(\bO))^{\ast}$, which consists of continuous linear functionals on $H_0^1(\bO)$, equipped with the norm
	\[\norm{l}_{H^{-1}(\bO)} := \sup\big\{ {\langle l, h\rangle}:  h \in H_0^1(\bO),\ {\norm{h}_{H_0^1(\bO)}}\leq 1\big\}.\]
	Additionally, we denote by $W^{2,q}(\bO)$, for $1\leq q\leq \infty$, the Sobolev space of functions with weak derivatives up to order two in $L^q(\bO)$. 
	
	We now introduce the sum and intersection spaces that will be used throughout this study. Observe that both $L^{p^\prime}(\bO)$ and $H^{-1}(\bO)$ are Banach spaces, endowed with the norms $\norm{\cdot}_{L^{p^\prime}(\bO)}$ and $\norm{\cdot}_{H^{-1}(\bO)}$, respectively, where the exponents satisfy $\frac{1}{p} + \frac{1}{p'} =1$. In addition, the space $L^p(\bO) \cap H_0^1(\bO)$  is dense in each of $L^{p^\prime}(\bO)$ and $H^{-1}(\bO)$ with respect to their respective norms.
	
	We now define the sum space
	\begin{align*}
		L^{p^\prime}(\bO) + H^{-1}(\bO) := \{l_1 + l_2 : l_1 \in L^{p^\prime}(\bO), \ l_2 \in H^{-1}(\bO)\}
	\end{align*}
	which forms a Banach space when equipped with the norm
	\begin{align*}
		\norm{u}_{L^{p^\prime}(\bO) + H^{-1}(\bO)}  = \inf\{\norm{l_1}_{L^{p^\prime}(\bO)} + \norm{l_2}_{H^{-1}(\bO)} : u = l_1 + l_2, l_1 \in L^{p^\prime}(\bO), l_2 \in H^{-1}(\bO)\}.
	\end{align*}
	The intersection space $L^p(\bO) \cap H_0^1(\bO)$ forms a Banach space when endowed with the norm
	\[\norm{u}_{L^p(\bO) \cap H_0^1(\bO)} := \max\{\norm{u}_{L^p(\bO)}, \norm{u}_{H_0^1(\bO)}\}.\]
	This norm is equivalent to both $\norm{u}_{L^p(\bO)}+ \norm{u}_{H_0^1(\bO)}$ and the Euclidean-type norm $\big(\norm{u}_{L^p(\bO)}^2+ \norm{u}_{H_0^1(\bO)}^2\big)^{1/2}$. Furthermore, the dual space of $L^{p^\prime}(\bO) + H^{-1}(\bO)$ can be identified as
	\begin{align*}
		(L^{p^\prime}(\bO) + H^{-1}(\bO))^{\ast} \cong L^p(\bO) \cap H_0^1(\bO),
	\end{align*}
	with the natural duality pairing defined by 
	\begin{align*}
		\langle l, h \rangle = \langle l_1, h \rangle + \langle l_2, h \rangle,
	\end{align*}
	for every $l = l_1 +l_2 \in L^{p^\prime}(\bO) + H^{-1}(\bO)$ and $h\in L^p(\bO) \cap H_0^1(\bO)$. Therefore, it holds that \cite[cf. Section 2]{RF+HK+HS-05}
	\begin{align*}
		\norm{l}_{L^{p^\prime}(\bO) + H^{-1}(\bO)} = \sup \big\{ \langle l_1 + l_2, h\rangle :  h \in L^p(\bO) \cap H_0^1(\bO),\ {\norm{h}_{L^p(\bO) \cap H_0^1(\bO)}}  \leq 1 \big\}.
	\end{align*}
	
	\subsection{Auxiliary results}
	Let us first recall some basic definitions and useful results for nonlinear operators from \cite{JA+PPZ-90,VB-93,DH-81, SL-93, FR-20} as follows:
	
	\subsubsection{M-accretive}
	\begin{definition}[{\cite[Definition 1.1, p. 36]{VB-93}}]
		Let $\Xn$ be a reflexive Banach space with the dual $\Xn^\ast$. The map $\Fn:D(\Fn)  \to \Xn^\ast$, with $D(\Fn)\subset 
		\Xn$, is called to be \textit{monotone} (or \textit{accretive}) if
		\begin{align*}
			\langle \Fn x - \Fn y, x-y \rangle \geq 0,\ \text{ for all }\ x,y \in D(\Fn ).
		\end{align*}
		Moreover, $\Fn $ is \textit{maximal monotone} if there is no monotone operator that properly contains it and \textit{$m-$accretive} if the range is the same as the codomain, i.e., $R(I + \Fn ) = \Xn^\ast$.
	\end{definition}
	
	\begin{theorem}[{\cite[A single valued version of Theorem 1.2]{VB-93}}]\label{Thm-Max-Mono}
		Let $\Xn$ and $\Xn^*$ be reflexive and strictly convex Banach spaces. 
		%Let $\Fn \subset \Xn\times \Xn^*$ and let $J: \Xn \to \Xn^*$ be the duality mapping of $\Xn$. 
		Then, a map  $\Fn : \Xn \to \Xn^*$ is maximal monotone if and if, for any $\lambda>0$ (equivalently, for some $\lambda >0$), $R(\Fn  + \lambda I) = \Xn^*$.
	\end{theorem}
	\begin{remark}
		Let us remark that Pazy’s book \cite{AP-83} deals only with linear operators and therefore does not provide the relevant definitions and results. Since we are working with a nonlinear operator, we instead refer to Barbu’s book \cite{VB-93}.
	\end{remark}
	
	\begin{corollary}[{\cite[Corollary 1.3]{VB-93}}]\label{Cor-Mono+Hemi+Coe=Onto}
		A monotone, hemicontinuous, and coercive map  $\Fn $ from a reflexive Banach space $\Xn$ to its dual $\Xn^*$, is surjective.
	\end{corollary}
	
	\begin{proposition}[{\cite[A single valued version of Proposition 3.3]{VB-93}}]\label{Prop-m-acc-Range}
		A map  $\Fn :\Xn\to \Xn^\ast$ is $m-$accretive if and only if $R(\Fn  + \lambda I) = \Xn^\ast$, for all (equivalently, for some) $\lambda>0$.
	\end{proposition}
	
	%	\begin{remark}
		%		Observe that, when $\Xn=\Hn\cong \Hn^*$ is a Hilbert space, then, the duality mapping $J : \Hn \to \Hn^* = \Hn$ will become the identity map, i.e., $I: \Hn\to \Hn$ only.
		%	\end{remark}
	
	\begin{remark}[{\cite[p. 103]{VB-93}}]\label{Rmk-max+mono-m-acc}
		By Theorem \ref{Thm-Max-Mono} and Proposition \ref{Prop-m-acc-Range}, note that, if  $\Xn=\Hn\cong \Hn^*$ is a Hilbert space, then $\Fn $ is $m-$accretive if and only if maximal monotone.
	\end{remark}
	
	\subsubsection{The $\omega$-limit set}
	\begin{definition}[{\cite[Definition 4.3.1]{DH-81}}]
		Suppose $\C$ is a complete metric space. Let $\{S(t): t\geq 0\}$ be a dynamical system on $\C$. A set $K\subset \C$ is \emph{invariant} if, for any $u_0\in K$, there exists a continuous curve $u:[0,\infty) \to K$ with $u(0) = u_0$ and
		\begin{align*}
			S(t)u(\tau) = u(t+ \tau)\ \text{ for }\ t,\tau \geq 0.
		\end{align*}
	\end{definition}
	\begin{definition}[{\cite[Definition 4.3.2]{DH-81}}]\label{Def-omega-limit-set/orbit}
		If $u_0\in \C$ is an initial data, $\gamma(u_0):=\{u(t)= S(t)u_0: t\geq 0\}$ is the \emph{orbit} through $u_0$, then the $\omega-\emph{limit set}$ for $u_0$ or for the orbit  $\gamma(u_0)$ is
		\begin{align*}
			\omega(u) := \{w\in \C: \exists\ t_n \to \infty \text{ such that } u(t_n)=S(t_n)u_0 \to w\}.
		\end{align*}
	\end{definition}
	\begin{lemma}[{\cite[Exercise 1]{DH-81}}]\label{Lem-Ex-1}
		If $u_1 \in \gamma(u_0)$, then $\omega(u) = \omega(u_1)$. Also
		\begin{align*}
			\omega(u) = \bigcap_{\tau\geq 0} \overline{\{S(t)u_0 : t\geq \tau\}} = \bigcap_{\tau\geq 0} \overline{\gamma(S(\tau)u_0)}.
		\end{align*}
	\end{lemma}
	
	\subsection{Linear operator}
	We now turn to the Laplace operator with Dirichlet boundary conditions, when $\bO$ is a Poincar\'e domain. One can use the Poincar\'e  inequality \eqref{2.1} to define an equivalent norm on $H_0^1(\bO)$ as $\|u\|_{H_0^1(\bO)}=\|\nabla u\|_{L^2(\bO)}$. 
	
	Let us define a bilinear form 
	\[a: H_0^1(\bO)\times H_0^1(\bO) \to \R \;\; \mbox{by}\; \; a(u,v) := (\nabla u, \nabla v),\; \; \mbox{for}\;\; u,v \in H_0^1(\bO).\]
	From the definition of $a(\cdot, \cdot)$, it follows that $a(\cdot, \cdot)$ is continuous on $H_0^1(\bO)$; specifically, 
	\[\abs{a(u,v)} \le \norm{u}_{H_0^1(\bO)} \norm{v}_{H_0^1(\bO)},\ \text{ for all } \ u,v \in H_0^1(\bO).\] 
	By the Riesz representation Theorem, there exists a unique linear operator $\A : H_0^1(\bO) \to H^{-1}(\bO)$, such that
	\begin{align*}
		a(u,v) = \left\langle\A u, v\right\rangle,\ \mbox{ for all }\ u,v \in H_0^1(\bO).
	\end{align*}
	Furthermore, the form $a(\cdot, \cdot)$ is coercive on $H_0^1(\bO)$, meaning there exists $ \alpha >0$ (in fact, $\alpha=1$) such that $a(u,u) \geq \alpha \norm{u}^2_{H_0^1(\bO)}$, for all $u\in H_0^1(\bO)$. 
	As a consequence, the Lax-Milgram Theorem guarantees that $\A$ is an isomorphism between $H_0^1(\bO)$ and $H^{-1}(\bO)$. Now, we introduce an unbounded linear operator $A$ on $L^2(\bO)$ defined by
	\begin{align*}
		A u & := \A u=-\Delta u, \ \mbox{ for all } \  u \in D(A):= \{u\in H_0^1(\bO): \A u \in L^2(\bO)\},
	\end{align*}
where $D(A)$ is equipped with the graph norm. For a general domain $\bO$, the definition of $D(A)$ differs from the one appearing in \eqref{Def-D(A)}, but they agree if $\partial\bO$ is of class $C^2$.	Observe that $A$ is a self-adjoint operator.  Since $\bO$ is a Poincar\'e domain, it follows from the Poincar\'e inequality \eqref{2.1} that 
	\begin{align*}
		\|u\|_{L^2(\bO)}^2\leq \frac{1}{\lambda_1}\|\nabla u\|_{L^2}^2=\frac{1}{\lambda_1}(A u,u)\leq \frac{1}{\lambda_1}\|A u\|_{L^2(\bO)}\|u\|_{L^2(\bO)},
	\end{align*}
	so that $$	\|u\|_{L^2(\bO)}\leq\frac{1}{\lambda_1}\|A u\|_{L^2(\bO)}, \ \text{ for all }\ u\in D(A). $$ 
	The elliptic regularity theory \cite[Theorem 15.2]{SA+AD+LN-59} provides 
	\begin{equation}\label{eqn-elliptic-reg}
		\|u\|_{H^2(\bO)}\leq C\left(\|u\|_{L^2(\bO)}+\|A u\|_{L^2(\bO)}\right)\leq C\left(\frac{1}{\lambda_1}+1\right)\|Au\|_{L^2(\bO)},
	\end{equation}
	which leads to the identification $D(A)=H^2(\bO)\cap H_0^1(\bO)$. Hence, the graph norm $\|A u\|_{L^2} $ is equivalent to the norm induced by $ H^2(\mathcal{O}) $ and the norm $\|(I+A)u\|_{L^2(\bO)}$, and $$D(A)=D(I+A)=H^2(\bO)\cap H_0^1(\bO).$$
	By choosing $\mathfrak{f}(\cdot, \cdot) = a(\cdot,\cdot)$ and $\Fn = A,$ in  Theorem \ref{Thm-Rep}, we also have $	D(A^{\frac{1}{2}}) = H_0^1(\bO).$
	
	\subsection{Nonlinear operator}
	Consider the nonlinear operator
	\begin{align}\label{Def-Nn}
		\Nn: L^p( \bO)\to L^{p^\prime}( \bO )\;\; \mbox{by} \;\; \Nn(u) := \abs{u}^{p-2} u, 
	\end{align}
	where $p^\prime = \frac{p}{p-1}$. As shown in \cite[Section 2.4]{SG+MTM-25} and \cite[p. 626]{MTM-22}, the nonlinear operator $\Nn$ is monotone in the following sense: for any $p \geq 2$,
	\begin{align*}
		\langle \Nn(u) - \Nn(v), u-v \rangle
		\geq \int_\bO \left(\abs{u(x)}^{p-1}- \abs{v(x)}^{p-1}\right) (\abs{u(x)} - \abs{v(x)})\, dx \geq 0.
	\end{align*}
	Additionally, the following inequality holds:
	\begin{align}
		\langle \Nn(u) - \Nn(v), u-v \rangle \geq \frac{1}{2} \big\|\abs{u}^{\frac{p}{2}-1}(u-v)\big\|^2_{L^2(\bO)} + \frac{1}{2} \big\|\abs{v}^{\frac{p}{2}-1}(u-v)\big\|^2_{L^2(\bO)}.\label{eqn-mono-2}
	\end{align}

	\subsection{Gradient flow}\label{Subsec-Gradient-flow}
	We now show that the energy functional $\En$, defined in \eqref{Def-En}, exhibits dissipative behavior over time. First, we demonstrate that the strong solution to problem \eqref{eqn-main-heat}, as characterized in Definition \ref{Def-strong-soln}, evolves as a gradient flow and obeys the energy inequality \eqref{eqn-energy-eqn} given below; see also \cite[Remark 4.12]{AB+ZB+MTM-25+}.

	For a fixed $u\in\bM$,   the  gradient of $\En$ tangent to $\bM$ is given by 
	\begin{align}
		\nabla_{\bM}\En(u) & = \pi_u(\nabla\En(u)) = \pi_u\left(-\Delta u+|u|^{p-2}u \right)\\
		& = \left(-\Delta u+|u|^{p-2}u \right) - \left( \norm{\nabla u}_{L^2(\bO)}^2 + \norm{u}_{L^p(\bO)}^p \right)u, \label{eqn-tangent-gradient}
	\end{align}
	where $\nabla_{\bM}$ is gradient of $\En$ on the tangent $T_u\bM$ and $\pi_u(\cdot)$ is the orthogonal projection onto $T_u\bM$.
	%, i.e., $\nabla_{\bM}\En(u)=\pi_u(\nabla\En)$.
	Accordingly, problem \eqref{eqn-main-heat} can thus be reformulated as
	\begin{align}\label{eqn-gradient-flow}
		\left\{\begin{aligned}
			\frac{\partial u(t)}{\partial t} & = -\nabla_{\bM}\En(u(t)),\\
			u(0) & = u_0, \\
			u(t)|_{\partial\bO} & = 0.
		\end{aligned}\right.
	\end{align}
	Hence the strong solution $u$ of the above problem is a \textit{gradient flow}.
	Moreover, for any $u\in\bM$ and a.e. $t\in[0,T]$, we also have
	\begin{align}
		\frac{d}{dt}\En(u(t))&=\left(\nabla_{\bM}\En(u(t)),\frac{\partial u(t)}{\partial t}\right)=\left(\nabla_{\bM}\En(u(t)),-\nabla_{\bM}\En(u(t))\right)\\&=-\|\nabla_{\bM}\En(u(t))\|_{L^2(\bO)}^2,\label{eqn-En'}
	\end{align}
	and so
	\begin{align}\label{eqn-energy-eqn}
		\En(u(t))+\int_0^t\|\nabla_{\bM}\En(u(s))\|_{L^2(\bO)}^2ds = \En(u_0),\ t\in[0,T].
	\end{align}
	It implies that $\En(u(\cdot))$ is decreasing in time.
	Thus, by using \eqref{eqn-gradient-flow} in the above equation, for all $t\geq 0$, we also have \eqref{eqn-dec-ener}.

	\section{Proof of Theorem \ref{Thm-Main-Exis}}\label{Sec-Thm-Exi}
	In this section, the main objective is to establish the proof of Theorem \ref{Thm-Main-Exis} in two parts. We begin by showing that the cut-off operator $\gn^K$ (see \eqref{Def-gn}) satisfies some useful estimate and is demicontinuous in nature. Using this, we prove that the main cut-off operator $\Gn^K + \Gamma I$ (see \eqref{Def-Gn-K}) satisfies the $m-$accretivity property. Thus, with the help of the cut-off technique and the abstract theory of $m-$accretive operators, in the first part, we show that there exists a unique global strong solution $u$ to the problem \eqref{eqn-main-heat} in the space $W^{1,\infty}([0,T];L^2(\bO))\cap L^{\infty}(0,T; D(A))$. In the second part, using the Yosida approximation, we show that the strong solution obtained in the first part indeed belongs to a higher regularity class, namely $u \in L^2(0,T; D(A^{\frac{3}{2}}))\cap C([0,T]; D(A))$ with $\frac{\partial u}{\partial t} \in L^2(0,T; H_0^1(\bO))$. Finally, we conclude this section with a time-regularity result for the strong solution $u$ obtained in the proof of Theorem \ref{Thm-Main-Exis}.
	
	To begin, let us fix $\bO\subset\R^d$ to be a Poincar\'e domain with $C^2-$boundary and choose
	\begin{align}\label{eqn-D(A)-in-L^p-assump}
		p \in \left\{
		\begin{aligned}
			[2,\infty), & \text{ when }\  1 \le d\le 4,\\ 
			\Big[2, \frac{2d-4}{d-4}\Big], & \text{ when }\  d \ge 5,
		\end{aligned}
		\right.
	\end{align}
	throughout this section, unless specified. In particular, the second condition in \eqref{eqn-D(A)-in-L^p-assump} suggests that the Sobolev embedding $D(A)\embed L^{2p-2}(\bO)$ is valid.
	
	Let us now consider the nonlinear operator of interest $D(A) \ni u \mapsto \Gn(u) \in L^2(\bO)$ defined as 
	\begin{align*}
		\Gn(u) & := A u + \abs{u}^{p-2}u - \big(\norm{\nabla u }_{L^2(\bO)}^2 + \norm{u}_{L^p(\bO)}^p\big)u\\
		& =: A u + \abs{u}^{p-2}u - \gn(u).
	\end{align*}
	For fixed $K\in\N$, we consider a modified operator
	\begin{align}
		\Gn^K(u):=  A u + \abs{u}^{p-2}u  - \gn^K(u),\label{Def-Gn-K}
	\end{align}
	where the cutt-off map $\gn^K: D(A) \to L^2(\bO)$ is defined as
	\begin{align}\label{Def-gn}
		\gn^K(u) := \left\{\begin{aligned}
			\gn(u)&,\ \text{ if }\ \norm{\nabla u }_{L^2(\bO)}^2 + \norm{u}_{L^p(\bO)}^p \leq K,\\
			\frac{K^2}{(\norm{\nabla u }_{L^2(\bO)}^2 + \norm{u}_{L^p(\bO)}^p)^2} \gn(u)&,\ \text{ if }\ \norm{\nabla u }_{L^2(\bO)}^2 + \norm{u}_{L^p(\bO)}^p > K.
		\end{aligned}\right.
	\end{align}

	\subsection{Monotone-type estimates and demicontinuity}\label{Sec-bound-demi}
	In this subsection, we show that the nonlinear cut-off operator $\gn^K: D(A) \to L^2(\bO),$ defined in \eqref{Def-gn}, satisfies some estimates that will be utilized in the next subsection (to prove that the main modified operator $\Gn^K+\Gamma I$ is monotone) and it is demicontinous.

	\begin{definition}\label{Def-Mono-Demi-Coer}
		Let us consider the map $N: D(A) \to L^2(\bO)$. 
		\begin{enumerate}
			\item The map $N$, with $D(N)= D(A)$, is called \emph{monotone} if and only if
			\begin{align*}
				(N (x) - N (y), x-y ) \geq 0,\ \text{ for all }\ x,y \in D(N).
			\end{align*}
			\item Let $\{\psi_k\}_{k\in \N} \subset D(A)$ be any sequence with $\psi_k \to \psi$ in $D(A)$, as $k\to\infty$. Then, the map $N$ is \emph{demicontinuous} if and only if
			$$ (N(\psi_k) - N(\psi), \eta)  \to 0\, \text{ as } \ k\to \infty,\ \text{ for any }\ \eta\in L^2(\bO).$$
			\item We say the map $N$ is \emph{coercive}, if and only if for any $\eta\in D(A)$, we have the following:
			\begin{align*}
				\lim_{\|\eta\|_{L^2(\bO)} \to \infty} \frac{ (N(\eta), \eta)}{\|\eta\|_{L^2(\bO)}} =\infty.
			\end{align*}
			%	i.e. for every $\eps>0$ there exists $R>0$ such that whenever $\|\eta\|_{D(A)}>R$, one has 
			%	\[
			%	\frac{ (N(\eta), \eta)}{\|\eta\|_{L^2(\bO)}}> \eps, \]
			%	where $(\cdot,\cdot)$ denotes the $L^2-$inner product.
		\end{enumerate}
	\end{definition}

	\begin{lemma}\label{Lem-gn-mono} 
		The  map $$D(A)\ni u \mapsto \gn^K(u) \in L^2(\bO)$$ is well-defined and it satisfies the following: for every $u,v\in D(A),$  there exists $C(K)>0$  such that
		\begin{align}
			(\gn^K(u) - \gn^K(v), u-v)  & \le \frac{1}{2}\norm{\nabla(u-v)}_{L^2(\bO)}^2 + C(K)\norm{u-v}_{L^2(\bO)}^2 \\
			& \quad + \frac{1}{4} \Big[\big\|\abs{u}^{\frac{p}{2}-1} \abs{u-v}\big\|_{L^2(\bO)}^2 + \big\|\abs{v}^{\frac{p}{2}-1} \abs{u-v}\big\|_{L^2(\bO)}^2 \Big],\label{eqn-gn-3-loc-mon}
		\end{align}
		where the constant
		\begin{align*}
			C(K) & = [1 + (2+ p^2 2^{2p-3})K\lambda_1^{-1}]K.
		\end{align*}
	\end{lemma}
	
	We discuss the proof of this lemma after proving the following two results. 
	
	\begin{lemma}\label{Lem-gn-1-1} 
		Let $u,v\in D(A)$ be such that $\norm{\nabla u }_{L^2(\bO)}^2 + \norm{u}_{L^p(\bO)}^p, \norm{\nabla v }_{L^2(\bO)}^2 + \norm{v}_{L^p(\bO)}^p > K$.  Then, the following inequality holds:
		\begin{align}
			\frac{K^2 (\norm{\nabla u}_{L^2(\bO)} + \norm{\nabla v}_{L^2(\bO)})\norm{\nabla v }_{L^2(\bO)}}{(\norm{\nabla u }_{L^2(\bO)}^2 + \norm{u}_{L^p(\bO)}^p)(\norm{\nabla v }_{L^2(\bO)}^2 + \norm{v}_{L^p(\bO)}^p)} \leq  2K.\label{eqn-gn-1-1}
		\end{align}
	\end{lemma}

	\begin{proof}
		The proof of \eqref{eqn-gn-1-1} is divided into four cases, we treat each one by one. Let us choose and fix $u,v \in D(A)$, and suppose $\norm{\nabla u }_{L^2(\bO)}^2 + \norm{u}_{L^p(\bO)}^p, \norm{\nabla v }_{L^2(\bO)}^2 + \norm{v}_{L^p(\bO)}^p > K$.
		
		\vskip 1mm
		\noindent
		\textit{Case I.} For $\norm{\nabla u}_{L^2(\bO)}^2, \norm{\nabla v}_{L^2(\bO)}^2 \leq K$,
		we have 
		\begin{align*}
			\frac{K^2 (\norm{\nabla u}_{L^2(\bO)} + \norm{\nabla v}_{L^2(\bO)})\norm{\nabla v }_{L^2(\bO)}}{(\norm{\nabla u }_{L^2(\bO)}^2 + \norm{u}_{L^p(\bO)}^p)(\norm{\nabla v }_{L^2(\bO)}^2 + \norm{v}_{L^p(\bO)}^p)}
			\le \frac{K^2 (K^{\frac{1}{2}} + K^{\frac{1}{2}})K^{\frac{1}{2}}}{K^2}
			= 2K,
		\end{align*}
		where we have used the fact that  $\frac{1}{\norm{\nabla u }_{L^2(\bO)}^2 + \norm{u}_{L^p(\bO)}^p},\frac{1}{\norm{\nabla v }_{L^2(\bO)}^2 + \norm{v}_{L^p(\bO)}^p} < \frac{1}{K}$.
		\vskip 1mm
		\noindent
		\textit{Case II.} For $\norm{\nabla u}_{L^2(\bO)}^2, \norm{\nabla v}_{L^2(\bO)}^2 > K$, we consider
		\begin{align*}
			&\frac{K^2 (\norm{\nabla u}_{L^2(\bO)} + \norm{\nabla v}_{L^2(\bO)})\norm{\nabla v }_{L^2(\bO)}}{(\norm{\nabla u }_{L^2(\bO)}^2 + \norm{u}_{L^p(\bO)}^p)(\norm{\nabla v }_{L^2(\bO)}^2 + \norm{v}_{L^p(\bO)}^p)}\\
			& < \frac{K^2 \norm{\nabla u}_{L^2(\bO)}\norm{\nabla v }_{L^2(\bO)}}{\norm{\nabla u }_{L^2(\bO)}^2 \norm{\nabla v }_{L^2(\bO)}^2} + \frac{K^2  \norm{\nabla v}_{L^2(\bO)}^2}{K\norm{\nabla v }_{L^2(\bO)}^2}
			\le \frac{K^2 }{K} + K
			= 2K,
		\end{align*}
		where we have used the estimates $\frac{1}{\norm{\nabla u }_{L^2(\bO)}^2 + \norm{u}_{L^p(\bO)}^p} \le \frac{1}{\norm{\nabla u }_{L^2(\bO)}^2}$ and $\frac{1}{\norm{\nabla v }_{L^2(\bO)}^2 + \norm{v}_{L^p(\bO)}^p} \le \frac{1}{\norm{\nabla v }_{L^2(\bO)}^2}$.
		\vskip 1mm
		\noindent
		\textit{Case III.} For $\norm{\nabla u}_{L^2(\bO)}^2\leq K$ and $ \norm{\nabla v}_{L^2(\bO)}^2 > K$, we find
		\begin{align*}
			&\frac{K^2 (\norm{\nabla u}_{L^2(\bO)} + \norm{\nabla v}_{L^2(\bO)})\norm{\nabla v }_{L^2(\bO)}}{(\norm{\nabla u }_{L^2(\bO)}^2 + \norm{u}_{L^p(\bO)}^p)(\norm{\nabla v }_{L^2(\bO)}^2 + \norm{v}_{L^p(\bO)}^p)}\\
			& < \frac{K^2 \norm{\nabla u}_{L^2(\bO)}\norm{\nabla v }_{L^2(\bO)}}{ K \norm{\nabla v }_{L^2(\bO)}^2} + \frac{K^2  \norm{\nabla v}_{L^2(\bO)}^2}{K\norm{\nabla v }_{L^2(\bO)}^2}
			\le \frac{K^2 }{K} + K
			= 2K,
		\end{align*}
		where we have used the estimates  $\frac{1}{\norm{\nabla u }_{L^2(\bO)}^2 + \norm{u}_{L^p(\bO)}^p} < \frac{1}{K}$, $\frac{1}{\norm{\nabla v }_{L^2(\bO)}^2 + \norm{v}_{L^p(\bO)}^p} \le \frac{1}{\norm{\nabla v }_{L^2(\bO)}^2}$ and $\frac{\norm{\nabla u}_{L^2(\bO)}}{\norm{\nabla v}_{L^2(\bO)}} <1$.
		\vskip 1mm
		\noindent
		\textit{Case IV.} For $\norm{\nabla u}_{L^2(\bO)}^2> K$ and $ \norm{\nabla v}_{L^2(\bO)}^2 \le K$, we have 
		\begin{align*}
			&\frac{K^2 (\norm{\nabla u}_{L^2(\bO)} + \norm{\nabla v}_{L^2(\bO)})\norm{\nabla v }_{L^2(\bO)}}{(\norm{\nabla u }_{L^2(\bO)}^2 + \norm{u}_{L^p(\bO)}^p)(\norm{\nabla v }_{L^2(\bO)}^2 + \norm{v}_{L^p(\bO)}^p)}\\
			& < \frac{K^2 \norm{\nabla u}_{L^2(\bO)}\norm{\nabla v }_{L^2(\bO)}}{ \norm{\nabla u }_{L^2(\bO)}^2 K} + \frac{K^2  \norm{\nabla v}_{L^2(\bO)}^2}{K\norm{\nabla v }_{L^2(\bO)}^2}
			\le \frac{K^2 }{K} + K
			= 2K,
		\end{align*}
		where we have used the fact that  $\frac{1}{\norm{\nabla v }_{L^2(\bO)}^2 + \norm{v}_{L^p(\bO)}^p} < \frac{1}{K}$, $\frac{1}{\norm{\nabla v }_{L^2(\bO)}^2 + \norm{v}_{L^p(\bO)}^p} \le \frac{1}{\norm{\nabla v }_{L^2(\bO)}^2}$ and $\frac{\norm{\nabla v}_{L^2(\bO)}}{\norm{\nabla u}_{L^2(\bO)}} <1$.
		
		Finally, from all the above four cases, we deduce
		\begin{align*}
			\frac{K^2 (\norm{\nabla u}_{L^2(\bO)} + \norm{\nabla v}_{L^2(\bO)})\norm{\nabla v }_{L^2(\bO)}}{(\norm{\nabla u }_{L^2(\bO)}^2 + \norm{u}_{L^p(\bO)}^p)(\norm{\nabla v }_{L^2(\bO)}^2 + \norm{v}_{L^p(\bO)}^p)} \leq 2K,
		\end{align*}
		which completes the proof.
	\end{proof}
	
	\begin{lemma}\label{Lem-gn-1-2} 
		Let  $u,v\in D(A)$ such that $\norm{\nabla u }_{L^2(\bO)}^2 + \norm{u}_{L^p(\bO)}^p, \norm{\nabla v }_{L^2(\bO)}^2 + \norm{v}_{L^p(\bO)}^p > K$.  Then, the following inequality holds:
		\begin{align}
			\frac{K^4 (\norm{u}_{L^p(\bO)}^{p} + \norm{v}_{L^p(\bO)}^{p})}{(\norm{\nabla u }_{L^2(\bO)}^2 + \norm{u}_{L^p(\bO)}^p)^2(\norm{\nabla v }_{L^2(\bO)}^2 + \norm{v}_{L^p(\bO)}^p)} < 2K^{2}.\label{eqn-gn-1-2}
		\end{align}
	\end{lemma}
	
	\begin{proof}
		As in the proof of previous lemma, the proof is divided into four cases. Let us choose and fix $u,v \in D(A)$, and suppose $\norm{\nabla u }_{L^2(\bO)}^2 + \norm{u}_{L^p(\bO)}^p, \norm{\nabla v }_{L^2(\bO)}^2 + \norm{v}_{L^p(\bO)}^p > K$.
		
		\vskip 1mm
		\noindent
		\textit{Case I.} For $\norm{u}_{L^p(\bO)}^p, \norm{v}_{L^p(\bO)}^p \leq K$,
		we have 
		\begin{align*}
			\frac{K^4 (\norm{u}_{L^p(\bO)}^{p} + \norm{v}_{L^p(\bO)}^{p})}{(\norm{\nabla u }_{L^2(\bO)}^2 + \norm{u}_{L^p(\bO)}^p)^2(\norm{\nabla v }_{L^2(\bO)}^2 + \norm{v}_{L^p(\bO)}^p)}
			\le \frac{K^4 (K + K)}{K^3}
			= 2K^2,
		\end{align*}
		where we have used the fact that $\frac{1}{\norm{\nabla u }_{L^2(\bO)}^2 + \norm{u}_{L^p(\bO)}^p},\frac{1}{\norm{\nabla v }_{L^2(\bO)}^2 + \norm{v}_{L^p(\bO)}^p} < \frac{1}{K}$.
		\vskip 1mm
		\noindent
		\textit{Case II.} For $\norm{u}_{L^p(\bO)}^p, \norm{v}_{L^p(\bO)}^p > K$, we consider
		\begin{align*}
			&\frac{K^4 (\norm{u}_{L^p(\bO)}^{p} + \norm{v}_{L^p(\bO)}^{p})}{(\norm{\nabla u }_{L^2(\bO)}^2 + \norm{u}_{L^p(\bO)}^p)^2(\norm{\nabla v }_{L^2(\bO)}^2 + \norm{v}_{L^p(\bO)}^p)}\\
			& < \frac{K^3 \norm{u}_{L^p(\bO)}^{p}}{\norm{ u }_{L^p(\bO)}^p K} + \frac{K^3  \norm{v}_{L^p(\bO)}^{p}}{K\norm{ v }_{L^p(\bO)}^p}
			< K^2 + K^2 = 2K^2,
		\end{align*}
		where we have used the estimates $\frac{1}{\norm{\nabla u }_{L^2(\bO)}^2 + \norm{u}_{L^p(\bO)}^p} \le \frac{1}{\norm{u }_{L^p(\bO)}^p}$ and $\frac{1}{\norm{\nabla v }_{L^2(\bO)}^2 + \norm{v}_{L^p(\bO)}^p} \le \frac{1}{\norm{ v }_{L^p(\bO)}^p}$.
		\vskip 1mm
		\noindent
		\textit{Case III.} For $\norm{ u}_{L^p(\bO)}^p \leq K$ and $ \norm{ v}_{L^p(\bO)}^p > K$, we obtain 
		\begin{align*}
			\frac{K^4 (\norm{u}_{L^p(\bO)}^{p} + \norm{v}_{L^p(\bO)}^{p})}{(\norm{\nabla u }_{L^2(\bO)}^2 + \norm{u}_{L^p(\bO)}^p)^2(\norm{\nabla v }_{L^2(\bO)}^2 + \norm{v}_{L^p(\bO)}^p)} < K \norm{u}_{L^p(\bO)}^{p} + K^2 
			< 2K^2,
		\end{align*}
		where we have used the fact that $\frac{1}{\norm{\nabla u }_{L^2(\bO)}^2 + \norm{u}_{L^p(\bO)}^p}, \frac{1}{\norm{\nabla v }_{L^2(\bO)}^2 + \norm{v}_{L^p(\bO)}^p} < \frac{1}{K}$ and $\frac{1}{\norm{\nabla v }_{L^2(\bO)}^2 + \norm{v}_{L^p(\bO)}^p} \le \frac{1}{\norm{ v }_{L^p(\bO)}^p}$. 
		\vskip 1mm
		\noindent
		\textit{Case IV.} For $\norm{ u}_{L^p(\bO)}^p> K$ and $ \norm{ v}_{L^p(\bO)}^p \le K$, we get 
		\begin{align*}
			\frac{K^4 (\norm{u}_{L^p(\bO)}^{p} + \norm{v}_{L^p(\bO)}^{p})}{(\norm{\nabla u }_{L^2(\bO)}^2 + \norm{u}_{L^p(\bO)}^p)^2(\norm{\nabla v }_{L^2(\bO)}^2 + \norm{v}_{L^p(\bO)}^p)}	 < K^2 + {K \norm{v}_{L^p(\bO)}^{p}}
			%< K^2 + K^2
			< 2K^2,
		\end{align*}
		where we have used the estimates $\frac{1}{\norm{\nabla u }_{L^2(\bO)}^2 + \norm{u}_{L^p(\bO)}^p}, \frac{1}{\norm{\nabla v }_{L^2(\bO)}^2 + \norm{v}_{L^p(\bO)}^p} < \frac{1}{K}$ and $\frac{1}{\norm{\nabla u }_{L^2(\bO)}^2 + \norm{u}_{L^p(\bO)}^p} \le \frac{1}{\norm{ v }_{L^p(\bO)}^p}$. Finally, from all the above four cases, we deduce
		\begin{align*}
			\frac{K^4 (\norm{u}_{L^p(\bO)}^{p} + \norm{v}_{L^p(\bO)}^{p})}{(\norm{\nabla u }_{L^2(\bO)}^2 + \norm{u}_{L^p(\bO)}^p)^2(\norm{\nabla v }_{L^2(\bO)}^2 + \norm{v}_{L^p(\bO)}^p)} < 2K^2.
		\end{align*}
		which completes the proof.
	\end{proof}
	
	We are now ready to prove Lemma \ref{Lem-gn-mono} by making use of the two lemmas established above.
	
	\begin{proof}[Proof of Lemma \ref{Lem-gn-mono}]
		Note that the well-definedness of $\gn^K:D(A) \to L^2(\bO)$, defined in \eqref{Def-gn}, follows by an application of the Sobolev  inequality. 
		
		Let us now choose and fix $u,v \in D(A)$. 
		An application of Taylor's formula and H\"older's inequality (with exponent $2$ and $2$) assert
		\begin{align}
			&\big( \norm{u}_{L^p(\bO)}^p-\norm{v}_{L^p(\bO)}^p \big)( v,  (u-v))\\
			& \leq p \int_0^1 (\abs{\theta u + (1-\theta)v}^{p-2}|\theta u + (1-\theta)v|, |u-v|) d\theta \norm{v}_{L^2(\bO)}\norm{u-v}_{L^2(\bO)}\\
			& \leq p2^{p-2} (\abs{u}^{p-1} + \abs{v}^{p-1}, \abs{u-v})\norm{ v}_{L^2(\bO)}\norm{  u-v}_{L^2(\bO)} \\
			& = p2^{p-2} \left[\left(\abs{u}^{\frac{p}{2}-1} \abs{u-v}, \abs{u}^{\frac{p}{2}}\right) + \left(\abs{v}^{\frac{p}{2}-1} \abs{u-v}, \abs{v}^{\frac{p}{2}}\right)\right] \norm{v}_{L^2(\bO)}\norm{u-v}_{L^2(\bO)}\\
			& \leq \frac{1}{4} \Big[\big\|\abs{u}^{\frac{p}{2}-1} \abs{u-v}\big\|_{L^2(\bO)}^2 + \big\|\abs{v}^{\frac{p}{2}-1} \abs{u-v}\big\|_{L^2(\bO)}^2 \Big]\\
			& \quad + p^2 2^{2p-4}\big(\norm{u}_{L^p(\bO)}^{p} + \norm{v}_{L^p(\bO)}^{p}\big)\norm{ v}_{L^2(\bO)}^2\norm{u-v}_{L^2(\bO)}^2.\label{eqn-gn-4-mono-2}
		\end{align}
		Using the above inequality, we derive the estimate \eqref{eqn-gn-3-loc-mon} in the following three cases:
		\vskip 1mm
		\noindent
		\textit{Case I.} $\norm{\nabla u}_{L^2(\bO)}^2 + \norm{u}_{L^p(\bO)}^p, \norm{\nabla v}_{L^2(\bO)}^2 + \norm{v}_{L^p(\bO)}^p \le K$.
		
		By using definition of $\gn^K$ (see \eqref{Def-gn}), integration by parts, inequality \eqref{eqn-gn-4-mono-2}, Young's and Poincar\'e's inequalities, we deduce
		\begin{align*}
			&(\gn^K(u) - \gn^K(v), u-v) = ((\norm{\nabla u}_{L^2(\bO)}^2 + \norm{u}_{L^p(\bO)}^p)u - (\norm{\nabla v}_{L^2(\bO)}^2 + \norm{v}_{L^p(\bO)}^p)v, u-v)\\
			& =	 (\norm{\nabla u}_{L^2(\bO)}^2 u - \norm{\nabla v}_{L^2(\bO)}^2 v, u-v) + (\norm{u}_{L^p(\bO)}^p u - \norm{v}_{L^p(\bO)}^p v, u-v)\\
			&  =  \norm{\nabla u}^2_{L^2(\bO)}\norm{u-v}^2_{L^2(\bO)} + (\norm{\nabla u}_{L^2(\bO)}^2-\norm{\nabla v}_{L^2(\bO)}^2)(v, u-v)\\
			& \quad + {\norm{ u }_{L^p(\bO)}^p} \norm{u - v}_{L^2(\bO)}^2 + \big({\norm{ u }_{L^p(\bO)}^p} -  {\norm{v}_{L^p(\bO)}^p}\big)(v, u-v)\\
			& =  \big[\norm{\nabla u}^2_{L^2(\bO)} + {\norm{ u }_{L^p(\bO)}^p} \big]\norm{u-v}^2_{L^2(\bO)} + (\norm{\nabla u}_{L^2(\bO)}^2-\norm{\nabla v}_{L^2(\bO)}^2)(v, u-v)\\
			& \quad + \big({\norm{ u }_{L^p(\bO)}^p} -  {\norm{v}_{L^p(\bO)}^p}\big)(v, u-v)\\
			& \leq 	K\norm{u-v}_{L^2(\bO)}^2 + (\norm{\nabla u}_{L^2(\bO)} + \norm{\nabla v}_{L^2(\bO)})\norm{v}_{L^2(\bO)} \norm{\nabla(u-v)}_{L^2(\bO)}\norm{u-v}_{L^2(\bO)}\\
			& \quad + \frac{1}{4} \Big[\big\|\abs{u}^{\frac{p}{2}-1} \abs{u-v}\big\|_{L^2(\bO)}^2 + \big\|\abs{v}^{\frac{p}{2}-1} \abs{u-v}\big\|_{L^2(\bO)}^2 \Big]
			+ p^2 2^{2p-3}\frac{K^2}{\lambda_1}\norm{u-v}_{L^2(\bO)}^2\\
			& \leq 	\frac{1}{2}\norm{\nabla(u-v)}_{L^2(\bO)}^2+ [1 + (2 + p^2 2^{2p-3})\lambda_1^{-1}K] K \norm{u-v}_{L^2(\bO)}^2\\ 
			& \quad + \frac{1}{4} \Big[\big\|\abs{u}^{\frac{p}{2}-1} \abs{u-v}\big\|_{L^2(\bO)}^2 + \big\|\abs{v}^{\frac{p}{2}-1} \abs{u-v}\big\|_{L^2(\bO)}^2 \Big]\\
			%+ \frac{K^2}{\lambda_1}\norm{u-v}_{L^2(\bO)}^2\\
			& \leq \frac{1}{2}\norm{\nabla(u-v)}_{L^2(\bO)}^2 + C(K) \norm{u - v}_{L^2(\bO)}^2+ \frac{1}{4} \Big[\big\|\abs{u}^{\frac{p}{2}-1} \abs{u-v}\big\|_{L^2(\bO)}^2 + \big\|\abs{v}^{\frac{p}{2}-1} \abs{u-v}\big\|_{L^2(\bO)}^2 \Big],
		\end{align*}
		where $\lambda_1$ denotes the Poincar\'e constant and  
		\begin{align*}
			C(K) & = [1 + (2 + p^2 2^{2p-3})\lambda_1^{-1}K] K.
		\end{align*}
		
		\vskip 1mm
		\noindent
		\textit{Case II.} $\norm{\nabla u}_{L^2(\bO)}^2 + \norm{u}_{L^p(\bO)}^p, \norm{\nabla v}_{L^2(\bO)}^2 + \norm{v}_{L^p(\bO)}^p > K$.
		
		Similar to the previous case, by utilizing the definition of $\gn^K$ (see \eqref{Def-gn}) and  integration by parts, we obtain
		\begin{align}
			&(\gn^K(u) - \gn^K(v), u-v)  = 
			\bigg( \frac{K^{2} }{\norm{\nabla u }_{L^2(\bO)}^2 + \norm{ u }_{L^p(\bO)}^p} u -  \frac{K^{2} }{\norm{\nabla v }_{L^2(\bO)}^2 +\norm{ v}_{L^p(\bO)}^p} v , u-v\bigg)\\
			& =  \frac{K^{2}}{\norm{\nabla u }_{L^2(\bO)}^2 + \norm{ u }_{L^p(\bO)}^p} \norm{u - v}_{L^2(\bO)}^2 \\
			& \quad + \bigg[\frac{K^{2}}{\norm{\nabla u }_{L^2(\bO)}^2 + \norm{ u }_{L^p(\bO)}^p} -  \frac{K^{2}}{\norm{\nabla v }_{L^2(\bO)}^2 + \norm{v}_{L^p(\bO)}^p}\bigg](v, u-v)\\
			& \leq  K\norm{u - v}_{L^2(\bO)}^2 \\
			& \quad + K^{2}\bigg[\frac{\norm{\nabla v }_{L^2(\bO)}^2 + \norm{ v }_{L^p(\bO)}^p -\norm{\nabla u }_{L^2(\bO)}^2 - \norm{ u }_{L^p(\bO)}^p}{(\norm{\nabla u }_{L^2(\bO)}^2 + \norm{ u }_{L^p(\bO)}^p)(\norm{\nabla v }_{L^2(\bO)}^2 + \norm{ v }_{L^p(\bO)}^p)} \bigg]\frac{\norm{\nabla v}_{L^2(\bO)}}{\lambda_1^{\frac{1}{2}}} \norm{u - v}_{L^2(\bO)} \\
			\nonumber& \le K \norm{u - v}_{L^2(\bO)}^2 + \frac{K^{2}}{\lambda_1^{\frac{1}{2}}}\bigg[\frac{(\norm{\nabla v }_{L^2(\bO)} + \norm{\nabla u }_{L^2(\bO)}) \norm{\nabla v }_{L^2(\bO)} \norm{\nabla (u - v)}_{L^2(\bO)}}{(\norm{\nabla u }_{L^2(\bO)}^2 + \norm{ u }_{L^p(\bO)}^p)(\norm{\nabla v }_{L^2(\bO)}^2 + \norm{ v }_{L^p(\bO)}^p)} \bigg]  \norm{u - v}_{L^2(\bO)} \\
			& \quad + \frac{K^{2}}{\lambda_1^{\frac{1}{2}}}\bigg[\frac{ \big(\norm{ v }_{L^p(\bO)}^p - \norm{ u }_{L^p(\bO)}^p \big) }{(\norm{\nabla u }_{L^2(\bO)}^2 + \norm{ u }_{L^p(\bO)}^p)(\norm{\nabla v }_{L^2(\bO)}^2 + \norm{ v }_{L^p(\bO)}^p)} \bigg]\norm{\nabla v}_{L^2(\bO)}  \norm{u - v}_{L^2(\bO)}.\label{eqn-est-gn-1-1}
		\end{align}
		Therefore, using \eqref{eqn-gn-1-1} (see Lemma \ref{Lem-gn-1-1}),  \eqref{eqn-gn-4-mono-2} and \eqref{eqn-gn-1-2} (see Lemma \ref{Lem-gn-1-2}) in \eqref{eqn-est-gn-1-1}, we infer
		\begin{align*}
			& (\gn^K(u) - \gn^K(v), u-v)\\
			& \le \frac{1}{2}\norm{\nabla(u-v)}_{L^2(\bO)}^2 + [1 + 2 \lambda_1^{-1}K] K \norm{u - v}_{L^2(\bO)}^2\\
			&\quad + \frac{1}{4} \Big[\big\|\abs{u}^{\frac{p}{2}-1} \abs{u-v}\big\|_{L^2(\bO)}^2 + \big\|\abs{v}^{\frac{p}{2}-1} \abs{u-v}\big\|_{L^2(\bO)}^2 \Big]\\
			& \quad + \frac{p^2 2^{2p-4}}{\lambda_1}K^{4}\bigg[ \frac{ [\norm{ v }_{L^p(\bO)}^p +  \norm{ u }_{L^p(\bO)}^p]\norm{\nabla v}_{L^2(\bO)}^2}{(\norm{\nabla u }_{L^2(\bO)}^2 + \norm{ u }_{L^p(\bO)}^p)^2(\norm{\nabla v }_{L^2(\bO)}^2 + \norm{ v }_{L^p(\bO)}^p)^2} \bigg] \norm{u - v}_{L^2(\bO)}^2\\
			& \le \frac{1}{2}\norm{\nabla(u-v)}_{L^2(\bO)}^2 + [1 + (2 + p^2 2^{2p-3})\lambda_1^{-1}K]K \norm{u-v}_{L^2(\bO)}^2\\
			& \quad + \frac{1}{4} \Big[\big\|\abs{u}^{\frac{p}{2}-1} \abs{u-v}\big\|_{L^2(\bO)}^2 + \big\|\abs{v}^{\frac{p}{2}-1} \abs{u-v}\big\|_{L^2(\bO)}^2 \Big]\\
			& \le \frac{1}{2}\norm{\nabla(u-v)}_{L^2(\bO)}^2 + C(K) \norm{u - v}_{L^2(\bO)}^2 + \frac{1}{4} \Big[\big\|\abs{u}^{\frac{p}{2}-1} \abs{u-v}\big\|_{L^2(\bO)}^2 + \big\|\abs{v}^{\frac{p}{2}-1} \abs{u-v}\big\|_{L^2(\bO)}^2 \Big],
		\end{align*}
		where the constant
		\begin{align*}
			C(K) = [1 + (2 + p^2 2^{2p-3})\lambda_1^{-1}K]K.
		\end{align*}
		
		{We notice that}, since $u, v \in D(A)$ are arbitrary, without loss of generality, it is enough to consider either of the cases $\norm{\nabla u}_{L^2(\bO)}^2 + \norm{u}_{L^p(\bO)}^p> K$ and $\norm{\nabla v}_{L^2(\bO)}^2 + \norm{v}_{L^p(\bO)}^p \leq K$ or $\norm{\nabla u}_{L^2(\bO)}^2 + \norm{u}_{L^p(\bO)}^p\leq K$ and $\norm{\nabla v}_{L^2(\bO)}^2 + \norm{v}_{L^p(\bO)}^p > K$.
		
		\vskip 1mm
		\noindent
		\textit{Case III.} When $\norm{\nabla u}_{L^2(\bO)}^2 + \norm{u}_{L^p(\bO)}^p \le K$ and $\norm{\nabla v}_{L^2(\bO)}^2 + \norm{v}_{L^p(\bO)}^p>K$.
		
		Again, from the definition of $\gn^K$ (see \eqref{Def-gn}) and integration by parts, we calculate 
		\begin{align*}
			&| (\gn(u) - \gn^K(v), u-v) | =
			\bigg|\bigg( \big(\norm{\nabla u}_{L^2(\bO)}^2 + \norm{u}_{L^p(\bO)}^p\big) u - \frac{K^2}{\norm{\nabla v}_{L^2(\bO)}^2 + \norm{v}_{L^p(\bO)}^p} v, u-v \bigg)\bigg| \\
			& = \big(\norm{\nabla u}_{L^2(\bO)}^2 + \norm{u}_{L^p(\bO)}^p \big) \norm{u-v}_{L^2(\bO)}^2\\
			&\quad + \bigg[\big(\norm{\nabla u}_{L^2(\bO)}^2 + \norm{u}_{L^p(\bO)}^p\big) - \frac{K^2}{\norm{\nabla v}_{L^2(\bO)}^2 + \norm{v}_{L^p(\bO)}^p}\bigg]|(v, u-v)| \\
			& \le K \norm{u-v}_{L^2(\bO)}^2  + \frac{\norm{\nabla u}_{L^2(\bO)}^2 + \norm{u}_{L^p(\bO)}^p}{\norm{\nabla v}_{L^2(\bO)}^2 + \norm{v}_{L^p(\bO)}^p} \big[ \norm{\nabla v}_{L^2(\bO)}^2 + \norm{v}_{L^p(\bO)}^p - \big(\norm{\nabla u}_{L^2(\bO)}^2 + \norm{u}_{L^p(\bO)}^p\big)\big] \\
			& \quad \times \frac{\norm{\nabla v}_{L^2(\bO)}}{\lambda_1^{\frac{1}{2}}} \norm{u-v}_{L^2(\bO)}\\
			& \le K\norm{u-v}_{L^2(\bO)}^2  + \frac{K(\norm{\nabla v}_{L^2(\bO)} + \norm{\nabla u}_{L^2(\bO)})}{\norm{\nabla v}_{L^2(\bO)}^2 + \norm{v}_{L^p(\bO)}^p} \norm{\nabla v}_{L^2(\bO)} \frac{\norm{\nabla(u-v)}_{L^2(\bO)}}{\lambda_1^{\frac{1}{2}}}\norm{u-v}_{L^2(\bO)}\\
			& \quad + \frac{K}{\norm{\nabla v}_{L^2(\bO)}^2 + \norm{v}_{L^p(\bO)}^p} \big({\norm{ v }_{L^p(\bO)}^p} -  {\norm{u}_{L^p(\bO)}^p}\big) \frac{\norm{\nabla v}_{L^2(\bO)}}{\lambda_1^{\frac{1}{2}}} \norm{u-v}_{L^2(\bO)}\\
			& \leq  \frac{1}{2}\norm{\nabla(u-v)}_{L^2(\bO)}^2 + (1 + 2K\lambda_1^{-1})K \norm{u-v}_{L^2(\bO)}^2 + p^2 2^{2p-3}K^2\lambda_1^{-1} \norm{u-v}_{L^2(\bO)}^2\\
			& \quad + \frac{1}{4} \Big[\big\|\abs{u}^{\frac{p}{2}-1} \abs{u-v}\big\|_{L^2(\bO)}^2 + \big\|\abs{v}^{\frac{p}{2}-1} \abs{u-v}\big\|_{L^2(\bO)}^2 \Big]\\
			& \leq \frac{1}{2}\norm{\nabla(u-v)}_{L^2(\bO)}^2 + C(K) \norm{u - v}_{L^2(\bO)}^2 + \frac{1}{4} \Big[\big\|\abs{u}^{\frac{p}{2}-1} \abs{u-v}\big\|_{L^2(\bO)}^2 + \big\|\abs{v}^{\frac{p}{2}-1} \abs{u-v}\big\|_{L^2(\bO)}^2 \Big],
		\end{align*}
		where we have used \eqref{eqn-gn-1-1}, \eqref{eqn-gn-4-mono-2}, \eqref{eqn-gn-1-2}, and
		\begin{align*}
			C(K) & = [1 + (2+ p^2 2^{2p-3})K\lambda_1^{-1}]K .
		\end{align*}
		Hence both the above three cases yield
		\begin{align*}
			(\gn^K(u) - \gn^K(v), u-v)  & \le \frac{1}{2}\norm{\nabla(u-v)}_{L^2(\bO)}^2 + C(K)\norm{u-v}_{L^2(\bO)}^2 \\
			& \quad + \frac{1}{4} \Big[\big\|\abs{u}^{\frac{p}{2}-1} \abs{u-v}\big\|_{L^2(\bO)}^2 + \big\|\abs{v}^{\frac{p}{2}-1} \abs{u-v}\big\|_{L^2(\bO)}^2 \Big],
		\end{align*}
		where the constant 
		\begin{align*}
			C(K) & = [1 + (2+ p^2 2^{2p-3})K\lambda_1^{-1}]K.
		\end{align*}
		This concludes the proof of lemma.
	\end{proof}

	Next, we provide a result on the demicontinuity of the cut-off operator $\gn^K:D(A)\to L^2(\bO)$, which will be used in the next subsection to prove that the modified operator $\Gn^K$ is \textit{$m-$accretive}.

	\begin{lemma}\label{gn-demi}
		The map $\gn^K: D(A) \to L^2(\bO)$ is demicontinuous.
	\end{lemma}
	
	\begin{proof}
		Let us choose and fix $\{\psi_k\}_{k\in \N} \subset D(A)$ with $\psi_k \to \psi$ in $D(A)$, as $k\to\infty$.
%		\begin{align*}
%			\norm{A\psi_k - A\psi}_{L^2(\bO)} \to 0, \text{ as }\ k \to \infty.
%		\end{align*}
		Due to the continuous embedding $D(A)\embed L^p(\bO), H_0^1(\bO)$ and the inequality \eqref{eqn-elliptic-reg} for $p$ satisfying \eqref{eqn-D(A)-in-L^p-assump}, we know that
		\begin{equation}
			 \norm{\cdot}_{L^p(\bO)} \le C_2 \norm{A\;\cdot\;}_{L^2(\bO)}\ \text{ and }\  \norm{\nabla\;\cdot\;}_{L^2(\bO)} \le C_1\norm{A\;\cdot\;}_{L^2(\bO)},\ \text{ for some }\ C_1,C_2>0,
		\end{equation}
which immediately implies from the assumptions that 
		\begin{align*}
			\norm{\psi_k - \psi}_{L^p(\bO)} \to 0,\ \norm{A(\psi_k - \psi)}_{L^2(\bO)} \to 0\ \text{ and }\ \norm{\nabla(\psi_k - \psi)}_{L^2(\bO)} \to 0,\ \text{ as }\  k\to \infty.
		\end{align*}
		Let us fix an arbitrary $K\in\N$ {and $\eta\in D(A)$}. Then, the proof is divided into four cases. 
		\vskip 1mm
		\noindent
		\textit{Case I.} When $\norm{\nabla \psi_k}_{L^2(\bO)}^2 +\norm{ \psi_k}_{L^p(\bO)}^p, \norm{\nabla \psi}_{L^2(\bO)}^2 + \norm{\psi}_{L^p(\bO)}^p \le K$.
		
		Using the definition of $\gn^K$ (from \eqref{Def-gn}) and H\"older's inequality (with exponent $2$ and $2$),  we estimate
		\begin{align*}
			&| (\gn^K(\psi_k) - \gn^K(\psi), \eta) | = \big| \big((\norm{\nabla\psi_k}_{L^2(\bO)}^2 + \norm{\psi_k}_{L^p(\bO)}^p)\psi_k - (\norm{\nabla\psi}_{L^2(\bO)}^2 + \norm{\psi}_{L^p(\bO)}^p)\psi, \eta \big) \big|\\
			& \leq \norm{\nabla\psi_k}_{L^2(\bO)}^2 \norm{\psi_k - \psi}_{L^2(\bO)}\norm{\eta}_{L^2(\bO)} + (\norm{\nabla\psi_k}_{L^2(\bO)}^2 - \norm{\nabla\psi}_{L^2(\bO)}^2 )\norm{\psi}_{L^2(\bO)} \norm{\eta}_{L^2(\bO)}\\
			&\quad + \norm{\psi_k}_{L^p(\bO)}^p\norm{\psi_k - \psi}_{L^2(\bO)}\norm{\eta}_{L^2(\bO)}  + (\norm{\psi_k}_{L^p(\bO)}^p - \norm{\psi}_{L^p(\bO)}^p) \norm{\psi}_{L^2(\bO)} \norm{\eta}_{L^2(\bO)}\\
			& \leq \big(\norm{\nabla\psi_k}_{L^2(\bO)}^2 + \norm{\psi_k}_{L^p(\bO)}^p\big)\norm{\psi_k - \psi}_{L^2(\bO)}  \norm{\eta}_{L^2(\bO)}\\
			& \quad + \big(\norm{\nabla\psi_k}_{L^2(\bO)} + \norm{\nabla\psi}_{L^2(\bO)} \big)\norm{\nabla(\psi_k - \psi)}_{L^2(\bO)}\norm{\psi}_{L^2(\bO)} \norm{\eta}_{L^2(\bO)}\\
			&\quad  + p(\norm{\psi_k}_{L^p(\bO)} + \norm{\psi}_{L^p(\bO)})^{p-1} \norm{\psi_k - \psi}_{L^p(\bO)} \norm{\psi}_{L^2(\bO)}  \norm{\eta}_{L^2(\bO)}\\
			& \to 0\ \text{as}\ k \to \infty,
		\end{align*}
		where we have used the embedding $D(A) \embed L^p(\bO)$.
		Therefore, $$| (\gn^K(\psi_k) - \gn^K(\psi), \eta) | \to 0,\ \text{ as }\ k \to \infty.$$
		\vskip 1mm
		\noindent
		\textit{Case II.} When $\norm{\nabla \psi_k}_{L^2(\bO)}^2 + \norm{\psi_k}_{L^p(\bO)}^p, \norm{\nabla \psi}_{L^2(\bO)}^2 + \norm{\psi}_{L^p(\bO)}^p>K$.
		
		Again, from the definition of $\gn^K$ (see \eqref{Def-gn}) and an application of H\"older's inequality twice (with exponents $2$ and $2$, and $p$ and $p/(p-1)$) yield 
		\begin{align*}
			&| (\gn^K(\psi_k) - \gn^K(\psi), \eta) | 
			= \bigg|\bigg( \frac{K^2}{\norm{\nabla \psi_k}_{L^2(\bO)}^2 + \norm{\psi_k}_{L^p(\bO)}^p}\psi_k - \frac{K^2}{\norm{\nabla \psi}_{L^2(\bO)}^2 + \norm{\psi}_{L^p(\bO)}^p}\psi, \eta \bigg)\bigg| \\
			& = \frac{K^2}{\norm{\nabla \psi_k}_{L^2(\bO)}^2 + \norm{\psi_k}_{L^p(\bO)}^p}|(\psi_k - \psi, \eta )|\\
			&\quad + K^2\bigg[\frac{1}{\norm{\nabla \psi_k}_{L^2(\bO)}^2 + \norm{\psi_k}_{L^p(\bO)}^p} - \frac{1}{\norm{\nabla \psi}_{L^2(\bO)}^2 + \norm{\psi}_{L^p(\bO)}^p}\bigg]|(\psi, \eta )| \\
			& \leq K\norm{\psi_k - \psi}_{L^2(\bO)}\norm{\eta}_{L^2(\bO)}  + (\norm{\nabla \psi}_{L^2(\bO)}^2 - \norm{\nabla \psi_k}_{L^2(\bO)}^2) \norm{\psi}_{L^2(\bO)} \norm{\eta}_{L^2(\bO)}\\
			&\quad + (\norm{\psi}_{L^p(\bO)}^p - \norm{\psi_k}_{L^p(\bO)}^p) \norm{\psi}_{L^2(\bO)} \norm{\eta}_{L^2(\bO)}\\
			& \leq \big[K\norm{\psi_k - \psi}_{L^2(\bO)} +  \big(\norm{\nabla\psi_k}_{L^2(\bO)} + \norm{\nabla\psi}_{L^2(\bO)} \big)\norm{\nabla (\psi_k - \psi)}_{L^2(\bO)} \norm{\psi}_{L^2(\bO)}\big]  \norm{\eta}_{L^2(\bO)} \\
			&\quad + p(\norm{\psi_k}_{L^p(\bO)} + \norm{\psi}_{L^p(\bO)})^{p-1} \norm{\psi_k - \psi}_{L^p(\bO)} \norm{\psi}_{L^2(\bO)}  \norm{\eta}_{L^2(\bO)} \\
			& \to 0\ \text{as}\ k \to \infty,
		\end{align*}
		where it is used that $D(A) \embed L^p(\bO)$. 
		%Next, without loss of generality, it is enough to consider either of the cases $\norm{\nabla \psi_k}_{L^2(\bO)}^2 + \norm{\psi_k}_{L^p(\bO)}^p> K$ and $\norm{\nabla \psi}_{L^2(\bO)}^2 + \norm{\psi}_{L^p(\bO)}^p \leq K$ or $\norm{\nabla \psi_k}_{L^2(\bO)}^2 + \norm{\psi_k}_{L^p(\bO)}^p\leq K$ and $\norm{\nabla \psi}_{L^2(\bO)}^2 + \norm{\psi}_{L^p(\bO)}^p > K$.
		
		\vskip 2mm
		\noindent
		\textit{Case III.} When $\|\nabla \psi_k\|^2_{L^2} + \|\psi_k \|^p_{L^p} \le K$ and $\|\nabla \psi\|^2_{L^2} + \|\psi \|^p_{L^p} > K$.
		
		The definition of $\gn^K$ (see \eqref{Def-gn}) and H\"older's inequality (with exponents $2$ and $2$) infer 
		\begin{align*}
			&| (\gn^K(\psi_k) - \gn^K(\psi), \eta) | 
			= \bigg|\bigg( \big(\norm{\nabla \psi_k}_{L^2(\bO)}^2 + \norm{\psi_k}_{L^p(\bO)}^p)\psi_k - \frac{K^2}{\norm{\nabla \psi}_{L^2(\bO)}^2 + \norm{\psi}_{L^p(\bO)}^p}\psi, \eta \bigg)\bigg| \\
			& \le K |(\psi_k - \psi, \eta )|\\
			&\quad + \frac{\norm{\nabla \psi_k}_{L^2(\bO)}^2 + \norm{\psi_k}_{L^p(\bO)}^p}{\norm{\nabla \psi}_{L^2(\bO)}^2 + \norm{\psi}_{L^p(\bO)}^p}\Big[{\norm{\nabla \psi}_{L^2(\bO)}^2 + \norm{\psi}_{L^p(\bO)}^p} - {\norm{\nabla \psi_k}_{L^2(\bO)}^2 - \norm{\psi_k}_{L^p(\bO)}^p}\Big]|(\psi, \eta )| \\
			& \leq K\norm{\psi_k - \psi}_{L^2(\bO)}\norm{\eta}_{L^2(\bO)}  + (\norm{\nabla \psi}_{L^2(\bO)}^2 - \norm{\nabla \psi_k}_{L^2(\bO)}^2) \norm{\psi}_{L^2(\bO)} \norm{\eta}_{L^2(\bO)}\\
			&\quad + (\norm{\psi}_{L^p(\bO)}^p - \norm{\psi_k}_{L^p(\bO)}^p) \norm{\psi}_{L^2(\bO)} \norm{\eta}_{L^2(\bO)}\\
			& \leq \big[K\norm{\psi_k - \psi}_{L^2(\bO)} +  \big(\norm{\nabla\psi_k}_{L^2(\bO)} + \norm{\nabla\psi}_{L^2(\bO)} \big)\norm{\nabla (\psi_k - \psi)}_{L^2(\bO)} \norm{\psi}_{L^2(\bO)}\big]  \norm{\eta}_{L^2(\bO)} \\
			&\quad + p(\norm{\psi_k}_{L^p(\bO)} + \norm{\psi}_{L^p(\bO)})^{p-1} \norm{\psi_k - \psi}_{L^p(\bO)} \norm{\psi}_{L^2(\bO)}  \norm{\eta}_{L^2(\bO)} \\
			& \to 0\ \text{as}\ k \to \infty.
		\end{align*}	
	\vskip 1mm
	\noindent
	\textit{Case IV.} When $\|\nabla \psi_k\|^2_{L^2} + \|\psi_k \|^p_{L^p} > K$ and $\|\nabla \psi\|^2_{L^2} + \|\psi \|^p_{L^p} \le K$.
	
	It follows from the definition of $\gn^K$ (see \eqref{Def-gn}) and H\"older's inequality (with exponents $2$ and $2$) that 
	\begin{align*}
		&| (\gn^K(\psi_k) - \gn^K(\psi), \eta) | 
		= \bigg|\bigg( \frac{K^2}{\norm{\nabla \psi_k}_{L^2(\bO)}^2 + \norm{\psi_k}_{L^p(\bO)}^p}\psi_k - \big(\norm{\nabla \psi}_{L^2(\bO)}^2 + \norm{\psi}_{L^p(\bO)}^p)\psi, \eta \bigg)\bigg| \\
		& \le K |(\psi_k - \psi, \eta )|\\
		&\quad + \frac{K}{\norm{\nabla \psi_k}_{L^2(\bO)}^2 + \norm{\psi_k}_{L^p(\bO)}^p}\Big[{\norm{\nabla \psi_k}_{L^2(\bO)}^2 + \norm{\psi_k}_{L^p(\bO)}^p} - {\norm{\nabla \psi}_{L^2(\bO)}^2 - \norm{\psi}_{L^p(\bO)}^p}\Big]|(\psi, \eta )| \\
		& \leq K\norm{\psi_k - \psi}_{L^2(\bO)}\norm{\eta}_{L^2(\bO)}  + (\norm{\nabla \psi_k}_{L^2(\bO)}^2 - \norm{\nabla \psi}_{L^2(\bO)}^2) \norm{\psi}_{L^2(\bO)} \norm{\eta}_{L^2(\bO)}\\
		&\quad + (\norm{\psi_k}_{L^p(\bO)}^p - \norm{\psi}_{L^p(\bO)}^p) \norm{\psi}_{L^2(\bO)} \norm{\eta}_{L^2(\bO)}\\
		& \leq \big[K\norm{\psi_k - \psi}_{L^2(\bO)} +  \big(\norm{\nabla\psi_k}_{L^2(\bO)} + \norm{\nabla\psi}_{L^2(\bO)} \big)\norm{\nabla (\psi_k - \psi)}_{L^2(\bO)} \norm{\psi}_{L^2(\bO)}\big]  \norm{\eta}_{L^2(\bO)} \\
		&\quad + p(\norm{\psi_k}_{L^p(\bO)} + \norm{\psi}_{L^p(\bO)})^{p-1} \norm{\psi_k - \psi}_{L^p(\bO)} \norm{\psi}_{L^2(\bO)}  \norm{\eta}_{L^2(\bO)} \\
		& \to 0\ \text{as}\ k \to \infty.
	\end{align*}
	Hence, all four cases above imply that the mapping $\gn^K: D(A) \to L^2(\bO)$ is demicontinuous.
	\end{proof}

	\subsection{M-accretivity}
	In this subsection, we turn our attention to proving one of the vital properties of the modified nonlinear operator $\Gn^K + \Gamma I$, namely \textit{$m-$accretivity}, for some $\Gamma>0$ depending on $K$. First, by using the bounds and the demicontinuity properties of the map $\gn^K$ obtained in the subsection \ref{Sec-bound-demi}, we show that $\Gn^K + \Gamma I$ is monotone, hemicontinuous and coercive, in the sense of Definition \ref{Def-Mono-Demi-Coer}, along with the results Theorem \ref{Thm-Max-Mono}, Corollary \ref{Cor-Mono+Hemi+Coe=Onto} and \cite[Theorem 1.4--1.6]{VB-93}, we prove the existence of a strong solution to the modified nonlinear heat equation \eqref{eqn-quant-heat}.

	\begin{theorem}\label{Thm-m-accretive}
		There exists a sufficiently large $\Gamma>0$, depending on $K$, such that $\Gn^K + \Gamma I$ is $m-$accretive in $L^2(\bO)$.
	\end{theorem}
	
	\begin{proof}[Proof of Theorem \ref{Thm-m-accretive}] Let us choose and fix $K\in\N$. We will show first that there exists a sufficiently large $\Gamma >0$ depending on $K$ such that 
		$\Gn^K + \Gamma I:D(A)\to L^2(\bO)$ is a monotone operator in the sense of Definition \ref{Def-Mono-Demi-Coer}.  
		
		%\coma{}
		\vskip 1mm
		\noindent
		\textbf{Step I.} First, let us choose and fix $u,v\in D(A)$. Then, by using \eqref{Def-Gn-K} and integration by parts, we find 
		\begin{align}
			& ((\Gn^K + \Gamma)u - (\Gn^K + \Gamma)v, u-v) =  (\Gn^K(u)-\Gn^K(v), u-v) + \Gamma  \norm{ u-v}_{L^2(\bO)}^2\\
			& =  \Gamma\norm{ u-v}_{L^2(\bO)}^2 + \|\nabla(u-v)\|_{L^2(\bO)}^2 + (|u|^{p-2}u -|v|^{p-2}v, u-v)  - ( \gn^K(u)- \gn^K(v), u-v)\\
			&\geq  \Gamma  \norm{ u-v}_{L^2(\bO)}^2 + \norm{\nabla( u-v)}_{L^2(\bO)}^2 + \frac{1}{2} \left[\norm{\abs{u}^{\frac{p-1}{2}} \abs{u-v}}_{L^2(\bO)}^2 + \norm{\abs{v}^{\frac{p-1}{2}} \abs{u-v}}_{L^2(\bO)}^2 \right]\\
			&\quad  - ( \gn^K(u)- \gn^K(v), u-v),\label{eqn-loc}
		\end{align}
		where, in the last step, we have used \eqref{eqn-mono-2}. 
		Now, from Lemma \ref{Lem-gn-mono}, we also have
		\begin{align}
			((\Gn^K + \Gamma)u - (\Gn^K + \Gamma)v, u-v)
			&\geq  \frac{1}{2}\norm{\nabla( u-v)}_{L^2(\bO)}^2 + (\Gamma  - C(K))\| u-v\|_{L^2(\bO)}^2\\
			& \quad + \frac{1}{2} \Big[\big\|\abs{u}^{\frac{p}{2}-1} \abs{u-v}\big\|_{L^2(\bO)}^2 + \big\|\abs{v}^{\frac{p}{2}-1} \abs{u-v}\big\|_{L^2(\bO)}^2 \Big]\\
			&\quad - \frac{1}{4} \Big[\big\|\abs{u}^{\frac{p}{2}-1} \abs{u-v}\big\|_{L^2(\bO)}^2 + \big\|\abs{v}^{\frac{p}{2}-1} \abs{u-v}\big\|_{L^2(\bO)}^2 \Big] \\
			& \ge %[\Gamma - C_1(K)]\|\nabla (u-v)\|_{L^2(\bO)}^2 + 
			[\Gamma - C(K)]\|u-v\|_{L^2(\bO)}^2.\label{eqn-acc-1}
		\end{align}
		Let us now choose $\Gamma$ such that  
		\begin{equation}\label{eqn-C(K)}
			\Gamma \ge C(K) = [1 + (2+ p^2 2^{2p-3})K\lambda_1^{-1}]K.
		\end{equation}
		Then, we deduce that the nonlinear operator $$\Gn^K + \Gamma I: D(A) \to L^2(\bO)$$ is monotone.
		\vskip 1mm
		\noindent
		\textbf{Step II.} We will  prove the hemicontinuity of the nonlinear operator $\Gn^K + \Gamma I$. For this purpose it is suffices to establish that $\Gn^K + \Gamma I$ is demicontinuous, see \cite[Theorem 1]{TK-64}. Let us show that $\Gn^K + \Gamma I$ is demicontinuous. Suppose that $\{\psi_k\}_{k\in \N} \subset D(A)$ is such that  $\psi_k \to \psi$ in $D(A)$ as $k\to \infty$. We will show that 
		$(\Gn^K + \Gamma I)\psi_k$ is weakly convergent in $L^2(\bO)$ to $(\Gn^K + \Gamma I)\psi$.

		By assumptions and the Sobolev embedding Theorem, since $p$ satisfies \eqref{eqn-D(A)-in-L^p-assump}, it follows that 
		\begin{equation}\label{eqn-D(A)-L^{2p-2}-H_0^1}
			\norm{A\psi_k -A \psi}_{L^2(\bO)} \to 0\ \text{ and }\ \norm{\psi_k - \psi}_{L^{2p-2}(\bO)} \to 0,\ \text{ as }\  k\to \infty.
		\end{equation}
		Thus, for any fixed $\eta \in D(A)$, we calculate
		\begin{align*}
			& ( (\Gn^K+\Gamma I)\psi_k - (\Gn^K+\Gamma I)\psi, \eta )\\
			& = (A (\psi_k - \psi), \eta) + (\abs{\psi_k}^{p-2}\psi_k - \abs{\psi}^{p-2}\psi, \eta) - (\gn^K(\psi_k) - \gn^K(\psi), \eta) + \Gamma (\psi_k - \psi), \eta)\\
			&  =: \mathcal{I}_1 + \mathcal{I}_2 + \mathcal{I}_3 + \mathcal{I}_4.
		\end{align*}
		We next estimate each $\mathcal{I}_i$ individually, for each $1\le i \le 4$, as follows:
		\begin{align*}
			\abs{\mathcal{I}_1} 
			& = \abs{(A(\psi_k - \psi), \eta )} \leq \norm{A\psi_k - A\psi}_{L^2(\bO)}  \norm{\eta}_{L^2(\bO)} \to 0\ \text{ as }\ k \to \infty.
		\end{align*}
		By applying Hölder's inequality twice (first with exponent $2$ and $2$, and then with $2p-2$ and $2p-2/p-2$), for any $2\le p < \infty$, we deduce 
		\begin{align*}
			\abs{\mathcal{I}_2} 
			& = \abs{ (\abs{\psi_k}^{p-2}\psi_k - \abs{\psi}^{p-2}\psi, \eta) } \\
			& \leq (p-1)\norm{\psi_k - \psi}_{L^{2p-2}(\bO)} \big(\norm{\psi_k}_{L^{2p-2}(\bO)} + \norm{\psi}_{L^{2p-2}(\bO)}\big)^{p-2} \norm{\eta}_{L^2(\bO)}\\
			& \to 0\ \text{ as }\ k \to \infty,
		\end{align*}
		where we have used the convergences \eqref{eqn-D(A)-L^{2p-2}-H_0^1}. Similarly, we also have
		\begin{align*}
			\abs{\mathcal{I}_4} \le \Gamma \norm{\psi_k - \psi}_{L^2(\bO)} \norm{\eta}_{L^2(\bO)} \to 0, \ \text{ as }\ k \to \infty.
		\end{align*}
		Lastly, from Lemma \ref{gn-demi}, we infer 
		\begin{align*}
			\mathcal{I}_3 \to 0, \ \text{ as }\ k \to \infty.
		\end{align*} 
		Therefore, the map $\Gn^K + \Gamma I: D(A) \to L^2(\bO)$ is demicontinuous and hence hemicontinuous.
		
		%Next, we show that $(I+A)(\Gn^K + \Gamma I)$ is coercive, which drive us to prove the \textit{$m-$accretivity} property of the operator $\Gn^K + \Gamma I$.
		
		\vskip 1mm
		\noindent
		\textbf{Step III.} We now show that 
		there exists $\Gamma_K>0$ such that for every $\Gamma > \Gamma_K$, the map $\Gn^K + \Gamma I:D(A) \to L^2(\bO)$ is coercive. Let us prove it in two different cases: 
		\vskip 1mm
		\noindent
		\textit{Case I.} Suppose that  $\norm{\nabla\psi}_{L^2(\bO)}^2 +  \norm{\psi}_{L^p(\bO)}^p \leq K$.
		\vskip 1mm
		\noindent
		By the definition of $\Gn^K$ given in \eqref{Def-Gn-K}, we find 
		\begin{align*}
			((\Gn^K +\Gamma I)\psi, \psi)
			& =  \norm{\nabla \psi}_{L^2(\bO)}^2 + \norm{\psi}_{L^p(\bO)}^p + [\Gamma - \norm{\nabla \psi}_{L^2(\bO)}^2 - \norm{\psi}_{L^p(\bO)}^p]  \norm{\psi}_{L^2(\bO)}^2\\
			& \ge  [\Gamma - K] \norm{\psi}_{L^2(\bO)}^2.
		\end{align*}
		Dividing on both sides by $\norm{\psi}_{L^2(\bO)}$ and letting $\norm{\psi}_{L^2(\bO)} \to \infty$ yield
		\begin{align*}
			\frac{((\Gn^K +\Gamma I)\psi, \psi) }{\norm{\psi}_{L^2(\bO)}} 
			& \ge [\Gamma - K] \norm{\psi}_{L^2(\bO)} \to \infty,
		\end{align*}
		for all $\Gamma > \Gamma_K = K$.
		
		\vskip 1mm
		\noindent
		\textit{Case II.} When $\norm{\nabla\psi}_{L^2(\bO)}^2 + \norm{\psi}_{L^p(\bO)}^p > K$.
		\vskip 1mm
		\noindent
		Again, by utilizing the definition of $\Gn^K$, we calculate 
		\begin{align*}
			((\Gn^K +\Gamma I)\psi, \psi)
			& =  \norm{\nabla \psi}_{L^2(\bO)}^2 + \norm{\psi}_{L^p(\bO)}^p + \bigg[\Gamma - \frac{K^{2}}{\norm{\nabla \psi}_{L^2(\bO)}^2 + \norm{\psi}_{L^p(\bO)}^p}\bigg] \norm{\psi}_{L^2(\bO)}^2\\
			& \ge  \big[\Gamma - K\big]\norm{\psi}_{L^2(\bO)}^2.
		\end{align*}
		Dividing on both sides by $\norm{\psi}_{L^2(\bO)}$ and letting $\norm{\psi}_{L^2(\bO)} \to \infty$ yield
		\begin{align*}
			\frac{((\Gn^K +\Gamma I)\psi, \psi) }{\norm{\psi}_{L^2(\bO)}} 
			& \ge [\Gamma - K] \norm{\psi}_{L^2(\bO)} \to \infty,
		\end{align*}
		for sufficiently large $\Gamma$ such that $\Gamma > \Gamma_K = K$.
		Thus, it shows that $\Gn^K + \Gamma I$ is coercive from $D(A)$ to $L^2(\bO)$, for $\Gamma > \Gamma_K = K$. 
		
		Observe from \eqref{eqn-C(K)} that $C(K) > K$, thus, it is enough to consider $\Gamma \ge C(K)$ to get that the map $\Gn^K + \Gamma I$ is monotone, hemicontinuous and coercive.

		\vskip 1mm
		\noindent
		\textbf{Step IV.} In the previous three steps, we have shown that the nonlinear operator $\Fn = \Gn^K + \Gamma I:D(A) \to L^2(\bO)$ satisfies the assumptions of Corollary \ref{Cor-Mono+Hemi+Coe=Onto}, i.e., it is a monotone, hemicontinuous, and coercive map for $\Gamma \ge C(K)$, see \eqref{eqn-C(K)}. Therefore, we infer that $\Gn^K + \Gamma I$ is surjective, i.e., $R(\Gn^K + \Gamma I) = L^2(\bO)$. 
		Now, let us define 
		\begin{equation}\label{Def-Kappa}
			\kappa(\psi):= A(\psi) + \Nn(\psi) - \gn^K(\psi)  + \Gamma \psi,
		\end{equation}
		where $D(\kappa):= \{\psi\in L^{2p-2}(\bO)\cap H_0^1(\bO) : A(\psi) + \Nn(\psi) - \gn^K(\psi)  \in L^2(\bO)\}$.
		
		Let us choose $\mathcal{X} = L^2(\bO) \cong \mathcal{X}^\ast$ in Theorem \ref{Thm-Max-Mono}. This implies that the map $\kappa = \Gn^K + \Gamma I$ is \textit{maximal-monotone }and hence \textit{$m-$accretive} (see Remark \ref{Rmk-max+mono-m-acc}), with domain $D(\kappa) \supseteq D(A)$, for 
		\begin{equation}\label{eqn-Gamma}
			\Gamma \ge \Gamma_K:= \max\{C(K), K\} = C(K),
		\end{equation}
		where $C(K)$ is appearing while proving the monotonicity property, see \eqref{eqn-C(K)}.
		Let us now conclude the proof by showing that the map $\Gn^K + \Gamma I$ is \textit{$m-$accretive} in $L^2(\bO)$ with the domain $D(A)$. 
		Fix $0 \neq\psi \in D(\kappa)$, $K\in\N$, and take the $L^2-$inner product of \eqref{Def-Kappa} with $A\psi$ to deduce
		\begin{align*}
			& \norm{A\psi}_{L^2(\bO)}^2  + \big[\Gamma - K\big] \norm{\nabla\psi}_{L^2(\bO)}^2\\
			&\le \norm{A\psi}_{L^2(\bO)}^2 + (p-1)\norm{\abs{\psi}^{\frac{p-2}{2}} \nabla \psi}_{L^2(\bO)}^2 + \Gamma\norm{\nabla\psi}_{L^2(\bO)}^2 - (\gn^K(\psi), A\psi)\\
			& = ((\Gn^K +\Gamma I)\psi, A\psi) = (\kappa(\psi), A\psi)  \leq \norm{\kappa(\psi)}_{L^2(\bO)}\norm{A\psi}_{L^2(\bO)},
		\end{align*}
		where we have used Step III. Thus, for sufficiently large $\Gamma \ge \Gamma_K$, see \eqref{eqn-Gamma}, we deduce
		\begin{align*}
			\norm{A\psi}_{L^2(\bO)} \leq \norm{\kappa(\psi)}_{L^2(\bO)}, \ \text{ which  implies }\ D(\kappa) \subseteq D(A).
		\end{align*}
		Thus, $D(\kappa) = D(A)$, and hence the map $\kappa = \Gn^K +\Gamma I$ is \textit{$m-$accretive} on $L^2(\bO)$ with domain $D(A)$, for sufficiently large $\Gamma \ge \Gamma_K$, see \eqref{eqn-Gamma}. It completes the proof.
	\end{proof}
	
	From Theorem \ref{Thm-m-accretive} and \cite[Theorem 1.8, Chapter IV]{VB-93}, we have the following  immediate result:
	
	\begin{proposition}\label{Prop-quant}
		Let $T>0$ and $u_0\in D(A)\cap \bM$ be fixed. Then, there exists a unique strong solution \begin{align}\label{eqn-regularity}u_K\in W^{1,\infty}([0,T];L^2(\bO))\cap L^{\infty}(0,T; D(A))\end{align}
		which solves (in $(0,T)\times \bO$)
		\begin{align}\label{eqn-quant-heat}
			\left\{
			\begin{aligned}
				\frac{d u_K(t)}{d t} + A u_K(t) + |u_K(t)|^{p-2}u_K(t) - \gn^K(u_K(t)) & = 0,\\
				u_K(0) & = u_0,\\
				u_K(t)|_{\partial \bO} & = 0,
			\end{aligned}
			\right.
		\end{align} 
		in $L^2(\bO)$. 	Moreover, $u_K(t)\in\bM$, for all $t>0$, only when $\norm{\nabla u_K(t)}_{L^2(\bO)}^2 + \norm{u_K(t)}_{L^p(\bO)}^p \le K$, otherwise, $L^2-$norm of the solution dissipates in time, i.e., $\norm{u_K(t)}_{L^2(\bO)} < 1,$ for all $t>0$ and the right derivative of $u_K$, denoted by $ \frac{d^+u_K(t)}{dt}$, exists for all $t\in [0,T)$.
	\end{proposition}

	\subsection{Proof of Theorem \ref{Thm-Main-Exis}: Part I}\label{Subsec-Thm-I}
	In this subsection, we present the first part of the proof of Theorem \ref{Thm-Main-Exis}. By utilizing the previously obtained results, we derive energy estimates for the modified system \eqref{eqn-quant-heat} and establish the existence and uniqueness of a strong solution in the class $u\in W^{1,\infty}([0,T];L^2(\bO))\cap L^{\infty}(0,T; D(A))$ for the original problem \eqref{eqn-main-heat}.
	
	%\begin{proof}[Proof of Theorem \ref{Thm-Main-Exis}]
	Let us begin by choosing and fixing $K\in\N$, and $u_0\in D(A)$. Suppose $u_K$ is the unique strong solution to the modified problem \eqref{eqn-quant-heat} (guaranteed by Proposition \ref{Prop-quant}). Then, we calculate the energy estimates in the following two cases:
	\vskip 1mm
	\noindent
	\textit{Case I.} When $\norm{\nabla u_K}_{L^2(\bO)}^2 + \norm{u_K}_{L^p(\bO)}^p \le K$.
	\vskip 1mm
	\noindent
	\textbf{(a)} Using the regularity given in \eqref{eqn-regularity}, we infer from Lions-Magenes Lemma \cite{JLL+EM-II-72}
	that the mapping $[0,T]\ni t\mapsto\|u_K(t)\|_{L^2(\bO)}^2\in\mathbb{R}$ is absolutely continuous. 
	Therefore, taking the $L^2-$inner product of the equation \eqref{eqn-quant-heat} with $u_K$ and integrating by parts, we have  for a.e. $t\in[0,T]$ 
	\begin{align*}
		\frac{d}{dt}\left(\norm{u_K(t)}_{L^2(\bO)}^2 -1 \right) 
		& = 2\left( u_K(t), \frac{\partial u_K(t)}{\partial t} \right)
		= 2 \left( u_K(t), \Gn^K(u_K(t))\right)\\
		& = -2 \norm{\nabla u_K(t)}_{L^2(\bO)}^2 -2 \norm{ u_K(t)}_{L^p(\bO)}^p \\
		& \quad + 2\left(\norm{\nabla u_K(t)}_{L^2(\bO)}^2 + \norm{ u_K(t)}_{L^p(\bO)}^p \right) \norm{ u_K(t)}_{L^2(\bO)}^2\\
		& = 2\left(\norm{\nabla u_K(t)}_{L^2(\bO)}^2 + \norm{ u_K(t)}_{L^p(\bO)}^p \right)( \norm{ u_K(t)}_{L^2(\bO)}^2 -1).
	\end{align*}
	Let us denote $\theta(t) = \norm{u_K(t)}_{L^2(\bO)}^2 -1$. Thus, the above equality transforms into
	\begin{align*}
		\frac{d\theta(t)}{dt} = 2\left(\norm{\nabla u_K(t)}_{L^2(\bO)}^2 + \norm{ u_K(t)}_{L^p(\bO)}^p\right)\theta(t).
	\end{align*}
	Applying the variation of constant formula, we deduce
	\begin{align*}
		\theta(t) = \theta(0) \exp\left[2\int_0^t \left(\norm{\nabla u_K(s)}_{L^2(\bO)}^2 + \norm{ u_K(s)}_{L^p(\bO)}^p\right) ds\right].
	\end{align*}
	Since $u_K \in L^\infty(0,T; D(A)) \embed L^\infty(0,T; L^p(\bO)\cap H_0^1(\bO))$, for any $p$ satisfying \eqref{eqn-D(A)-in-L^p-assump}, and because $$\theta(0) = \norm{u_K(0)}_{L^2(\bO)}^2 -1= \norm{u_0}_{L^2(\bO)}^2 - 1=0, $$ we immediately have $\norm{u_K(t)}_{L^2(\bO)}^2 -1 = 0,$ for all  $t\in [0,T].$ Thus, it shows that
	\begin{align}\label{eqn-||um||=1}
		u_K(t)\in \bM ,\ \text{ for all }\ t\in [0,T],
	\end{align}
	Moreover, the fact $\norm{u_K(t)}_{L^2(\bO)} = 1$, for all $t\in[0,T]$, implies
	\begin{align}\label{eqn-der}
		\frac{d}{dt} \norm{u_K(t)}_{L^2(\bO)}^2 = \left(u_K(t), \frac{\partial u_K(t)}{\partial t}\right) = 0,\ \text{for a.e.}\ t\in[0,T].
	\end{align}
	\vskip 1mm
	\noindent
	\textbf{(b)} Once again  using the regularity given in \eqref{eqn-regularity}, we infer from Lions-Magenes Lemma that the mapping $[0,T]\ni t\mapsto\|\nabla u_K(t)\|_{L^2(\bO)}^2\in\mathbb{R}$ is absolutely continuous. Hence, taking the $L^2-$inner product of equation \eqref{eqn-quant-heat} with $Au_K$ and applying Lemma \ref{AbsC}, we obtain for a.e. $t\in[0,T]$ that
	\begin{align*}
		\frac{1}{2}\frac{d}{dt}\norm{\nabla u_K(t)}_{L^2(\bO)}^2 
		& {= \left(\frac{\partial u_K(t)}{\partial t}, Au_K(t) \right)} \\
		&= \left(\frac{\partial u_K(t)}{\partial t}, -\frac{\partial u_K(t)}{\partial t} - |u_K(t)|^{p-2}u_K(t) + \gn^K(u_K(t)) \right)\\
		& = - \norm{\frac{\partial u_K(t)}{\partial t}}_{L^2(\bO)}^2 - \left( \frac{\partial u_K(t)}{\partial t} , |u_K(t)|^{p-2}u_K(t)\right) \\
		& \quad + \left(\norm{\nabla u_K(t)}_{L^2(\bO)}^2  + \norm{u_K(t)}_{L^p(\bO)}^p\right) \left( \frac{\partial u_K(t)}{\partial t},u_K(t)\right).
	\end{align*}
	Thus, by using the equation \eqref{eqn-der},  we assert
	\begin{align}
		\frac{1}{2}\frac{d}{dt}\norm{\nabla u_K(t)}_{L^2(\bO)}^2 + \norm{\frac{\partial u_K(t)}{\partial t}}_{L^2(\bO)}^2 +\frac{1}{p}\frac{d}{dt}\norm{u_K(t)}^{p}_{L^p(\bO)} = 0.\label{eqn-first-estimates}
	\end{align}
	Integrating the above equality over the time interval $[0,t]$ yields
	\begin{align}
		\frac{1}{2}\norm{\nabla u_K(t)}_{L^2(\bO)}^2 & + \int_0^t\norm{\frac{\partial u_K(s)}{\partial s}}_{L^2(\bO)}^2 ds + \frac{1}{p}\norm{u_K(t)}^{p}_{L^p(\bO)}\\
		& = \frac{1}{2}\norm{\nabla u_0}_{L^2(\bO)}^2
		+ \frac{1}{p}\norm{u_0}^{p}_{L^p(\bO)}.\label{eqn-energy-equality}
	\end{align}
	
	\vskip 1mm
	\noindent
	\textit{Case II.} When $\norm{\nabla u_K}_{L^2(\bO)}^2 + \norm{u_K}_{L^p(\bO)}^p > K$.
	\vskip 1mm
	\noindent
	\textbf{(a)} Similar to Case I, we take the $L^2-$inner product of equation \eqref{eqn-quant-heat} with $u_K$ and perform  integrating by parts for a.e. $t\in[0,T]$ to find
	\begin{align}
		\frac{d}{dt}\left(\norm{u_K(t)}_{L^2(\bO)}^2 -1 \right) 
		& = 2\left( u_K(t), \frac{\partial u_K(t)}{\partial t} \right)
		= 2 \left( u_K(t), \Gn^K(u_K(t))\right)\\
		& = -2 \norm{\nabla u_K(t)}_{L^2(\bO)}^2 -2 \norm{ u_K(t)}_{L^p(\bO)}^p \\
		& \quad + 2\left(\frac{K^2}{\norm{\nabla u_K(t)}_{L^2(\bO)}^2 + \norm{ u_K(t)}_{L^p(\bO)}^p} \right) \norm{ u_K(t)}_{L^2(\bO)}^2\\
		& < 2\left(\norm{\nabla u_K(t)}_{L^2(\bO)}^2 + \norm{ u_K(t)}_{L^p(\bO)}^p \right)( \norm{ u_K(t)}_{L^2(\bO)}^2 -1).\label{eqn-um<1}
	\end{align}
	We denote $\theta(t) = \norm{u_K(t)}_{L^2(\bO)}^2 -1$. Consequently, the equation above can be rewritten as
	\begin{align*}
		\frac{d\theta(t)}{dt} < 2\left(\norm{\nabla u_K(t)}_{L^2(\bO)}^2 + \norm{ u_K(t)}_{L^p(\bO)}^p\right)\theta(t).
	\end{align*}
	Applying the variation of constant formula, we deduce
	\begin{align*}
		\theta(t) < \theta(0) \exp\left[2\int_0^t \left(\norm{\nabla u_K(s)}_{L^2(\bO)}^2 + \norm{ u_K(s)}_{L^p(\bO)}^p\right) ds\right].
	\end{align*}
	Since $u_K \in L^\infty(0,T; D(A)) \embed L^\infty(0,T; L^p(\bO)\cap H_0^1(\bO))$ and $\theta(0) = \norm{u_K(0)}_{L^2(\bO)}^2 -1= \norm{u_0}_{L^2(\bO)}^2 - 1=0$, we immediately have $\norm{u_K(t)}_{L^2(\bO)} < 1,$ for all  $t\in [0,T].$ Thus, it shows that
	\begin{align}\label{eqn-||um||<1-3}
		u_K(t)\notin \bM ,\ \text{ for all }\ t\in [0,T].
	\end{align}
	Moreover, using the fact that $\norm{u_K(t)}_{L^2(\bO)} < 1$, for all $t\in[0,T]$, in \eqref{eqn-um<1} implies
	\begin{align}\label{eqn-der>>}
		\frac{d}{dt} \norm{u_K(t)}_{L^2(\bO)}^2 = \left(u_K(t), \frac{\partial u_K(t)}{\partial t}\right) < 0,\ \text{for a.e.}\ t\in[0,T].
	\end{align}
	\vskip 1mm
	\noindent
	\textbf{(b)} Taking the $L^2-$inner product of the equation \eqref{eqn-quant-heat} with $Au_K$, utilizing Lemma \ref{AbsC} and integrating by parts yield
	\begin{align*}
		\frac{1}{2}\frac{d}{dt}\norm{\nabla u_K(t)}_{L^2(\bO)}^2 
		&= \left(\frac{\partial u_K(t)}{\partial t}, Au_K(t) \right) \\
		&= \left(\frac{\partial u_K(t)}{\partial t}, -\frac{\partial u_K(t)}{\partial t} - |u_K(t)|^{p-2}u_K(t) + \gn^K(u_K(t)) \right)\\
		& = - \norm{\frac{\partial u_K(t)}{\partial t}}_{L^2(\bO)}^2 - \left( \frac{\partial u_K(t)}{\partial t} , |u_K(t)|^{p-2}u_K(t)\right) \\
		& \quad + \left(\frac{K^2}{\norm{\nabla u_K(t)}_{L^2(\bO)}^2 + \norm{u_K(t)}_{L^p(\bO)}^p}\right) \left( \frac{\partial u_K(t)}{\partial t},u_K(t)\right).
	\end{align*}
	Thus, by using the equation \eqref{eqn-der>>},  we obtain
	\begin{align}
		\frac{1}{2}\frac{d}{dt}\norm{\nabla u_K(t)}_{L^2(\bO)}^2 + \norm{\frac{\partial u_K(t)}{\partial t}}_{L^2(\bO)}^2 +\frac{1}{p}\frac{d}{dt}\norm{u_K(t)}^{p}_{L^p(\bO)} < 0.
	\end{align}
	Now, by integrating the above equality with respect to time from $0$ to $t$, we get
	\begin{align}
		\frac{1}{2}\norm{\nabla u_K(t)}_{L^2(\bO)}^2 & + \int_0^t\norm{\frac{\partial u_K(s)}{\partial s}}_{L^2(\bO)}^2 ds + \frac{1}{p}\norm{u_K(t)}^{p}_{L^p(\bO)} \\
		& < \frac{1}{2}\norm{\nabla u_0}_{L^2(\bO)}^2
		+ \frac{1}{p}\norm{u_0}^{p}_{L^p(\bO)}.\label{eqn-u_K-H_0^1-L^p}
	\end{align}
	In particular, combining the above two cases, we have 
	\begin{align}\label{est-Lp-cap-H_0^1}
		\frac{1}{2}\norm{\nabla u_K(t)}_{L^2(\bO)}^2	+ \frac{1}{p}\norm{u_K(t)}^{p}_{L^p(\bO)} 
		& \le C, \ \text{ for all }\  t\in [0,T],
	\end{align}
	where $C= \frac{1}{2}\norm{\nabla u_0}_{L^2(\bO)}^2
	+ \frac{1}{p}\norm{u_0}^{p}_{L^p(\bO)}$ is independent of $K$. Hence, by choosing $K \ge C$ and redefining $u_K = u$, we deduce
	\begin{align*}
		\gn^K(u) = (\norm{\nabla u}_{L^2(\bO)}^2 + \norm{u}_{L^p(\bO)}^p) u,
	\end{align*}
	so that $u$ is a strong solution to the original problem \eqref{eqn-main-heat} such that 
	\begin{align*}
		u\in W^{1,\infty}([0,T];L^2(\bO))\cap L^{\infty}(0,T; D(A)),
	\end{align*}
	satisfying \eqref{eqn-dec-ener} (see \eqref{eqn-energy-equality}). Moreover, the uniqueness of the strong solution $u$ obtained above follows from \cite[Section 4.4]{AB+ZB+MTM-25+}. 
	
	The proof of Theorem \ref{Thm-Main-Exis} will be completed once we establish that
	\begin{align*}
		u \in L^2(0,T; D(A^{\frac{3}{2}}))\cap C([0,T]; D(A))\ \text{ and }\ \frac{\partial u}{\partial t} \in L^2(0,T; H_0^1(\bO)),
	\end{align*}
	This verification is deferred to the next subsection; see Remark \ref{Rmk-Regularity}.

	\begin{remark}
	Similar to the Case II (b), see \eqref{eqn-u_K-H_0^1-L^p}), for the solution $u$ to the problem \eqref{eqn-main-heat}, we also have
	\begin{align}
		&\sup_{t\in[0,T]}\Big[\frac{1}{2}\norm{\nabla u(t)}_{L^2(\bO)}^2+ \frac{1}{p}\norm{u(t)}^{p}_{L^p(\bO)}\Big] + \int_0^T \Big\|\frac{du(s)}{ds}\Big\|_{L^2(\bO)}^2 ds  \\
		& \leq  \frac{1}{2}\norm{\nabla u_0}_{L^2(\bO)}^2
		+ \frac{1}{p}\norm{u_0}^{p}_{L^p(\bO)}.\label{eqn-H_0^1-L^p}
	\end{align}
\end{remark}

	\subsection{Proof of Theorem \ref{Thm-Main-Exis}: Part II}\label{Subsec-Thm-II}
	This subsection aims to establish the second part of the proof of Theorem \ref{Thm-Main-Exis}, namely the regularity results for the strong solution obtained in the previous subsection. We first prove a regularity result for the solution to problem \eqref{eqn-main-heat}, motivated by \cite[Proposition 4.6]{AB+ZB+MTM-25+}. Then, by applying the well-known Yosida approximation, a regularization technique relevant to our main problem \eqref{eqn-main-heat}, we derive two additional regularity results. Let us fix $T\in(0,\infty)$ and recall that $\bO\subset \R^d$ is a Poincar\'e domain with $C^2-$boundary.

	\subsubsection{A regularity result}
	We choose and fix  $p$ as in \eqref{eqn-D(A)-in-L^p-assump}. First, we state a regularity result for the unique strong solution of problem \eqref{eqn-main-heat}, which will serve as a foundation for deriving additional regularity properties of the solution.
	
	\begin{proposition}\label{Prop-u-L2-D(A)}
		Suppose $u_0\in D(A)\cap\bM$. Then, the unique strong solution $u$ to the problem \eqref{eqn-main-heat}, as established in Subsection \ref{Subsec-Thm-I}, satisfies the following estimate:
		\begin{align}
			&\int_0^T \Big(\norm{\Delta u(t)}_{L^2(\bO)}^2 + 2(p-1)\|\abs{u(t)}^{\frac{p-2}{2}}\nabla u(t)\|_{L^2(\bO)}^2 + \norm{u(t)}_{L^{2p-2}(\bO)}^{2p-2} \Big) dt\\
			& \leq  C(T, \norm{u_0}_{L^p(\bO) \cap H_0^1(\bO)}).\label{eqn-est-L^2p-2}
		\end{align}
	\end{proposition}
	
	\begin{proof}
		Let us fix $T>0$ and choose $u_0\in D(A)$. Then, from first part of Theorem \ref{Thm-Main-Exis}, see Subsection \ref{Subsec-Thm-I}, there exists a unique strong solution $u$ of the problem \eqref{eqn-main-heat} such that $u \in L^{\infty}(0,T; D(A))$ and $\frac{\partial u}{\partial t} \in L^\infty(0,T; L^2(\bO))$. Since $u \in L^\infty(0,T; D(A))$ and $\bO$ is a Poincar\'e domain, we have $D(A)=H^2(\bO)\cap H_0^1(\bO)$ and an application of the Sobolev inequality yields $D(A) \embed L^p(\bO) \embed (D(A))^\ast$ and $\frac{\partial u}{\partial t}\in L^2(0,T; L^2(\bO))$, it follows by the Lions-Magenes Lemma that, the map $$[0,T]\ni t\mapsto\|u(t)\|_{L^p(\bO)}^p\in\mathbb{R}$$ is absolutely continuous. Thus, along with the absolute continuity of the function $[0,T]\ni t\mapsto\|\nabla u(t)\|_{L^2(\bO)}^2\in\mathbb{R}$, for a.e. $t\in[0,T]$, we deduce
		\begin{align}
			\frac{d}{dt}\bigg[\frac{\norm{\nabla u(t)}_{L^2(\bO)}^2}{2} + \frac{\norm{u(t)}_{L^p(\bO)}^p}{p}\bigg]
			& = \left(\frac{\partial u(t)}{\partial t}, Au(t) + \abs{u(t)}^{p-2}u(t)\right)\\
			& = (-\Gn(u(t)), Au(t) + \abs{u(t)}^{p-2}u(t)).	\label{eqn-Exis-Apriori-1}
		\end{align}
		Performing integration by parts yields
		\begin{align}
			\left( \abs{u}^{p-2}u , -\Delta u \right) 
			& = \int_\bO \nabla((\abs{u(x)}^2)^{\frac{p-2}{2}}u(x))\cdot \nabla u(x)  dx\\
			&= \int_\bO \bigg[ \frac{p-2}{2}(2 u(x) \nabla u(x))((\abs{u(x)}^2)^{\frac{p-4}{2}}u(x)) + \abs{u(x)}^{p-2}\nabla u(x) \bigg] \cdot\nabla u(x) dx\\
			&= \int_\bO \left[ (p-2) \abs{u(x)}^{p-2} \abs{\nabla u(x)}^2 + \abs{u(x)}^{p-2}\abs{\nabla u(x)}^2 \right] dx\\
			&= (p-1)\|\abs{u}^{\frac{p-2}{2}} \nabla u\|_{L^2(\bO)}^2. \label{eqn-Lap-Nn}
		\end{align}
		By applying integration by parts, we conclude 
		\begin{align*}
			\left(\Gn(u), -\Delta u + \abs{u}^{p-2}u\right)
			& = (\Delta u - |u|^{p-2}u + ( \norm{\nabla u}_{L^2(\bO)}^2 + \norm{u}_{L^p(\bO)}^p) u, -\Delta u + \abs{u}^{p-2}u)\\
			& = - \| A u + |u|^{p-2}u\|_{L^2(\bO)}^{2}  + \big(\norm{\nabla u}_{L^2(\bO)}^2 + \norm{u}_{L^p(\bO)}^p \big)^2.
		\end{align*}
		By inserting the previously obtained identity into \eqref{eqn-Exis-Apriori-1}, we infer
		\begin{align}
			\frac{d}{dt}\bigg[\frac{\norm{\nabla u(t)}_{L^2(\bO)}^2}{2} + \frac{\norm{u(t)}_{L^p(\bO)}^p}{p}\bigg]
			& + \| A u(t) + |u(t)|^{p-2}u(t)\|_{L^2(\bO)}^{2}\\ 
			& = \big(\norm{\nabla u(t)}_{L^2(\bO)}^2 + \norm{u(t)}_{L^p(\bO)}^p \big)^2.
		\end{align}
		Now, integrating with respect to time over the interval $[0,T]$ and applying \eqref{eqn-H_0^1-L^p}, we obtain
		\begin{align}
			\int_0^T \| A u(t) + |u(t)|^{p-2}u(t) \|_{L^2(\bO)}^{2} dt
			& \leq C\big(T, \norm{u_0}_{L^p(\bO) \cap H_0^1(\bO)}\big).
		\end{align}
		In particular, we get
		\begin{align}
			&\int_0^T\| A u(t) + |u(t)|^{p-2}u(t)\|_{L^2(\bO)}^{2} dt \\
			\nonumber& = \int_0^T \Big[\norm{\Delta u(t)}_{L^2(\bO)}^2 + 2(p-1)\|\abs{u(t)}^{\frac{p-2}{2}}\nabla u(t)\|_{L^2(\bO)}^2 + \norm{u(t)}_{L^{2p-2}(\bO)}^{2p-2} \Big] dt\\
			& \leq  C \big(T, \norm{u_0}_{L^p(\bO) \cap H_0^1(\bO)} \big).\label{eqn-D(A)-est}
		\end{align}
		If we combine this with \eqref{eqn-Lap-Nn}, then it follows that $Au, \abs{u}^{\frac{p-2}{2}}\nabla u \in L^2(0,T; L^2(\bO))$ and $u\in L^{2p-2}(0,T; L^{2p-2}(\bO))$. Hence it completes the proof.
	\end{proof}

	\subsubsection{Yosida approximation and regularity results}\label{Subsec-Yosida}
	We now move to some more useful results, which will be obtained with the help of the well-known Yosida approximation technique.
	
	Let us choose and fix $u_0\in D(A)\cap\bM$. Suppose $u$ is the unique strong solution to the problem \eqref{eqn-main-heat} such that 
	$$u\in W^{1,\infty}([0,T];L^2(\bO))\cap L^{\infty}(0,T; D(A)),$$
	guaranteed by Theorem \ref{Thm-Main-Exis}. Then, throughout this subsection, for the solution $u$, we consider the following Yosida approximation motivated from \cite[ p. 28]{ZB+WL+JZ-21},
	\begin{align}\label{Def-u_mu}
		u_\mu = \Jn(u) := \mu(\mu I + A)^{-1} u,
	\end{align}
	where, for every $\mu>0$, the operator $\Jn$ is a self-adjoint operator which commutes with the Laplace operator $A$, i.e., $\Jn A = A \Jn$.
	
	Next, we consider the following Cauchy problem in $(0,T)\times\bO$, i.e., a Yosida approximation problem corresponding to the problem \eqref{eqn-main-heat}:
	\begin{equation}\label{eqn-Yosida-heat}
		\left\{\begin{aligned}
			\frac{\partial u_\mu(t)}{\partial t} + \Gn_\mu(u_\mu(t)) & = 0,\\
			u_\mu(0) & = \Jn u_0, \\
			u_\mu(t)|_{\partial\bO} & = 0,
		\end{aligned}\right.
	\end{equation}
	where $\Gn_\mu(u_\mu) := A u_\mu + \Jn(|u_\mu|^{p-2}u_\mu) - ( \norm{\nabla u_\mu}_{L^2(\bO)}^2 + \norm{u_\mu}_{L^p(\bO)}^p ) u_\mu$. Note that $ \Jn u_0\to u_0$ in $D(A)$ as $\mu\to\infty$, whenever $u_0\in D(A)$ (see \eqref{eqn-conv} below). 
	
	Observe from \eqref{Def-u_mu} that $\Jn$ consists the inverse of a self-adjoint operator $(\mu I + A)$. Therefore, it is intriguing to expect a regularizing effect. Thus, by utilizing this property of $\Jn$, we derive some useful energy estimates corresponding to the above-mentioned Yosida approximated system \eqref{eqn-Yosida-heat}. 
	
	%We begin with a useful remark on spectral theory and functional calculus, that will be used to prove a convergence result which guarantees the Yosida-approximated solution converges in the $D(A)-$norm to the solution of the original system \eqref{eqn-main-heat}.
	
	We begin with a useful remark on resolvent identity, that will be used to prove a convergence result which guarantees the Yosida-approximated solution converges in the $D(A)-$norm to the solution of the original system \eqref{eqn-main-heat}.

	\begin{remark}\label{Rmk-Au_mu-est}
		%\color{blue}{
			Let us suppose that $u\in D(A)$. Then, for $\mu>0$, we have
			\begin{equation*}
				\| A u_\mu - A u\|_{L^2(\bO)} \le \| \mu (\mu I + A)^{-1} - I \|_{\Ls(L^2(\bO))} \| A u \|_{L^2(\bO)},
			\end{equation*}
			where $\|\cdot\|_{\Ls(L^2(\bO))}$ denotes the operator norm. We claim that 
			\begin{equation}\label{eqn-op-norm-le-1} 
				\| \mu (\mu I + A)^{-1} - I \|_{\Ls(L^2(\bO))} \le 1.
			\end{equation}
			Let us recall that $A: D(A)\subset L^2(\bO) \to L^2(\bO)$ is self-adjoint and positive, so is $\mu I + A$ for $\mu>0$. As a result, for $0\neq u\in D(A)$, after the integration by parts, we have
			\begin{align*}
				 \|(\mu I + A)u\|_{L^2(\bO)} \| u\|_{L^2(\bO)} \ge ((\mu I + A)u, u ) =	\| \nabla u \|^2_{L^2(\bO)} + \mu \| u\|^2_{L^2(\bO)} \ge \mu \| u\|_{L^2(\bO)}^2.
			\end{align*}
			Therefore, we deduce
			\begin{equation}\label{eqn-oper-bound}
				\| (\mu I + A)^{-1} \|_{\Ls(L^2(\bO))} \le \frac1\mu, \ \text{ for }\ \mu>0. 
			\end{equation}
			We recall from \cite[Proposition 6.9]{HB-11} that if $T:H\to H$ is a self-adjoint bounded linear operator, then 
			$$
			\| T \|_{\Ls(L^2(\bO))} =\max\{ |M|, |m|\},
			$$
			where 
			$$
			M = \sup_{\| u \|_{H} =1} (Tu, u), \qquad m =\inf_{\| u \|_{H} =1} (Tu, u).
			$$
			Let us choose $T = I - \mu (\mu I + A)^{-1}$ and fix $H=L^2(\bO)$ in the above settings. Then, we obtain 
			$$
			(Tu, u) = \| u\|_{L^2(\bO)}^2 - \mu ((\mu I + A)^{-1} u, u) \le \| u\|_{L^2(\bO)}^2.
			$$
			Hence, by taking supremum over $\|u\|_{L^2(\bO)} =1$, we get $M \le 1$. Moreover, due to \eqref{eqn-oper-bound}, we see
			$$
			(Tu, u)\ge  \| u\|_{L^2(\bO)}^2 - \mu \frac 1 \mu \| u\|_{L^2(\bO)}^2=0,
			$$
			thus, taking infimum over $\|u\|_{L^2(\bO)} =1$ yields $m=0$. Combining these estimates yields $\|T\|_{\Ls(L^2(\bO))} \le 1$, hence \eqref{eqn-op-norm-le-1} follows as well as
			\begin{equation}\label{eqn-A_mu-A}
				\| A u_\mu - A u\|_{L^2(\bO)} \le \|Au\|_{L^2(\bO)}.
			\end{equation}
			%}
	\end{remark}
	
	\begin{remark}
		Let us emphasize that the resolvent-based strategy can be avoided by employing spectral measures, and making use of spectral theory and functional calculus, since the operator $A$ is self-adjoint. As this constitutes a stronger and more general approach, we present it separately in Appendix \ref{Sec-Appendix} for reader’s convenience.
	\end{remark}

	\begin{proposition}\label{Lem-mu->u-cgs}
		If $u_\mu$ denotes the Yosida approximated solution (see \eqref{Def-u_mu}) corresponding to $u$, then, $u_\mu \in L^\infty(0,T; D(A))$ and it satisfies the following limit:
		\begin{align}\label{eqn-Yosida-cgs}
			\lim_{\mu\to \infty} \norm{Au_\mu - Au}_{L^\infty(0,T;L^2(\bO))} = 0.
		\end{align}
	\end{proposition}
	
	\begin{proof}
		Let us choose and fix $u_0\in D(A)\cap\bM$. Suppose $u$ is the unique strong solution to the problem \eqref{eqn-main-heat} such that $u\in W^{1,\infty}([0,T];L^2(\bO))\cap L^\infty(0,T; D(A))$ (guaranteed by Theorem \ref{Thm-Main-Exis}). Assume that $u_\mu$ is the Yosida approximated solution, defined in \eqref{Def-u_mu}, corresponding to the strong solution $u$.
		Using Remark \ref{Rmk-Au_mu-est} and in particular \eqref{eqn-A_mu-A}, we have 
		\begin{align}
			\sup_{t\in[0,T]} \norm{Au_\mu(t) - Au(t)}_{L^2(\bO)} \le \sup_{t\in[0,T]} %\norm{\frac{-A}{\mu I + A}}_{\Ls(L^2(\bO))} 
			\| A u \|_{L^2(\bO)} 
			%& \le \| u \|_{L^\infty(0,T;D(A))} 
			< \infty,
		\end{align}\label{eqn-L^inf-D(A)-mu} 
		since $u\in L^\infty(0,T;D(A))$.  It follows that
		\begin{align}\label{eqn-mu-D(A)}
			\sup_{t\in[0,T]} \norm{Au_\mu(t)}_{L^2(\bO)} \leq C,
		\end{align}
		where $C>0$ is independent of $\mu$, hence $u_\mu \in L^\infty(0,T; D(A))$. 
		Note from the definition of $\Jn$ that
		\begin{align}
			\Jn = \mu (\mu I + A)^{-1} = I - A(\mu I + A)^{-1}.
		\end{align}
		Then, by using the bound \eqref{eqn-oper-bound}, we estimate for all $u\in D(A)$ that 
		\begin{align}
			\norm{ \Jn u - u }_{L^2(\bO)} &=\norm{- A(\mu I + A)^{-1}u}_{L^2(\bO)}= \norm{ (\mu I + A)^{-1} A u }_{L^2(\bO)} \\& \le \norm{ (\mu I + A)^{-1}}_{\Ls(L^2(\bO))} \norm{Au}_{L^2(\bO)}
			\le \frac{1}{\mu} \norm{Au}_{L^2(\bO)}\nonumber\\&\to 0\ \text{ as }\ \mu\to\infty. \label{eqn-conv}
		\end{align}
		Since $D(A)$ is dense in $L^2(\bO)$, for any $u\in L^2(\bO)$, there exists a sequence $\{u_n\}_{n\in\N}\in D(A)$ such that $u_n\to u$ in $L^2(\bO)$ as $n\to\infty$. Then, by using triangle inequality, we have 
		\begin{align}\label{eqn-111}
			\|\Jn u - u\|_{L^2(\bO)}\leq\|\Jn (u-u_n)\|_{L^2(\bO)}+\|\Jn u_n-u\|_{L^2(\bO)}+\|u_n-u\|_{L^2(\bO)}. 
		\end{align}
		Using the bound \eqref{eqn-oper-bound}, we deduce 
		$$\|\Jn\|_{\Ls(L^2(\bO))} \le 1, \ \text{ for }\ \mu>0. $$
		Therefore, the first term in the right hand side of \eqref{eqn-111}  satisfies $$\|\Jn (u-u_n)\|_{L^2(\bO)}\leq  \|\Jn\|_{\Ls(L^2(\bO))}\|u_n-u\|_{L^2(\bO)}\leq  \|u_n-u\|_{L^2(\bO)}.$$ 
		Therefore, from \eqref{eqn-111}, we infer 
		\begin{align}\label{eqn-222}
			\|\Jn u - u\|_{L^2(\bO)}\leq \|\Jn u_n-u\|_{L^2(\bO)}+2\|u_n-u\|_{L^2(\bO)}. 
		\end{align}
		Since $u_n\in D(A)$,  by using \eqref{eqn-conv}, we obtain 	for fixed $n$, $ \|\Jn u_n-u\|_{L^2(\bO)}\to 0$ as $\mu\to \infty$.  Then, letting $\mu\to\infty$ followed by $n\to\infty$ in \eqref{eqn-222}, we deduce that 
		\begin{align}
			\mbox{$\Jn u \to  u$  as $\mu\to\infty$ for any $u\in L^2(\bO)$. }\label{eqn-Jn-cgs}
		\end{align}
		Since $u\in L^\infty(0,T;D(A))$, it follows from the convergence \eqref{eqn-Jn-cgs} that
		\begin{align*}
			\lim_{\mu\to \infty} \sup_{t\in[0,T]} \norm{Au_\mu(t) - Au(t)}_{L^2(\bO)} = \lim_{\mu\to \infty} \sup_{t\in[0,T]} \norm{\Jn Au(t) - Au(t)}_{L^2(\bO)} = 0,
		\end{align*}
		which completes the proof. 
	\end{proof}
	
	\begin{remark}\label{rem-nonlinear}
		Observe that, by applying H\"older's inequality (with the exponent $d$ and $(2d)/(d-2)$), along with the embedding $D(A) \embed L^p(\bO)$ and the estimate given in  \eqref{eqn-mu-D(A)}, we have
		\begin{align}
			\int_0^T\norm{\abs{u_\mu(t)}^{p-2}\nabla u_\mu(t)}_{L^2(\bO)}^2dt 
			& \leq C \int_0^T\norm{\abs{u_\mu(t)}^{p-2}}_{L^d(\bO)}^2\norm{\nabla u_\mu(t)}_{L^{\frac{2d}{d-2}}(\bO)}^2dt\\
			& \leq C \int_0^T \norm{u_\mu(t)}_{L^{d(p-2)}(\bO)}^{2p-4}\norm{A u_\mu(t)}_{L^2(\bO)}^2dt\\
			& \leq C T \sup_{t\in[0,T]}\norm{A u_\mu(t)}_{L^2(\bO)}^{2p-2}\\
			& \leq C(T),\label{eqn-mu-|u|^{p-2}grad-u}
		\end{align}
		where $C(T)>0$ is independent of $\mu$. 
	\end{remark}
	
	Next, we present two useful regularity results concerning the unique strong solution to problem \eqref{eqn-main-heat}, derived using the properties of the Yosida-approximated solutions as outlined below. 
	
	We begin by presenting a remark similar in spirit to Remark \ref{Rmk-Au_mu-est}, which will be employed in the next proposition.
	
	{\begin{remark}	
			Let us choose and fix $z\in L^2(\bO)$. Then, by regularity result (see \cite[Theorem 9.25]{HB-11}), for each $\mu>0$, there exists $y\in D(A)$ such that 
			$$(\mu I + A)y = z %\iff  %y = (\mu I + A)^{-1}z \iff  Ay = z-\mu y\ 
			\text{ in }\ L^2(\bO).$$ Therefore, $y = (\mu I + A)^{-1}z$ in $L^2(\bO)$ and 
			using the above fact, we estimate
			\begin{align}
				\|A^{\frac{1}{2}} (\mu I + A)^{-1}z\|_{L^2(\bO)}^2 & = \|A^{\frac{1}{2}}y \|^2_{L^2(\bO)} = (Ay, y) = (z-\mu y, y) = (z,y) - \mu \|y\|_{L^2(\bO)}^2\\ 
				& \le \|z\|_{L^2(\bO)}\|y\|_{L^2(\bO)} - \mu \|y\|_{L^2(\bO)}^2.
			\end{align}
			For each fixed $z\in L^2(\bO)$, the function
			\begin{align*}
				[0,\infty) \ni l \mapsto \|z\|_{L^2(\bO)} l - \mu l^2 \in \R
			\end{align*}
			attains its maxima at the point $l = \frac{\|z\|_{L^2(\bO)}}{2\mu}$. Hence, we have
			\begin{align}
				\|A^{\frac{1}{2}} (\mu I + A)^{-1}z\|_{L^2(\bO)}^2 \le \frac{\|z\|_{L^2(\bO)}^2}{4\mu}.
			\end{align}
			Since the above inequality holds for any $z\in L^2(\bO)$, it implies that
			\begin{align}\label{eqn-resolvent-A^{1/2}}
				\|A^{\frac{1}{2}} (\mu I + A)^{-1}\|_{\Ls(L^2(\bO))} \le \frac{1}{2\sqrt{\mu}}, \ \text{ for }\ \mu >0,
			\end{align}
			where $\|\cdot\|_{\Ls(L^2(\bO))}$ denotes the operator norm.
	\end{remark}}

{The following result yields the required regularity of the time derivative of the Yosida-approximated solution $u_\mu$, which in turn allows us to deduce the time derivative regularity of the original solution $u$ of the system \eqref{eqn-main-heat}, see Remark~\ref{Rmk-Regularity}.}
	
	\begin{proposition}\label{Prop-A^{3/2}}
		The Yosida approximated solution $u_\mu$ to the problem \eqref{eqn-quant-heat} satisfies
		\begin{align}
			u_\mu\in L^{2}(0,T;D(A^\frac{3}{2}))\ \text{ and }\ \frac{\partial u_\mu}{\partial t} \in L^2(0,T; H_0^1(\bO)). 
			\label{eqn-D(A{3/2})-u'}
		\end{align}
		Moreover, the following estimate holds true:
		\begin{align}
			\sup_{t\in[0,T]}\norm{Au_\mu(t)}_{L^2(\bO)}^2 + \int_0^T\norm{A^{\frac{3}{2}}u_\mu(s)}_{L^{2}(\bO)}^{2} ds \le
			C\big(p,T, \norm{u_0}_{D(A)} \big),	\label{eqn-est-D(A^{3/2})}
		\end{align}
		where $C\big(p,T, \norm{u_0}_{D(A)} \big) > 0$ is independent of $\mu$.
	\end{proposition}
	
	\begin{proof} 
		\textbf{Step I.} First, let us choose and fix $T>0$ and $\mu>0$. Suppose $u_\mu$ is the Yosida approximated unique strong solution corresponding to $u$ such that $u \in L^\infty(0,T; D(A))$. Thus, by using the fact that $A$ is a self-adjoint positive operator and \eqref{eqn-resolvent-A^{1/2}}, we deduce
		\begin{equation}
			{\norm{A^{\frac{3}{2}} u_\mu (t)}_{L^2(\bO)}
				\le \mu\norm{ A^{\frac{1}{2}} (\mu I + A)^{-1}}_{\Ls(L^2(\bO))} \norm{Au (t)}_{L^2(\bO)}
				= \frac{\sqrt{\mu}}{2}\norm{Au(t)}_{L^2(\bO)} < \infty},\label{eqn-A3/2}
		\end{equation}
		where for each $\mu>0$, the above estimate implies that $u_\mu \in L^2(0,T; D(A^{\frac{3}{2}}))$. In a similar way, for each fixed $\mu>0$, using \eqref{eqn-resolvent-A^{1/2}}, we also have
		\begin{align*}
				\int_0^T\norm{A^{\frac{1}{2}} \frac{\partial u_\mu (t)}{\partial t} }_{L^2(\bO)}^2 dt
				& = \int_0^T \norm{  \mu A^{\frac{1}{2}}(\mu I + A)^{-1}}_{\Ls(L^2(\bO))}^2 \norm{\frac{\partial u(t)}{\partial t} }_{L^2(\bO)}^2 dt \\
				& = \frac{\mu}{4}\int_0^T\norm{\frac{\partial u(t)}{\partial t}}_{L^2(\bO)}^2 dt  \le \frac{\mu}{4} \norm{\frac{\partial u}{\partial t}}_{L^2(0,T;L^2(\bO))}^2  < \infty,
		\end{align*}
		where we have utilized the fact that $ \frac{\partial u}{\partial t} \in L^2(0,T; L^2(\bO))$. Thus, for each fixed $\mu>0$, we have 
		\begin{align}\label{eqn-u'_mu-H_0^1}
			\frac{\partial u_\mu}{\partial t} \in L^2(0,T; H_0^1(\bO)) \implies 	\frac{\partial A u_\mu}{\partial t} \in L^2(0,T; H^{-1}(\bO)). 
		\end{align}
		Moreover, the estimate \eqref{eqn-A3/2} implies that $A u_{\mu}\in L^2(0,T;H_0^1(\bO))$. Therefore, 	the Lions-Magenes Lemma yields that the function 
		$$[0,T]\ni t \mapsto \norm{Au_\mu(t)}_{L^2(\bO)}^2 \in \R$$
		is absolutely continuous  and for a.e. $t\in [0,T]$, we have
		\begin{align}
			\frac{d}{dt}\norm{Au_\mu(t)}_{L^2(\bO)}^2 = 2\bigg\langle\frac{\partial A u_\mu(t)}{\partial t}, A u_\mu(t) \bigg\rangle = 2\bigg\langle\frac{\partial  u_\mu(t)}{\partial t}, A^2 u_\mu(t) \bigg\rangle.
		\end{align}
		
		\vskip 1mm
		\noindent
		\textbf{Step II.} We proceed by taking the {duality pairing} of the first equation in the problem \eqref{eqn-Yosida-heat} with $A^2 u_\mu(t)$. By using  integration by parts and the Lions-Magenes Lemma, for a.e. $t\in[0,T]$, we get 
		\begin{align}
			&\frac{1}{2}\frac{d}{dt}\norm{A u_\mu(t)}_{L^{2}(\bO)}^{2}  + \big\|A^{\frac{3}{2}}u_\mu(t)\big\|_{L^{2}(\bO)}^{2} \\
			& =   (A^{\frac{1}{2}}\Jn(\abs{u_\mu(t)}^{p-2} u_\mu(t)), A^{\frac{3}{2}}  u_\mu(t))+ (\norm{\nabla u_\mu(t)}_{L^2(\bO)}^2 + \norm{u_\mu(t)}_{L^p(\bO)}^p)\norm{A u_\mu(t)}_{L^{2}(\bO)}^{2}.\hspace{1cm}\label{eqn-Hol+You}
		\end{align}
		In particular, applying H\"older's inequality (with exponents $2$ and $2$), Young's inequality, integration by parts and identity \eqref{eqn-Lap-Nn}, 
		we further obtain
		\begin{align}
			|(A^{\frac{1}{2}}\Jn(\abs{u_\mu}^{p-2} u_\mu), A^{\frac{3}{2}}  u_\mu)|&\leq\|A^{\frac{1}{2}}(\abs{u_\mu}^{p-2} u_\mu)\|_{L^2(\bO)}\|A^{\frac{3}{2}}  u_\mu\|_{L^2(\bO)}\\&=(p-1)\||u_\mu|^{p-2}\nabla u_\mu\|_{L^2(\bO)}\|A^{\frac{3}{2}}  u_\mu\|_{L^2(\bO)}\\&\leq\frac{1}{2}\|A^{\frac{3}{2}}  u_\mu\|_{L^2(\bO)}^2+\frac{(p-1)^2}{2}\||u_\mu|^{p-2}\nabla u_\mu\|_{L^2(\bO)}^2.
		\end{align}
		Plugging the above estimate into \eqref{eqn-Hol+You} and using Gronwall's inequality yield
		\begin{align*}
			&	\norm{A u_\mu(t)}_{L^{2}(\bO)}^{2} +  \int_0^T\norm{A^{\frac{3}{2}}u_\mu(t)}_{L^{2}(\bO)}^{2} dt \\
			& \leq \bigg[\norm{A \Jn u_0}_{L^{2}(\bO)}^{2} + (p-1)^2\int_0^T \norm{\abs{u_\mu(s)}^{p-2}\nabla u_\mu(t)}_{L^2(\bO)}^2 dt \bigg]\\
			& \quad \times \exp\bigg\{2\int_0^T (\norm{\nabla u_\mu(t)}_{L^2(\bO)}^2 + \norm{u_\mu(t)}_{L^p(\bO)}^p)dt\bigg\}\\
			& \leq  \big[ \norm{A u_0}_{L^{2}(\bO)}^{2} + C\big(p, T, \norm{u_0}_{D(A)}\big)\big] \exp\big\{2 C\big(p, T, \norm{u_0}_{L^p(\bO)\cap H_0^1(\bO)}\big)\big\},
		\end{align*}
		where it is used that $\|\Jn v\|_{L^2(\bO)} \le \|v\|_{L^2(\bO)}$ for all $v\in L^2(\bO)$, and the bounds \eqref{eqn-H_0^1-L^p}, \eqref{eqn-mu-|u|^{p-2}grad-u} and \eqref{eqn-mu-D(A)}. Hence
		\begin{align}
			\sup_{t\in[0,T]}\norm{Au_\mu(t)}_{L^2(\bO)}^2 + 2\int_0^T\norm{A^{\frac{3}{2}}u_\mu(t)}_{L^{2}(\bO)}^{2} dt
			\leq C(p, T, \norm{u_0}_{D(A)}),
		\end{align}
		where $C$ is independent of $\mu$.
		
		\vskip 1mm
		\noindent
		\textbf{Step III.} Taking {duality pairing} of the first equation of the system \eqref{eqn-Yosida-heat} with $A\frac{\partial u_\mu}{\partial t}$ gives
		\begin{align*}
			\left\langle \frac{\partial u_\mu}{\partial t}, A\frac{\partial u_\mu}{\partial t} \right\rangle & = \left\langle -A u_\mu - \Jn(|u_\mu|^{p-2} u_\mu) + \left( \norm{\nabla u_\mu}_{L^2(\bO)}^2 + \norm{u_\mu}_{L^p(\bO)}^p \right) u_\mu, A \frac{\partial u_\mu}{\partial t} \right\rangle.
		\end{align*}
		Integration by parts and  the self-adjointness of $\Jn$ yield 
		\begin{align*}
			\norm{A^{\frac{1}{2}} \frac{\partial u_\mu}{\partial t}}_{L^{2}(\bO)}^2 + \left\langle A u_\mu, A \frac{\partial u_\mu}{\partial t}\right\rangle & = - \left(|u_\mu|^{p-2} u_\mu, \Jn A \frac{\partial u_\mu}{\partial t}\right)\\
			&\quad + \left( \norm{\nabla u_\mu}_{L^2(\bO)}^2 + \norm{u_\mu}_{L^p(\bO)}^p \right)\frac12 \frac{d}{dt}\norm{\nabla u_\mu}_{L^{2}(\bO)}^2.
		\end{align*}
		Since $Au_\mu, A\frac{\partial u_\mu}{\partial t} \in L^2(0,T; L^2(\bO))$, by Lions-Magenes Lemma, it follows that
		\begin{align*}
			\norm{A^{\frac{1}{2}} \frac{\partial u_\mu}{\partial t}}_{L^{2}(\bO)}^2 + \frac{1}{2}\frac{d}{dt}\norm{A u_\mu}_{L^{2}(\bO)}^2 & =  - \left(|u_\mu|^{p-2} u_\mu, \Jn A \frac{\partial u_\mu}{\partial t} \right)\\
			&\quad + \left( \norm{\nabla u_\mu}_{L^2(\bO)}^2 + \norm{u_\mu}_{L^p(\bO)}^p \right)\frac{1}{2} \frac{d}{dt}\norm{\nabla u_\mu}_{L^{2}(\bO)}^2.
		\end{align*}
		Next, the Cauchy-Schwarz and Young's inequalities yield
		\begin{align*}
			\norm{A^{\frac{1}{2}} \frac{\partial u_\mu}{\partial t}}_{L^{2}(\bO)}^2 + \frac{1}{2}\frac{d}{dt}\norm{A u_\mu}_{L^{2}(\bO)}^2
			& \le \frac{(p-1)^2}{2} \norm{\abs{u_\mu}^{p-2}\nabla u_\mu}_{L^2(\bO)}^2  + \frac{1}{2}\norm{A^{\frac{1}{2}} \frac{\partial u_\mu}{\partial t}}_{L^{2}(\bO)}^2\\
			&\quad + \frac{1}{2}\left( \norm{\nabla u_\mu}_{L^2(\bO)}^2 + \norm{u_\mu}_{L^p(\bO)}^p \right) \frac{d}{dt}\norm{\nabla u_\mu}_{L^{2}(\bO)}^2.
		\end{align*}
		Integrating over the interval $[0,t]$, for $t\in(0,T]$, we infer
		\begin{align*}
			&\|Au_\mu(t)\|_{L^2(\bO)}^2 + \int_0^{t}	\norm{A^{\frac{1}{2}} \frac{\partial u_\mu(s)}{\partial s}}_{L^{2}(\bO)}^2ds\\
			&\leq {(p-1)^2} \int_0^t \norm{\abs{u_\mu(s)}^{p-2}\nabla u_\mu(s)}_{L^2(\bO)}^2 ds %+ \int_0^{t} \norm{ \frac{\partial u_\mu(s)}{\partial s}}_{L^{2}(\bO)}^2 ds 
			+ \|Au_\mu(0)\|_{L^2(\bO)}^2\\
			& \quad + \sup_{t\in[0,T]} \left( \norm{\nabla u_\mu(t)}_{L^2(\bO)}^2 + \norm{u_\mu(t)}_{L^p(\bO)}^p \right)^2\Big[ \norm{\nabla u_\mu(t)}_{L^2(\bO)}^2 -\norm{\nabla u_\mu(0)}_{L^2(\bO)}^2 \Big].
		\end{align*}
		Now, utilizing the fact $\|\Jn v\|_{L^2(\bO)} \le \|v\|_{L^2(\bO)}$ for all $v\in L^2(\bO)$, the property $\Jn A =  A \Jn$, \eqref{eqn-H_0^1-L^p} and taking supremum over $[0,T]$ on both sides assert
		\begin{align*}
			&\sup_{t\in[0,T]}\|Au_\mu(t)\|_{L^2(\bO)}^2 + \int_0^{T}	\norm{A^{\frac{1}{2}} \frac{\partial u_\mu(t)}{\partial t}}_{L^{2}(\bO)}^2dt\\
			&\leq  C (p, \norm{u_0}_{L^p(\bO)\cap H_0^1(\bO)}) + \| Au_0\|_{L^2(\bO)}^2\\
			&\quad + T\sup_{t\in[0,T]} \Big[ \big( \norm{\nabla u_\mu(t)}_{L^2(\bO)}^2 + \norm{u_\mu(t)}_{L^p(\bO)}^p \big)^2 \norm{\nabla u_\mu(t)}_{L^2(\bO)}^2 \Big],
		\end{align*}
		where we used Remark \ref{rem-nonlinear}. Since $u_\mu\in L^\infty(0,T; D(A))$ with the bounds \eqref{eqn-mu-D(A)} and \eqref{eqn-est-D(A^{3/2})}, along with \eqref{eqn-H_0^1-L^p} for $u_\mu$ instead $u,$ it immediately follows that
		\begin{equation}\label{eqn-L^infty-H01}
			\norm{ \frac{\partial u_\mu}{\partial t}}_{L^2(0,T;H_0^1(\bO))}^2 \leq C(p,T, \norm{Au_0}_{L^2(\bO)}).
		\end{equation}
		It shows that $\frac{\partial u_\mu}{\partial t} \in L^2(0,T;H_0^1(\bO))$, which  concludes the proof.
	\end{proof}
	
	\begin{remark}\label{Rmk-Regularity}
		Note that the bounds in  \eqref{eqn-est-D(A^{3/2})} and \eqref{eqn-L^infty-H01} imply that  sets $\{u_\mu\}_{\mu>0}$ and $\{\frac{\partial u_\mu}{\partial t}\}_{\mu>0}$ are bounded in $L^2(0,T; D(A^{\frac{3}{2}}))$ and $L^2(0,T;H_0^1(\bO))$, respectively. Therefore, by using the weakly lower semicontinuity property of the norms, we deduce
		\begin{align*}
			\int_0^T \norm{A^{\frac{3}{2}}u(t)}_{L^2(\bO)}^2 dt \le \liminf_{\mu \to \infty} \int_0^T\norm{A^{\frac{3}{2}}u_\mu(t)}_{L^{2}(\bO)}^{2} dt
			\leq C(p, T, \norm{u_0}_{D(A)}) < \infty,
		\end{align*}
		and 
		\begin{align*}
			\int_0^T \norm{A^{\frac{1}{2}} \frac{\partial u(t)}{\partial t}}_{L^2(\bO)}^2 dt \le \liminf_{\mu \to \infty} \int_0^T\norm{A^{\frac{1}{2}} \frac{\partial u_\mu(t)}{\partial t}}_{L^{2}(\bO)}^{2} dt
			\leq C(p,T, \norm{Au_0}_{L^2(\bO)}) < \infty.
		\end{align*}
		Thus, $u\in L^2(0,T; D(A^{\frac{3}{2}}))$ and $\frac{\partial u}{\partial t} \in L^2(0,T;H_0^1(\bO))$. Hence by an application of the Lions-Magenes Lemma \cite{JLL+EM-II-72}, we finally have
		\begin{align}
			u \in C([0,T]; D(A)).
		\end{align}
	\end{remark}
	This completes the remaining part of the proof of one of the main result of this article, Theorem \ref{Thm-Main-Exis}.
	
	Let us now conclude this section by presenting a time regularity result as follows.
	
	\subsection{A time regularity result}
	
	Next result helps us to obtain an improved regularity in time, which ultimately helps in the convergence result discussed in Section \ref{Sec-Asy-anal}.
	
	{\begin{remark}
			Note that, since the goal of the following proposition is merely to establish that $\frac{\partial^2 u_\mu}{\partial^2 t}\in L^2(0,T;H^{-1}(\bO))$, there is no need to verify any compatibility conditions, cf. \cite[Theorem 2.1]{RT-82} or \cite[Thoerem 6, p. 386]{LCE-10}.
	\end{remark}}
	
	\begin{proposition}\label{Lem-u_t-is-C}
		The Yosida approximated solution $u_\mu$ to the problem \eqref{eqn-quant-heat} satisfies %$\frac{\partial u_\mu}{\partial t}\in L^2(0,T;H_0^1(\bO))$, 
		$\frac{\partial^2 u_\mu}{\partial^2 t}\in L^2(0,T;H^{-1}(\bO))$.
		%, where $' = \frac{\partial}{\partial t}$. 
		Thus, the following map is absolutely continuous:
		\begin{align}
			(0,T)\ni t \mapsto \norm{\frac{\partial u_\mu(t)}{\partial t}}_{L^2(\bO)}^2\in \R.%, \ \text{ and in particular }\ u_\mu \in C^1((0,T); L^2(\bO)).
		\end{align}
		Moreover, for a.e. $t\in (0,T)$, it follows that
		\begin{align*}
			\frac{d}{dt}\norm{\frac{\partial u_\mu(t)}{\partial t}}_{L^2(\bO)}^2 =2 \left\langle \frac{\partial^2 u_\mu(t)}{\partial^2 t}, u_\mu(t) \right\rangle.
		\end{align*}
	\end{proposition}
	
	\begin{proof}
		Let us first show that $\frac{\partial^2 u_\mu}{\partial^2 t} \in L^2(0,T; H^{-1}(\bO))$.
		
		\vskip 1mm
		\noindent
		\textbf{Step I.} We choose and fix $\psi \in H_0^1(\bO)$, with $\norm{\psi}_{H_0^1(\bO)} \leq 1$. Since $u_\mu$ is a solution to the problem \eqref{eqn-Yosida-heat}, it follows  for a.e. $t\geq 0$, that
		\begin{align*}
			\left\langle \frac{\partial^2 u_\mu(t)}{\partial^2 t}, \psi \right\rangle
			& =  \left\langle \frac{\partial}{\partial t} \big(-A u_\mu(t) - \Jn(|u_\mu(t)|^{p-2}u_\mu(t))), \psi \right\rangle\\ 
			& \quad + \left\langle( \norm{\nabla u_\mu(t)}_{L^2(\bO)}^2 + \norm{u_\mu(t)}_{L^p(\bO)}^p)  u_\mu(t)\big), \psi \right\rangle.
		\end{align*}
		Integration by parts in space and chain rule in time yields
		\begin{align}
			\hspace{-9mm}\left\langle \frac{\partial^2 u_\mu(t)}{\partial^2 t}, \psi \right\rangle
			& = \left(\nabla \frac{\partial u_\mu(t)}{\partial t}, \nabla \psi \right) - (p-1)\left( \Jn\Big(|u_\mu(t)|^{p-2}\frac{\partial u_\mu(t)}{\partial t}\Big), \psi \right)\\
			&\quad +\left( \norm{\nabla u_\mu(t)}_{L^2(\bO)}^2 + \norm{u_\mu(t)}_{L^p(\bO)}^p \right) \left(\frac{\partial u_\mu(t)}{\partial t}, \psi \right) \label{eqn-L^2H^{-1}-1}\\
			& \quad + \left(2\int_\bO \nabla u_\mu(t) \cdot \nabla \frac{\partial u_\mu(t)}{\partial t} dx + p\int_\bO |u_\mu(t)|^{p-2}u_\mu(t) \frac{\partial u_\mu(t)}{\partial t} dx \right)(u_\mu(t), \psi).
		\end{align}
		In particular, an application of the Cauchy-Schwarz inequality infer
		\begin{equation}\label{eqn-Holder-1}
			\left\{
			\begin{aligned}
				\int_\bO \nabla u_\mu \cdot \nabla \frac{\partial u_\mu}{\partial t} dx&= -	\int_\bO \Delta u_\mu \frac{\partial u_\mu}{\partial t} dx  \leq \norm{A u_\mu}_{L^{2}(\bO)} \norm{ \frac{\partial u_\mu}{\partial t}}_{L^{2}(\bO)},\\
				\int_\bO |u_\mu|^{p-2}u_\mu \frac{\partial u_\mu}{\partial t} dx & \leq \norm{u_\mu}_{L^{2p-2}(\bO)}^{p-1} \norm{\frac{\partial u_\mu}{\partial t}}_{L^{2}(\bO)}. 
			\end{aligned}
			\right.
		\end{equation}
		Thus, by utilizing the inequality $\|\Jn v\|_{L^2({\bO})} \le \|v\|_{L^2({\bO})}$ for all $v \in L^2(\bO)$, along with the Cauchy-Schwarz and Poincar\'e's inequalities, and applying the estimates from \eqref{eqn-Holder-1} in \eqref{eqn-L^2H^{-1}-1}, we deduce
		\begin{align*}
			&\left|\left\langle \frac{\partial^2 u_\mu}{\partial^2 t}, \psi \right\rangle\right|\\
			& \leq \bigg[ \norm{A^{\frac{1}{2}} \frac{\partial u_\mu}{\partial t}}_{L^2(\bO)} + (p-1)\norm{|u_\mu|^{p-2}\frac{\partial u_\mu}{\partial t}}_{H^{-1}(\bO)} + \frac{1}{\lambda_1^{\frac{1}{2}}} \big( \norm{\nabla u_\mu}_{L^2(\bO)}^2 + \norm{u_\mu}_{L^p(\bO)}^p \big) \norm{\frac{\partial u_\mu}{\partial t}}_{L^2(\bO)} \\
			& \qquad + \frac{1}{\lambda_1^{\frac{1}{2}}}\left(2\norm{A u_\mu}_{L^{2}(\bO)} + p\norm{u_\mu}_{L^{2p-2}(\bO)}^{p-1}\right) \norm{\frac{\partial u_\mu}{\partial t}}_{L^{2}(\bO)}\norm{u_\mu}_{L^2(\bO)} \bigg] \norm{\psi}_{H_0^1(\bO)},
		\end{align*}
		{where $\lambda_1$ is the Poincar\'e constant and implicitly used fact that $(A^{1/2}\cdot,A^{1/2}\cdot) = (\nabla \cdot, \nabla \cdot).$}
		By using the embedding $H_0^1(\bO)\embed L^{\frac{2d}{d-2}}(\bO)\embed L^{\frac{2d}{d+2}}(\bO) \embed H^{-1}(\bO)$, it follows that
		\begin{align}
			& \left|\left\langle \frac{\partial^2 u_\mu}{\partial^2 t}, \psi \right\rangle\right|^2\\
			&\nonumber \leq \frac{5}{\lambda_1}\bigg[ \lambda_1 \norm{A^{\frac{1}{2}} \frac{\partial u_\mu}{\partial t}}_{L^2(\bO)}^2 + C\norm{|u_\mu|^{p-2}\frac{\partial u_\mu}{\partial t}}_{L^{\frac{2d}{d+2}}(\bO)}^2 + \big( \norm{\nabla u_\mu}_{L^2(\bO)}^2 + \norm{u_\mu}_{L^p(\bO)}^p \big)^2 \norm{\frac{\partial u_\mu}{\partial t}}_{L^2(\bO)}^2\\
			& \quad + \bigg(4\norm{A u_\mu}_{L^{2}(\bO)}^2+ p^2 \norm{u_\mu}_{L^{2p-2}(\bO)}^{2p-2} \bigg)\norm{\frac{\partial u_\mu}{\partial t}}_{L^{2}(\bO)}^2\norm{u_\mu}_{L^2(\bO)}^2 \bigg] \norm{\psi}_{H_0^1(\bO)}^2. \label{eqn-L^2H^{-1}-2}
		\end{align}
		\textbf{Step II.} Next, by an application of H\"older's inequality (with exponents $r$ and $s$), we assert
		\begin{align*}
			\norm{|u_\mu|^{p-2}\frac{\partial u_\mu}{\partial t}}_{L^{\frac{2d}{d+2}}(\bO)}^2 
			\leq C \norm{u_\mu}_{L^{r(p-2)}(\bO)}^{2p-4} \norm{\frac{\partial u_\mu}{\partial t}}_{L^s(\bO)}^2, \text{ with }\ \frac{1}{s} + \frac{1}{r}= \frac{d+2}{2d}.
		\end{align*}
		Using the embeddings $D(A)\embed L^{\frac{d(p-2)}{2}}(\bO)$, for $p$ in \eqref{Embed-D(A)-L^p}, and $H_0^1(\bO) \embed L^{\frac{2d}{d-2}}(\bO)$, we obtain
		\begin{align*}
			\norm{|u_\mu|^{p-2}\frac{\partial u_\mu}{\partial t}}_{L^{\frac{2d}{d+2}}(\bO)}^2
			& \leq   \norm{u_\mu}_{L^{\frac{d(p-2)}{2}}(\bO)}^{2p-4} \norm{ \frac{\partial u_\mu}{\partial t}}_{L^{\frac{2d}{d-2}}(\bO)}^2 \leq C\norm{Au_\mu}_{L^2(\bO)}^{2p-4} \norm{A^{\frac{1}{2}} \frac{\partial u_\mu}{\partial t}}_{L^2(\bO)}^2. 
		\end{align*}
		Thus, using the above fact and $ \norm{u_\mu}_{L^2(\bO)}^2\leq  \norm{u}_{L^2(\bO)}^2=1$ in \eqref{eqn-L^2H^{-1}-2}, we deduce
		\begin{align*}
			\norm{\frac{\partial^2 u_\mu}{\partial^2 t}}_{H^{-1}(\bO)}^2
			& \leq C\bigg[ \big[\lambda_1 + C \norm{Au_\mu}_{L^2(\bO)}^{2p-4}  \big]\norm{\frac{\partial u_\mu}{\partial t}}_{H_0^1(\bO)}^2 + \big[ \big( \norm{\nabla u_\mu}_{L^2(\bO)}^2 + \norm{u_\mu}_{L^p(\bO)}^p \big)^2\\
			& \qquad\quad + 4\norm{A u_\mu}_{L^2(\bO)}^2+ p^2 \norm{u_\mu}_{L^{2p-2}(\bO)}^{2p-2} \big] \norm{\frac{\partial u_\mu}{\partial t}}_{L^2(\bO)}^2\bigg]  .
		\end{align*}
		Integrating with respect to $t$ and, applying \eqref{eqn-mu-D(A)}, \eqref{eqn-H_0^1-L^p}, \eqref{eqn-est-L^2p-2}, and \eqref{eqn-L^infty-H01}, we infer
		\begin{align}
			\int_0^{T} \norm{\frac{\partial^2 u_\mu(t)}{\partial^2 t}}_{H^{-1}(\bO)}^2 dt
			& \leq C\sup_{t\in[0,T]} \Big[\lambda_1 +  C\norm{Au_\mu(t)}_{L^2(\bO)}^{2p-4} \Big] \int_0^{T} \norm{\frac{\partial u_\mu(t)}{\partial t}}_{H_0^1(\bO)}^2 dt \\
			& \quad + C\sup_{t\in[0,T]}\Big[ \big( \norm{\nabla u_\mu(t)}_{L^2(\bO)}^2 + \norm{u_\mu(t)}_{L^p(\bO)}^p \big)^2 + p^2\norm{Au_\mu(t)}_{L^2(\bO)}^{2p-2}\\
			& \qquad + 4\norm{\nabla u_\mu(t)}_{L^2(\bO)}^2\Big]  \int_0^{T} \norm{\frac{\partial u_\mu(t)}{\partial t}}_{L^2(\bO)}^2 dt \\
			& \leq C(p, T, \| u_0\|_{D(A)}) < \infty.\label{eqn-est-u''}
		\end{align}
		Finally, we produce that $\frac{\partial^2 u_\mu}{\partial^2 t} \in L^2(0,T; H^{-1}(\bO))$. On the other hand, from \eqref{eqn-L^infty-H01}, we already know that $\frac{\partial u_\mu}{\partial t}\in L^2(0,T; H_0^1(\bO))$.
		Thus, an application of Lions-Magenes Lemma (see \cite{JLL+EM-II-72}) asserts that the following map is absolutely  continuous:
		\begin{align}
			(0,T)\ni t \mapsto \left\|\frac{\partial u_\mu(t)}{\partial t} \right\|_{L^2}^2\in \R.
		\end{align}
		Moreover, it follows for a.e. $t\in [0,T]$ that
		\begin{align*}
			\frac{d}{dt}\norm{\frac{\partial u_\mu(t)}{\partial t}}_{L^2(\bO)}^2 =2 \left\langle \frac{\partial^2 u_\mu(t)}{\partial^2 t}, u_\mu(t) \right\rangle,
		\end{align*}
		which completes the proof.
	\end{proof}
	
	\begin{remark}\label{Rmk-C^1} 
		Observe from \eqref{eqn-L^infty-H01} and \eqref{eqn-est-u''} that the sets $\{\frac{\partial u_\mu}{\partial t}\}_{\mu>0}$ and $\{\frac{\partial^2 u_\mu}{\partial^2 t}\}_{\mu>0}$ are bounded independent of $\mu$ in $L^2(0,T; H_0^1(\bO))$ and $L^2(0,T; H^{-1}(\bO))$, respectively. Thus, by the weakly lower semicontinuity property of the norm, we obtain
		\begin{align*}
			\int_0^{T}\norm{\frac{\partial u(t)}{\partial t}}_{H_0^1(\bO)}^2 dt 
			& \le \liminf_{\mu \to \infty}\int_0^{T}\norm{\frac{\partial u_\mu(t)}{\partial t}}_{H_0^1(\bO)}^2 dt \leq C(p,T, \norm{u_0}_{D(A)}) < \infty,
		\end{align*}
		and
		\begin{align*}
			\int_0^{T} \norm{\frac{\partial^2 u(t)}{\partial^2 t}}_{H^{-1}(\bO)}^2 dt
			& \leq	\liminf_{\mu \to \infty}\int_0^{T} \norm{\frac{\partial^2 u_\mu(t)}{\partial^2 t}}_{H^{-1}(\bO)}^2 dt
			\leq C(p, T, \| u_0\|_{D(A)}) < \infty.
		\end{align*}
		By the Lions-Magenes Lemma, it follows that $u\in C^1((0,T); L^2(\bO))$, for any $T>0$. Hence it implies $u\in C^1((0,\infty); L^2(\bO))$.
	\end{remark}
	
	\begin{remark}
		Let us note that when $\bO\subset \R^d$ is a bounded domain, the previous results can be obtained by leveraging the fact that the Dirichlet-Laplacian $A$ has a discrete spectrum. In this case, one replaces the resolvent apporach with the spectral one. Alternatively, in the spectral-measure approach (see Section \ref{Sec-Appendix}), one replaces $\lambda$ by the eigenvalues $\lambda_j's$ and the norm $\norm{\big[ \int_0^\infty \varphi(\lambda)E(d\lambda)\big] u}_{L^2(\bO)}^2$ by the series $\sum_{j=1}^{\infty} \varphi(\lambda_j) |(u, e_j)|^2$.
	\end{remark}

	\section{Asymptotic analysis}\label{Sec-Asy-anal}
	In this section, we analyze the long-time behavior of the unique global strong solution to problem \eqref{eqn-main-heat-recall}. Using the regularity results from Section \ref{Subsec-Thm-II}, we show that the orbits $\{u(t): t\geq 0\}$ and $\{u(t): t\geq 1\}$ are bounded in $D(A^\alpha)$ and $D(A^\beta)$ for any $\alpha \in (\frac{1}{2},1)$ and $\beta \in (1,\frac{3}{2})$, respectively. Applying Sobolev embeddings, we establish that $\{u(t): t\geq1\}$ is precompact and that the omega-limit set $\omega(u)$ is compact in the space $W^{2,q}(\bO)\cap W^{1,q}_0(\bO)$. After introducing key tools for the \L{}ojasiewicz-Simon inequality \cite{MAJ-98} needed in our setting, we show that stationary solutions to \eqref{eqn-stationary} are exactly the critical points of the energy functional $\En$ and every element of the $\omega(u)$ is a critical point. Finally, using the \L{}ojasiewicz-Simon inequality, we prove that the solution converges in $W^{2,q}(\bO)\cap W^{1,q}_0(\bO)$ to a stationary state as time goes to infinity.
	
	Since we employ the Rellich-Kondrachov compactness Theorem, we assume throughout this section that $\bO \subset \mathbb{R}^d$ is a bounded domain.

	\subsection{Some uniform in time bounds on the solution}\label{Subsec-Other-reg}
	In this subsection, we present two regularity results, when the initial condition is less  regular than $D(A)$ and in $D(A)$, i.e., when $u_0 \in D(A^\alpha)$, for $\alpha\in (\frac{1}{2}, 1)$   and $u_0 \in D(A)$. Finally, by utilizing these results, we show that the omega-limit set $\omega(u)$ is compact.
	
	We choose and fix 
	\begin{align}\label{eqn-H_0^1-in-L^p-D(A)-L^2p-2}
		p \in \left\{
		\begin{aligned}
			[2,\infty), & \text{ when }\  1\le d \le4,\\ 
			\bigg[2, \frac{2d-4}{d-4}\bigg), & \text{ when }\  d\ge 5,
		\end{aligned}
		\right.
	\end{align}
	so that the Sobolev embedding $D(A)\embed L^{2p-2}(\bO)$ is valid.
	Then, we have the following result, motivated from \cite[Theorem 2.5]{PR-06}.  
	\begin{proposition}\label{Prop-D(A^alpha)}
		Let us choose and fix $\alpha \in (\frac{1}{2}, 1)$. If $u$ is a global strong solution, then for  $u_0\in D(A^\alpha) \cap\bM$,   the orbit $\{u(t): t\geq 0\} $ is bounded in $D(A^\alpha)$.
	\end{proposition}
	
	\begin{proof}
		\textbf{Step I.} Let us choose and fix $u_0\in D(A^\alpha)$, for $\alpha\in (\frac{1}{2}, 1)$. Suppose $u$ is the unique strong solution to the problem \eqref{eqn-main-heat}, guaranteed by Theorem \ref{Thm-Main-Exis}. 
		By the variation of constants formula, we have
		\begin{align}\label{eqn-variation-of-cons-form}
			u(t) = e^{-At}u_0 + \int_0^t e^{-A(t-s)} F(u(s))ds,
		\end{align}
		where $F(u) = -\abs{u}^{p-2}u + (\norm{\nabla u}_{L^2(\bO)}^2 + \norm{u}_{L^p(\bO)}^p)u$.
		By an application of $A^\alpha$ in \eqref{eqn-variation-of-cons-form} gives the following:
		\begin{align*}
			A^\alpha u(t) = A^\alpha e^{-At} u_0 + \int_{0}^{t} A^\alpha e^{-A(t-s)} F(u(s)) ds.
		\end{align*}
		Taking the $L^2(\bO)-$norm on both sides, we obtain
		\begin{align*}
			\norm{A^\alpha u(t)}_{L^2(\bO)} & = \bigg\|A^\alpha e^{-At} u_0 + \int_{0}^{t} A^\alpha e^{-A(t-s)} F(u(s)) ds\bigg\|_{L^2(\bO)}\\
			& \leq \norm{A^\alpha e^{-At} u_0}_{L^2(\bO)} + \int_{0}^{t}  \norm{A^\alpha e^{-A(t-s)} F(u(s))}_{L^2(\bO)} ds.
		\end{align*}
		From \cite[Proposition 1.4.3]{DH-81}, we also have
		\begin{align}\label{eqn-D(A^alpha)-norm}
			\norm{A^\alpha u(t)}_{L^2(\bO)} \leq C_\alpha e^{-\lambda_1 t} \norm{u_0}_{D(A^\alpha)} + C_\alpha \int_{0}^{t}  \frac{e^{-\lambda_1(t-s)}}{(t-s)^\alpha} \norm{F(u(s))}_{L^2(\bO)} ds.
		\end{align}
		{Since $\bO$ is a bounded domain, the Poincar\'e constant $\lambda_1$ is equal to the first eigenvalue of Dirichlet Laplacian.}
		Recall from \eqref{eqn-energy-eqn} (see Section \ref{Subsec-Gradient-flow})
		that, the energy functional $\En(u) = \frac{1}{2}\norm{\nabla u}_{L^2(\bO)}^2 + \frac{1}{p}\norm{u}_{L^p(\bO)}^p$, is dissipative in time and using the estimate  \eqref{eqn-H_0^1-L^p}, we have
		\begin{align}\label{eqn-grad+L^p-est-Energy-1}
			\begin{aligned}
				\norm{\nabla u}_{L^2(\bO)}^2 & \leq 2 E(u) \leq 2 E(u_0) \ \text{ and}\\
				\norm{u}_{L^p(\bO)}^p & \leq p E(u) \leq p E(u_0).
			\end{aligned}
		\end{align}
		\textbf{Step II.} Note that, when $d=1,2$, the Sobolev embedding $H_0^1(\bO) \embed L^{2p-2}(\bO)$ holds for any $2\le p < \infty$, so that 
		\begin{align}\label{eqn-L^{2p-2}-D(A^alpha')}
			\norm{u}_{L^{2p-2}(\bO)}^{p-1} \leq C \norm{\nabla u}_{L^2(\bO)}^{p-1}.
		\end{align}
		Thus, by utilizing bounds from  \eqref{eqn-grad+L^p-est-Energy-1} and  \eqref{eqn-L^{2p-2}-D(A^alpha')} in \eqref{eqn-D(A^alpha)-norm}, we deduce
		\begin{align}\label{eqn-A^alpha}
			\norm{A^\alpha u(t)}_{L^2(\bO)} & \leq C_\alpha e^{-\lambda_1 t}\norm{u_0}_{D(A^\alpha)} + \sup_{t\geq 0}\norm{\nabla u}_{L^2(\bO)}^{p-1}C_\alpha \int_{0}^{t}  \frac{e^{-\lambda_1(t-s)}}{(t-s)^\alpha}  ds\\
			&\leq C_\alpha\norm{u_0}_{D(A^\alpha)} + C_\alpha \lambda_1^{\alpha-1} \Gamma(1-\alpha)  ds := \wtilde{K}_\alpha < \infty.
		\end{align}
		It asserts that $\{u(t) : t\geq 0\}$ is bounded in $D(A^\alpha)$, for $\alpha \in (\frac{1}{2},1)$.
		
		\vskip 1mm
		\noindent
		\textbf{Step III.} On the other hand, when $d \ge 3$, the embedding $H_0^1(\bO)\embed L^{2p-2}(\bO)$ holds only for any $2 \leq p \leq 1 + \frac{d}{d-2}$, which follows from the Step II. 
		Now, let us prove the estimate \eqref{eqn-A^alpha} when $d \geq 3$ and for $1 + \frac{d}{d-2} < p$. 
		An immediate application of the Gagliardo-Nirenberg interpolation inequality, {see \cite[Theorem B]{HB+PM-19}}, yields
		\begin{align*}
			\norm{u}_{L^{2p-2}(\bO)}^{p-1} & \leq C \|A^{\alpha^\prime}u\|_{L^2(\bO)}^{p-1}, \ \text{ for }\ \frac{1}{2}< \alpha^\prime < \alpha <1,
		\end{align*}
		where $\alpha^\prime = \frac{d}{4} \big[\frac{p-2}{p-1}\big]$, and $\alpha^\prime > \frac{1}{2}$ if and only if $p > \frac{2d -2}{d-2} = 1 + \frac{d}{d-2}$. Similarly, when $d=3,4$, the inequality $\alpha'  < 1$ holds for all $2\le p <\infty$. However, for $d\ge5$, the condition $\alpha' = \frac{d}{4} \big[\frac{p-2}{p-1}\big] < 1$ is satisfied if and only if $p < \frac{2d-4}{d-4}$.
		
		Further, an application of the interpolation inequality \cite[Exercise 5*]{DH-81} yields  $$\norm{A^{r}u}_{L^2(\bO)} \le \norm{A^{l}u}_{L^2(\bO)}^\theta \norm{A^{q}u}_{L^2(\bO)}^{1-\theta},\ \text{ for }\ r = \theta l + (1-\theta)q\ \text{ and }\ \theta \in (0,1).$$  Choosing $r = \alpha^\prime$, $l = 1/2$ and $q = \alpha$, we find 
		\begin{align*}
			\norm{u}_{L^{2p-2}(\bO)}^{p-1} & \le C \norm{A^{\alpha^\prime}u}_{L^2(\bO)}^{p-1}\leq C \norm{A^{1/2}u}_{L^2(\bO)}^{\frac{(\alpha - \alpha^\prime)(p-1)}{\alpha - 1/2}} \norm{A^{\alpha}u}_{L^2(\bO)}^{\frac{(\alpha^\prime - 1/2)(p-1)}{\alpha - 1/2}}.
		\end{align*}
		Using Young's inequality, i.e., if $a,b>0$ and $1/\wtilde{p} + 1/\wtilde{q} =1$, then $ab \le \eps a^{\wtilde{p}} + C(\eps) b^{\wtilde{q}}$ with $C(\eps) = (\eps \wtilde{p})^{- \wtilde{q}/ \wtilde{p}} (1/\wtilde{q})$, by choosing $\wtilde{p} = \frac{\alpha - 1/2}{(\alpha - 1/2) - (\alpha^\prime - 1/2)(p-1)}$, $\wtilde{q} = \frac{\alpha - 1/2}{(\alpha^\prime - 1/2)(p-1)}$ and $\eps = 1/(2 C_\alpha \lambda_1^{\alpha -1} \Gamma(1-\alpha))$, we obtain
		\begin{align}
			\norm{u}_{L^{2p-2}(\bO)}^{p-1} & \leq \frac{1}{2 C_\alpha \lambda_1^{1-\alpha} \Gamma(1-\alpha)} \norm{A^{\alpha}u}_{L^2(\bO)} + C(\eps) \norm{A^{1/2}u}_{L^2(\bO)}^{\frac{(\alpha - 1/2)(p-1)}{(\alpha - 1/2) - (\alpha^\prime - 1/2)(p-1)}}\\
			& \leq \frac{1}{2 C_\alpha \lambda_1^{1-\alpha} \Gamma(1-\alpha)} \norm{A^{\alpha}u}_{L^2(\bO)} + C(\eps, E(u_0)),\label{eqn-L^{2p-2}-D^{alpha}}
		\end{align}
		provided $\wtilde{p}, \wtilde{q} >1$, i.e., $\wtilde{p} =  \frac{\alpha - 1/2}{(\alpha - 1/2) - (\alpha^\prime - 1/2)(p-1)}>1$ if and only if $0 > - (\alpha^\prime - 1/2)(p-1)$.
		
		Utilizing the  inequalities \eqref{eqn-grad+L^p-est-Energy-1} and \eqref{eqn-L^{2p-2}-D^{alpha}}, we estimate $\norm{F(u)}_{L^2(\bO)}$ as
		\begin{align*}
			\norm{F(u)}_{L^2(\bO)} & = \big\|-|u|^{p-2}u + \big( \norm{\nabla u}_{L^2(\bO)}^2 + \norm{u}_{L^p(\bO)}^p \big) u\big\|_{L^2(\bO)}\\
			& \leq \norm{u}_{L^{2p-2}(\bO)}^{p-1} + \big( \norm{\nabla u}_{L^2(\bO)}^2 + \norm{u}_{L^p(\bO)}^p \big) \norm{u}_{L^2(\bO)}\\
			& \leq \frac{1}{2 C_\alpha \lambda_1^{\alpha -1}\Gamma(1-\alpha)} \norm{A^{\alpha}u}_{L^2(\bO)} + C(\eps, E(u_0)) + (2+p)E(u_0)\\ 
			& =: \frac{1}{2 C_\alpha \lambda_1^{\alpha -1} \Gamma(1-\alpha)} \norm{A^{\alpha}u}_{L^2(\bO)} + K < \infty,
		\end{align*}
		where we have used the fact that the solution $u$ is invariant under the Manifold $\bM$, i.e., $\norm{u(t)}_{L^2(\bO)} =1$, for any $t\ge 0$.
		Using the above inequality in \eqref{eqn-D(A^alpha)-norm}, we finally have
		\begin{align*}
			\norm{A^\alpha u(t)}_{L^2(\bO)} & \leq C_\alpha \norm{u_0}_{D(A^\alpha)} + C_\alpha  \int_{0}^{t}  \frac{e^{-\lambda_1(t-s)}}{(t-s)^\alpha} \bigg[\frac{1}{2 C_\alpha \lambda_1^{\alpha -1} \Gamma(1-\alpha)} \norm{A^{\alpha}u(s)}_{L^2(\bO)} + K \bigg] ds\\
			& \leq C_\alpha \norm{u_0}_{D(A^\alpha)} + \frac{1}{2 } \sup_{t \ge 0} \norm{A^{\alpha}u(t)}_{L^2(\bO)} + C_\alpha K \lambda_1^{\alpha -1}  \Gamma(1-\alpha).
		\end{align*}
		Upon taking supremum over $\{t\ge0\}$, it follows that 
		\begin{align}
			\frac{1}{2 }\sup_{t \ge 0}\norm{A^\alpha u(t)}_{L^2(\bO)} 
			& \leq C_\alpha\norm{u_0}_{D(A^\alpha)}  + C_\alpha K \lambda_1^{\alpha -1}\Gamma(1-\alpha) =: K_\alpha< \infty.\label{eqn-k-alpha}
		\end{align}
		Hence the orbit $\{u(t) : t\geq 0\}$ is bounded in $D(A^\alpha)$, for $\alpha \in (\frac{1}{2},1)$.
	\end{proof}
	
	\begin{remark}
		Note that from previous proposition, for any $u_0\in D(A^\alpha)$ and any $\alpha\in(\frac{1}{2},1)$, we got the orbit $\{u(t):t\geq 0\}$ is bounded in $D(A^\alpha)$. Similarly, the next theorem demonstrates a regularizing effect, i.e., for any $u_0\in D(A)$ and every $\beta\in(1,\frac{3}{2})$, the orbit $\{u(t):t\geq 1\}$  is bounded in $D(A^\beta)$.
	\end{remark}
	Next, let us assume and fix
	\begin{equation}\label{eqn-values-p}
		\left\{
		\begin{aligned}
			&p\in[2,\infty),\ \text{ for }\ 1\le d \le4,\\
			& p\in\left[2, \frac{2d-6}{d-4}\right),\ \text{ for }\ d \ge 5, 
		\end{aligned}
		\right.
	\end{equation}
	such that the Sobolev embedding $D(A) \embed L^{d(p-2)}(\bO)$ is true. Then, by utilizing Proposition \ref{Prop-D(A^alpha)}, we state the following result motivated from \cite[Theorem 2.6]{PR-06}.
	
	\begin{theorem}\label{Thm-D(A^beta)}
		Suppose $\beta \in \left(1,\frac{3}{2}\right)$.
		If $u_0\in D(A)\cap\bM$, then the solution $u$ of the problem \eqref{eqn-main-heat} satisfies
		\begin{align*}
			\sup_{t\geq 1} \norm{u(t)}_{D(A^\beta)} \leq K_\beta.
		\end{align*}
	\end{theorem}
	
	\begin{proof}
		Let us fix $\beta \in \left(1, \frac{3}{2}\right)$. Applying the operator $A^{\beta}$ on both sides of the equation \eqref{eqn-variation-of-cons-form} and taking the $L^2-$norm of the above equation yield
		\begin{align}
			\norm{A^{\beta} u(t)}_{L^{2}(\bO)} & = \left\|A^{\beta} e^{-At}u_0 + \int_0^t A^{\beta - \frac{1}{2}} e^{-A(t-s)} A^{\frac{1}{2}} F(u(s))ds\right\|_{L^2(\bO)}\\
			& \leq t^{1-\beta}e^{-\lambda_1t}\norm{A u_0}_{L^2(\bO)} + \int_0^t \frac{e^{-\lambda_1(t-s)}}{(t-s)^{\beta - \frac{1}{2}}} \norm{A^{\frac{1}{2}}F(u(s))}_{L^2(\bO)}ds,\label{eqn-A^beta-1}
		\end{align}
		where we have used \cite[Theorem 1.4.3]{DH-81}.
		In particular, by using the identity \eqref{eqn-Lap-Nn}, we obtain
		\begin{align*}
			\norm{A^{\frac{1}{2}}F(u)}_{L^2(\bO)}
			& = (p-1) \norm{\abs{u}^{p-2}\nabla u}_{L^2(\bO)} + (\norm{\nabla u}_{L^2(\bO)}^2 + \norm{u}_{L^p(\bO)}^p) \norm{\nabla u}_{L^2(\bO)}.
		\end{align*}
		An application of H\"older's inequality (with exponents $r$ and $s$) yields 
		\begin{align*}
			\norm{\abs{u}^{p-2}\nabla u}_{L^2(\bO)}  \leq C \norm{\abs{u}^{p-2}}_{L^r(\bO)} \norm{\nabla u}_{L^s(\bO)}
			& = C \norm{u}_{L^{r(p-2)}(\bO)}^{p-2} \norm{\nabla u}_{L^s(\bO)},
		\end{align*}
		where $\frac{1}{r} = \frac{1}{2} - \frac{1}{s}.$
		Choosing  $0 < \alpha < 1$, $r=\frac{d}{2\alpha-1}$ and $s=\frac{2d}{d+2-4\alpha}$, then by using the Gagliardo-Nirenberg interpolation inequality {\cite[see Theorem B]{HB+PM-19}} twice, we deduce
		\begin{align}
			\norm{u}_{L^{r(p-2)}(\bO)}^{p-2}\norm{\nabla u}_{L^{s}(\bO)} & = \norm{u}_{L^{\frac{d(p-2)}{2\alpha-1}}(\bO)}^{p-2}\norm{\nabla u}_{L^{\frac{2d}{d+2-4\alpha}}(\bO)}\\
			& \le {C \norm{A^{\alpha} u}_{L^2(\bO)}^{p-2} \norm{A^{\alpha} u}_{L^2(\bO)}} = \norm{A^{\alpha} u}_{L^2(\bO)}^{p-1},\label{eqn-final-est-alpha}
		\end{align}
		where $\alpha = \frac{d(p-2) + 2}{4p-4} <1$ if and only if $p < \frac{2d-6}{d-4}$, which validates the embedding $D(A)\embed L^{d(p-2)}(\bO)$.
		Using the inequalities \eqref{eqn-grad+L^p-est-Energy-1} and \eqref{eqn-final-est-alpha} in \eqref{eqn-A^beta-1}, we have 
		\begin{align*}
			\norm{A^{\beta} u(t)}_{L^{2}(\bO)}
			& \leq t^{1-\beta}\norm{A u_0}_{L^2(\bO)}\\
			& \quad + C_\beta\int_0^\infty \frac{e^{-\lambda_1(t-s)}}{(t-s)^{\beta -\frac{1}{2}}} \Big[\norm{A^\alpha u(t)}_{L^2(\bO)}^{p-1} + \sqrt{2}(2+p)E(u_0)^{\frac{3}{2}}\Big] ds.
		\end{align*}
		By using Proposition \ref{Prop-D(A^alpha)} and taking supremum over $\{t\ge 1\}$ on both sides, we obtain
		\begin{align*}
			\sup_{t \ge 1}\norm{A^{\beta} u(t)}_{L^{2}(\bO)}
			& \leq \norm{A u_0}_{L^2(\bO)}+  \lambda_1^{\beta -\frac{3}{2}} \Gamma\Big(\frac{3}{2}-\beta\Big) C_\beta K_\alpha^{p-1} =: K_\beta< \infty,
		\end{align*}
		where $K_{\alpha}$ is defined in \eqref{eqn-k-alpha}.	This shows that, for $p,d$ in \eqref{eqn-values-p},
		we have the estimate
		\begin{align*}
			\sup_{t\geq 1} \norm{u(t)}_{D(A^\beta)} \leq K_\beta,
		\end{align*}
		which completes the proof.
	\end{proof}
	
	Let us now fix the nonlinearity exponent of nonlinearity $p$ from \eqref{eqn-values-p} and choose 
	\begin{equation}
		q\in \left[2, \frac{2d}{d+4-4\beta}\right), \ \beta \in \left(1,\frac{3}{2}\right). \label{eqn-values-q}
	\end{equation}
	
	\begin{corollary}[{\cite[Corollary 2.7]{PR-06}}]\label{cor-precompact}
		Suppose $q,p,\beta$ are as mentioned in \eqref{eqn-values-p} and \eqref{eqn-values-q}.
		Then, the orbit $\{u(t) : t\geq 1\}$ is precompact in $W^{2,q}(\bO)\cap W^{1,q}_0(\bO)$.
	\end{corollary}
	
	\begin{proof}
		From Theorem \ref{Thm-D(A^beta)}, it is immediate that the orbit
		$\{u(t) : t\geq 1\}$ is bounded in $D(A^\beta)$, for $\beta \in \left(1, \frac{3}{2}\right)$.
		Furthermore, by an application of the Rellich-Kondrachov  Theorem,
		we have the following compact embedding:
		\begin{align*}
			W^{2\beta, 2}(\bO) \embed W^{2,q}(\bO),\ \text{ when }\ q < \frac{2d}{d+4-4\beta}.
		\end{align*}
		On the other hand, for $\beta \in (1,\frac{3}{2})$, the inequality $2 < \frac{2d}{d+4-4\beta}$ holds.  Hence $\{u(t) : t\geq 1\}$ is precompact in $W^{2,q}(\bO)\cap W^{1,q}_0(\bO)$, provided the assumption \eqref{eqn-values-q} holds, which completes the proof.
	\end{proof}

	\begin{corollary}\label{Cor-omega-compact}
		Let the parameters $p,q,d, \beta$ assume the values given in \eqref{eqn-values-p} and \eqref{eqn-values-q}. Then, the set $\omega(u)$ is non-empty, compact and connected in $W^{2,q}(\bO)\cap W^{1,q}_0(\bO)$. Moreover,
		\begin{align}\label{eqn-dist-cgs}
			\lim_{ t \to \infty}\dist (u(t), \omega(u)):= \lim_{ t \to \infty}\inf_{\{w\in \omega(u)\}}\norm{u(t) - w}_{W^{2,q}(\bO)} = 0.
		\end{align}
	\end{corollary}
	
	\begin{proof}
		The proof follows immediately due the pre-compactness of $\{u(t) : t\geq 1\}$ in $W^{2,q}(\bO)\cap W^{1,q}_0(\bO)$ and by a general result \cite[Theorem 4.3.3]{DH-81}.
	\end{proof}

	\subsection{Some preliminaries for the \L{}ojasiewicz-Simon inequality}
	In this subsection, we present fundamental definitions, tools, and results, such as analytic functions, superposition operators, Fredholm operators, and more, to facilitate the study of the \L{}ojasiewicz-Simon inequality in Hilbert spaces.
	
	\subsubsection{Analytic maps and Fredholm index}
	\begin{definition}[{\cite[Definition 2.1]{FR-20}}]\label{def-analytic}
		Let $\Vn,\Wn$ be real Banach spaces, $\U\subseteq \Vn$ be an open set. A function 
		\begin{align*}
			f:\U\ni u \mapsto f(u) \in \Wn
		\end{align*}
		is called \emph{real analytic} at $u_0\in \U$, if there exist $\delta>0$ and continuous $\R-$multilinear forms $a_k: \Vn^k:=\underbrace{\Vn\times\cdots\times \Vn}_{k-times} \to \Wn$ for all $k\in\N\cup\{0\}$ such that
		\begin{align}\label{eqn-series}
			\sum_{n=0}^{\infty} \norm{a_k}_{\Ls(\Vn^k,\Wn)} \norm{u-u_0}_\Vn^n \ \text{ converges  and }\ f(u) = \sum_{k=0}^{\infty} a_k(u-u_0)\ \text{ in }\ \Wn,
		\end{align}
		for all $\norm{u-u_0}_{\Vn} < \delta$, where $a_k(u-u_0)^k := a_n(u-u_0,\dots, u-u_0)\in\Wn$. A function is real analytic throughout $D$, if it is analytic at each individual point $u_0$ in $D$.
	\end{definition}
	
	\begin{theorem}[{\cite[Theorem 2.2]{FR-20}}]\label{thm-analytic}
		Let $\Vn,\Wn,\Xn$ be Banach spaces, $D\subset\Vn$ and $\wtilde{D}\subset\Wn$ be open and $f:D\to\Wn$, $g:\wtilde{D} \to\Xn$ be analytic with $f(D)\subset\wtilde{D}$. Then, $g\circ f: D \to \Xn$ is analytic.
	\end{theorem}
	Examples of analytic maps include  bounded linear maps.
	\begin{example}[{\cite[Example 2.3]{FR-20}}]\label{Lem-analytic}
		Let $\Vn, \Wn$ be Banach spaces. If $T: \Vn \to \Wn$ is linear and continuous, then it is analytic. The same result holds true for multilinear maps also. 
	\end{example}
	
	\begin{theorem}[{\cite[Theorem 6.8]{JA+PPZ-90}}]\label{Thm-superposition}
		Suppose $\bO \subset \R^d$ is a bouned domain with $C^1$-boundary. Let $F \in C(\R)$. Then, the superposition operator $\Fn: C(\overline{\bO}) \to  C(\overline{\bO})$, $\Fn(v) = F(v)$ is analytic if and if the function $F$ is analytic.
	\end{theorem}

	\begin{definition}[{\cite[Definition 2.4]{FR-20}}]
		Let $\Vn,\Wn$ be Banach spaces. An operator $T\in\Ls(\Vn,\Wn)$ is called a \emph{Fredholm operator}, if both $\dim (\mathrm{Ker}(T))$ and $\mathrm{codim}(\mathrm{Im}(T),\Wn) = \dim(\Wn,\mathrm{Im}(T))$ are finite. The number $$\mathrm{index}(T):= \dim (\mathrm{Ker}(T)) - \mathrm{codim}(\mathrm{Im}(T),\Wn)$$ is called the \emph{Fredholm index of $T$}.
	\end{definition}
	\begin{proposition}[{\cite[XVII, Corollaries 2.6 and 2.7]{SL-93}}]\label{Prop-Fredholm}
		Let $T\in \Ls(\Vn,\Wn)$ be a Fredholm operator. Then
		\begin{itemize}
			\item[$(i)$] the image $\mathrm{Im}(T) \subset \Wn$ is closed,
			\item[$(ii)$] for any compact operator $K: \Vn \to \Wn$, the perturbed operator $T+K$ is Fredholm with $\mathrm{index}\;(T+K) = \mathrm{index}(T)$. This holds, in particular, if $K$ has finite rank.
		\end{itemize}
	\end{proposition}
	
	Now, let us state the famous \L{}ojasiewicz-Simon inequality in the setting of Hilbert spaces taken from \cite[Corollary 5.2]{FR-20}.

	Let $\Vn$ be a Hilbert space, $\U\subseteq \Vn$ an open set, $m\in\N$ and $\En: \U \to \R$, $\Gn:\U \to \R^m$ be analytic. Let us choose and fix $\a \in \U$. Suppose the following assumptions are satisfied:
	\begin{itemize}
		\item[$(i)$] $(\Hn, (\cdot,\cdot)_\Hn)$ is a Hilbert space such that $\Vn\embed \Hn$ densely,
		\item[$(ii)$] the functional $\En$ admits a gradient $\nabla\En(u)$ in $\Hn$ at every point $u\in \U,$ and the mapping $\U\ni u\mapsto \nabla\En(u) \in \Hn$ is analytic,
		\item[$(iii)$] the second Fr\'echet derivative $\En^{\prime\prime}(\a) = (\nabla\En)^\prime(\a): \Vn \to \Hn$ is Fredholm of index zero,
		\item[$(iv)$] for every $u\in \U$, each component $\Gn_k(u):\U \to \R$ of $\Gn$ has an $\Hn$-gradients $\nabla \Gn_k$, and the mapping $\U \ni u \mapsto \nabla \Gn_k(u) \in \Hn$ is analytic, for all $k=1,\dots,m$,
		\item[$(v)$] for each $k=1,\dots,m$, the Fr\'echet derivative $(\nabla\Gn_k)^\prime (\a) : \Vn \to \Hn$ is a compact operator,
		\item[$(vi)$] $\Gn(\a) = 0$ and the $\Hn$-gradients $\nabla\Gn_1 (\a),\dots, \nabla\Gn_m (\a)$ are linearly independent.
	\end{itemize}
	
	\begin{proposition}\label{Prop-Submanifold}
		Assuming that the functional $\Gn$ defined above meets the conditions $\{(iv), (v), (vi)\}$, then the set $\bM := \{u\in \U : \Gn(u) = 0\}$ is a locally analytic submanifold of $\Vn$ of codimension $m$  in a neighborhood of $\a$.
	\end{proposition}
	
	\begin{theorem}\label{Thm-LSI}
		Suppose $\Gn$ and $\En$ satisfy the set of assumptions $\{(i)-(vi)\}$. If $\a$ is a critical point of the restriction $\En|_{\bM}$, then a refined \L{}ojasiewicz-Simon gradient inequality holds at $\a$. More precisely, there exist constants $C,\sigma>0$ and an exponent $\theta \in (0,\frac{1}{2}]$ such that for every $u\in\bM$ satisfying $\|u - \a\|_\Vn \leq \sigma$,
		the following inequality holds:
		\begin{align*}
			\abs{\En(u)- \En(\a)}^{1-\theta} \leq C \norm{\pi_u(\nabla\En (u))}_{\Hn} = C \norm{\nabla_\bM\En (u)}_{\Hn},
		\end{align*}
		where $\pi_u(\cdot)$ is the orthogonal projection onto $\overline{T_u\bM} := \overline{T_u\bM}^{\norm{\cdot}_\Hn}$.
	\end{theorem}
	
	\subsection{The \L{}ojasiewicz-Simon gradient inequality}\label{Subsec-LSI}
	In this subsection, we provide the famous \L{}ojasiewicz-Simon inequality on submanifolds in a Hilbert space in our setting. 
	
	Let us fix $p\in\{2,4,\ldots\}$ such that 
	\begin{equation}\label{eqn-H_0^1-in-L^p-asp}
		p\in [2,\infty),  \text{ if }\  1\le d \le 3,
		\ \text{ and }\ q\in \left[2, \frac{2d}{d+4-4\beta}\right), \ \beta \in \left(1,\frac{3}{2}\right)
	\end{equation}
	throughout this section, so that the Sobolev embedding
	$D(A)\embed C(\overline{\bO})$ is valid. Let us recall the following problem in $(0,\infty)\times\bO$:
	\begin{align}\label{eqn-main-heat-recall}
		\left\{\begin{aligned}
			\frac{\partial u(t)}{\partial t}  & = \Delta u(t) -|u(t)|^{p-2}u(t) + \left( \norm{\nabla u(t)}_{L^2(\bO)}^2 + \norm{u(t)}_{L^p(\bO)}^p \right) u(t),\\
			u(0) & = u_0, \\
			u(t)|_{\partial\bO} & = 0,
		\end{aligned}\right.
	\end{align}
	where $u:[0,\infty) \times\bO \to \R$ with $u(t)=u(t,x)$.
	
	\begin{remark}
		In the works \cite{MAJ-98, PR-06}, the authors have considered the system \[\frac{\partial u(t)}{\partial t} = \Delta u(t) - f(x,u(t)),\] and the corresponding energy functional as
		\begin{align*}
			\En(u) := \int_{\bO} \left[\frac{1}{2}\abs{\nabla u(x)}^2 + F(x,u(x))\right]dx,
		\end{align*}
		where $\displaystyle\frac{\partial F}{\partial u} = f$. By comparing to the unconstrained problem in our case, we have
		\begin{align}\label{eqn-unconstrained}
			\frac{\partial u(t)}{\partial t} = \Delta u(t) -|u(t)|^{p-2}u(t).
		\end{align}
		Thus, $f(u) := |u|^{p-2}u$ and $F(u)= \abs{u}^p$. Then, the corresponding energy functional is
		\begin{align}\label{Def-En-compare}
			\En(u):= \int_{\bO} \left[\frac{1}{2}\abs{\nabla u(x)}^2 + \frac{1}{p}|u(x)|^p\right]dx.
		\end{align}
		In the work of Rybka \cite{PR-06}, in the two dimensional setting, the forcing $f(x,u) = \sum_{k=1}^{N} \lambda_k u^{k-1}$, which is a polynomial of order $N$, is obviously analytic. It implies that the functional $\En$ satisfies the \L{}ojasiewicz inequality \cite{SL-63}. In our case, we consider $p$ as in \eqref{eqn-H_0^1-in-L^p-asp}, in particular, we assume that $p$ is even, so that the nonlinear function $F(u)=|u|^p = u^p$ is	analytic.
	\end{remark}
	
	\begin{remark}
		It is worth emphasizing that when the analysis is restricted to positive solutions on $\mathcal{O}$, the nonlinearity $F(u) = |u|^p$ becomes analytic for any integer exponent $p$ satisfying the condition in \eqref{eqn-H_0^1-in-L^p-asp}, without the need for $p$ to be even. This analyticity enables the application of the \L{}ojasiewicz-Simon inequality to establish convergence. Moreover, under the additional assumption of positivity,  uniqueness of the stationary state can also be guaranteed; see \cite[Subsection 5.2]{AB+ZB+MTM-25+} for further details.
	\end{remark}
	\noindent
	Before stating our result, in the context of Theorem \ref{Thm-LSI}, let us fix $\U =\Vn = D(A) = H^2(\bO)\cap H_0^1(\bO)$ and $\Hn=L^2(\bO)$  throughout this section. Then, by Sobolev embedding $D(A)\embed L^p(\bO)$, we also have
	\begin{align}
		\Vn = D(A) \embed L^p(\bO)\cap  H_0^1(\bO) \embed L^2(\bO) = \Hn,
	\end{align}
	for $p$ fixed in \eqref{eqn-H_0^1-in-L^p-asp}, and the embedding of $\Vn\embed\Hn$ is dense.

	Note that the energy $\En$ defined already in \eqref{Def-En} merely require $u\in L^p(\bO)\cap H_0^1(\bO)$, but an obvious space for $L^2-$gradient flow, suggested by the embedding  $D(A)\embed L^p(\bO)$, is $D(A)$. So, let us consider the energy and constraint functionals as follows:
	\begin{align}
		\En : D(A) \ni u \mapsto \int_\bO \left[\frac{1}{2}\abs{\nabla u(x)}^2 + \frac{u^p(x)}{p}\right] dx \in \R \label{def-energy}
	\end{align}
	and
	\begin{align}
		\Gn : D(A) \ni u \mapsto \frac{1}{2} \int_\bO \left[u^2(x) - \frac{1}{\abs{\bO}}\right] dx \in \R,\label{def-constraint}
	\end{align}
	respectively.

	\begin{lemma}\label{Lem-Gn-analytic}
		The map $\Gn$, defined in \eqref{def-constraint}, is analytic with
		\begin{align*}
			\Gn^\prime(u) h= \int_\bO g^\prime(u(x)) h(x)dx,
		\end{align*}
		where $h\in L^2(\bO)$ and 
		\begin{align}\label{Def-g}
			g(u) = \frac{1}{2} \left(u^2 - \frac{1}{\abs{\bO}}\right).
		\end{align}
		%	It implies $g^\prime(u) = u$.
		In particular, $\Gn^\prime(u) = g^\prime(u)=u\in  L^2(\bO)$.
	\end{lemma}
	
	\begin{proof}
		First note from \eqref{Def-g} that $g(y)=\frac{1}{2} \left(y^2 - \frac{1}{\abs{\bO}}\right)$ is a quadratic polynomial, it is obvious that $g$ is analytic.
		Since $D(A)	\embed C(\overline{\bO})$ is continuous and dense, for $d\leq3$, it implies that the map 
		$$ D(A) \ni u \mapsto u \in C(\overline{\bO})$$ is analytic. By Theorem \ref{Thm-superposition}, the map $D(A) \ni u \mapsto g(u) \in C(\overline{\bO})$ is analytic. As the integral operator is linear and bounded, thus from Example \ref{Lem-analytic}, it is also analytic. By using the fact that composition of analytic functions is analytic, finally Theorem \ref{thm-analytic} asserts that $\Gn$ is analytic. 
		
		Now, for any fixed $u,h\in L^2(\bO)$, we assert
		\begin{align*}
			\Gn^\prime(u)h = \frac{d}{d\eps} \Gn(u+\eps h) \bigg|_{\eps=0} = \int_{\bO} \frac{d}{d\eps} g(u(x) + \eps h(x)) dx \bigg|_{\eps=0} = \int_{\bO}u(x) h(x) dx = (u, h).
		\end{align*}
		Hence it implies $\Gn^\prime(u) = g^\prime(u)=u \in  L^2(\bO)$.
	\end{proof}

	\begin{lemma}\label{Lem-Gn''-compact}
		Let us fix $u\in\mathcal{V}$. 
		The operator $\Gn^{\prime\prime}(u) :D(A) \to L^2(\bO)$ is compact.
	\end{lemma}
	
	\begin{proof}
		Let us first choose and fix $u\in D(A)$. For each fixed $h\in D(A)$, we find
		\begin{align*}
			\Gn^{\prime\prime}(u)h = \frac{d}{d\eps} \Gn^\prime(u+\eps h) \bigg|_{\eps=0} =  \frac{d}{d\eps} g^\prime(u+\eps h) \bigg|_{\eps=0} =g^{\prime\prime}(u)h  = h\in L^2(\bO).
		\end{align*}
		It implies $D(A)\ni h \mapsto \Gn^{\prime\prime}(u)h = h \in L^2(\bO)$, for all $u,h\in D(A)$. {Since the domain $\bO$ is bounded, it follows by the Rellich-Kondrachov Theorem that the injection map (or embedding) $D(A)\ni h\mapsto h\in L^2(\bO)$ is compact,
			%, the $u\inD(A)\embed C(\overline{\bO})$ and hence $$D(A)\ni v\mapsto g^{\prime\prime}(u) \in L^2(\bO)$$ is compact. 
			which further implies that $\Gn^{\prime\prime}(u):D(A) \to L^2(\bO)$ is compact.}
	\end{proof}
	
	\begin{lemma}\label{Lem-En-analytic}
		The energy functional $\En$, defined in \eqref{def-energy}, is analytic with
		\begin{align}
			\En^\prime(u)h =  \int_\bO [-\Delta u + u^{p-1}]h dx = (-\Delta u + u^{p-1}, h), \ \text{ for }\ u,h \in D(A).
		\end{align}
		Moreover, $\En^\prime(u) \in L^2(\bO),$ for all $u\in D(A)$.
	\end{lemma}
	
	\begin{proof}
		Let us first recall that, we are dealing even $p$ (see \eqref{eqn-H_0^1-in-L^p-asp}), therefore, for each fixed $u\in D(A)$, it follows that $u^p$ is analytic. Moreover, using Example \ref{Lem-analytic}, the following maps are analytic:
		\begin{itemize}
			\item The map 
			$$D(A)\ni u \mapsto \abs{\nabla u}^2 \in C(\overline{\bO}),$$
			as it is the sum of diagonal entries of the bounded bilinear map $$D(A) \ni u \mapsto (\partial_i u, \partial_j u)\in \R, \ \text{ for }\ 1\leq i,j \leq d.$$
			\item The map 
			$$C(\overline{\bO}) \ni w \mapsto \int_\bO w(x) dx \in \R$$ is linear and bounded.
		\end{itemize}
		Since the energy functional $\En$ can be written as the composition of the above stated maps,  Theorem \ref{thm-analytic} implies that $\En$ is analytic. Now for $u\in D(A)$, we calculate
		\begin{align*}
			\En^\prime(u)h & = \frac{d}{d\eps} \En(u+\eps h) \bigg|_{\eps=0} = \int_{\bO} \frac{d}{d\eps} \bigg[\frac{\abs{\nabla(u+\eps h)}^2}{2} + \frac{(u+\eps h)^p}{p}\bigg] dx \bigg|_{\eps=0}\\
			& = \int_{\bO}[\nabla u\cdot \nabla h + u^{p-1}h]dx = (-\Delta u + u^{p-1}, h).
		\end{align*}
		Again by an application of the Sobolev embedding Theorem, $D(A) \embed L^p(\bO)$ and for $u\in D(A)$, we have 
		\begin{align}\label{eqn-En-in-L^2}
			\nabla \En(u) = \En^\prime(u) = -\Delta u + u^{p-1}\ \text{ in }\ L^2(\bO),
		\end{align}
		which establishes the required result.
	\end{proof}
	
	\begin{lemma}\label{Lem-En'-analytic}
		The function $D(A) \ni u \mapsto \En^\prime(u) \in L^2(\bO)$ is analytic.
	\end{lemma}
	
	\begin{proof}
		Note from \eqref{eqn-En-in-L^2} (see Lemma \ref{Lem-En-analytic}) that $$\En^\prime(u) = -\Delta u + u^{p-1} \ \text{ in }\ L^2(\bO).$$ 
		Then, Example \ref{Lem-analytic} and the embedding $D(A) \embed L^p(\bO)$ asserts that the following maps are analytic:
		\begin{align*}
			&D(A) \ni u \mapsto \partial_i\partial_j u \in 
			L^2(\bO)\ \text{ for any }\ 1\leq i,j \leq d,\\
			& D(A) \ni u \mapsto u^{p-1}  \in C(\overline{\bO})\embed L^2(\bO).
		\end{align*}
		Since the sum of two analytic functions is analytic, it follows that the map
		%by Theorem \ref{Thm-superposition}, 
		$u\mapsto \En^\prime(u)$ is analytic.
	\end{proof}
	
	\begin{lemma}\label{Lem-En''-Fredholm}
		The Fr\'echet derivative $\En^{\prime\prime}(u):D(A) \to L^2(\bO)$ is Fredholm of index zero.
	\end{lemma}
	
	\begin{proof}
		Let us choose and fix $u\in D(A)$. Then, for each fixed $h\in D(A)$, Lemma \ref{Lem-En-analytic} yields
		\begin{align}
			\En^{\prime\prime}(u)h = \frac{d}{d\eps} \En^\prime(u+\eps h)\bigg|_{\eps=0} = -\Delta h + (p-1) u^{p-2} h =: (A+ K)h.
		\end{align}
		Since the operator $\En^{\prime\prime}(u):\Vn= D(A) \to L^2(\bO)$ is a compact perturbation 
		(\cite[Theorem 7.26]{DG+NST-01}) of the Dirichlet-Laplacian $ -\Delta :\Vn= D(A) \to L^2(\bO)$, which is an isomorphism by standard elliptic theory \cite[see Theorem 9.15]{DG+NST-01}. It follows by Proposition \ref{Prop-Fredholm} that the map $\En^{\prime\prime}(u)$ is Fredholm of zero index.
	\end{proof}
	
	Next, using the supporting results above, we derive the \L{}ojasiewicz-Simon gradient inequality within our framework, where the constraint set is $\bM$.

	\begin{theorem}[{\cite[Theorem 1.4]{FR-20}}]\label{Thm-LSI-our-setting}
		Let $\a \in D(A),$ $\Gn(\a) =0$ with $g^\prime(\a) \not\equiv 0$, be a constraint critical point of $\En|_\bM$, where 
		\begin{align*}
			\bM = \{u\in L^2(\bO): \Gn(u) = 0\}.
		\end{align*}
		Then, $\bM$ is locally an analytic submanifold near $a$ and satisfies the \L{}ojasiewicz-Simon gradient inequality on $\bM$, i.e., there exist $C,\sigma>0$ and $\theta \in (0,\frac{1}{2}]$ such that for any $u\in\bM$ with $\|u - \a\|_{D(A)} \leq \sigma$, we have
		\begin{align}
			\abs{\En(u)- \En(\a)}^{1-\theta} &\leq C \norm{\En^\prime (u) - \frac{\int_\bO \En^\prime(u) g^\prime(u) dx}{\int_\bO (g^\prime(u))^2 dx} g^\prime(u)}_{L^2(\bO)}\\
			&= C \norm{\pi_u(\nabla\En (u))}_{L^2(\bO)} = C \norm{\nabla_\bM\En (u))}_{L^2(\bO)}.\label{eqn-Lojasiewicz-Rupp}
		\end{align}
	\end{theorem}
	
	\begin{proof}
		The proof crucially depends upon the \L{}ojasiewicz-Simon inequality given in Theorem \ref{Thm-LSI}.
		\vskip 1mm
		\noindent
		\textbf{Step I. }
		First, note that $D(A)$ is densely embedded in $ L^2(\bO)$. By Lemmas \ref{Lem-En-analytic} and   \ref{Lem-Gn-analytic}, it follows that $\En$ and $\Gn$ are analytic. Moreover, Lemma \ref{Lem-En'-analytic} provides that $\En^\prime(\a) = \nabla \En (\a)$ is analytic from $D(A)$ to $L^2(\bO)$.
		\vskip 1mm
		\noindent
		\textbf{Step II.} The Fredholm property of $\En^{{\prime\prime}}(\a): \Vn \to L^2(\bO)$ holds from  Lemma \ref{Lem-En''-Fredholm}. Using Lemma \ref{Lem-Gn-analytic}, $\Gn^\prime$ extends analytically to satisfy the assumption (iv) of Theorem \ref{Thm-LSI}. Furthermore, Lemma \ref{Lem-Gn''-compact} implies that the Fr\'echet derivative of $\Gn^\prime$, i.e., $\Gn^{\prime\prime}(\a): D(A) \to L^2(\bO)$ is compact. By assumption on $\Gn'(\a) = g^\prime(\a) \not\equiv0$, it implies that the map $\Gn^\prime (\a) $ is surjective from $ L^2(\bO)$ to $\R$. 
		\vskip 1mm
		\noindent
		\textbf{Step III.} Thus, by Proposition \ref{Prop-Submanifold}, the set $\bM$ is locally a manifold near $\a$. Further, by Theorem \ref{Thm-LSI}, there exist $C,\sigma>0$ and $\theta \in (0,\frac{1}{2}]$ such that for $u\in\bM$ with $\norm{u - \a}_{D(A)} \leq \sigma$, we assert
		\begin{align}
			\abs{\En(u)- \En(\a)}^{1-\theta} \leq C \norm{\En^\prime (u) - \frac{\int_\bO \En^\prime(u) g^\prime(u) dx}{\int_\bO (g^\prime(u))^2 dx} g^\prime(u)}_{L^2(\bO)}.\label{eqn-Lojasiewicz-Rupp-1}
		\end{align}
		By the definition of energy $\En$ given in \eqref{def-energy}, we deduce
		\begin{align*}
			\En^\prime (u) - \frac{\int_\bO \En^\prime(u) g^\prime(u) dx}{\int_\bO (g^\prime(u))^2 dx} g^\prime(u) & = -\Delta u + u^{p-1} - \frac{\left(-\Delta u + u^{p-1}, u\right)}{\|u\|_{L^2(\bO)}^2} u\\
			& = -\Delta u + u^{p-1} - \left(\|\nabla{u}\|_{L^2(\bO)} + \|u\|_{L^p(\bO)}^p\right) u\\
			& = \nabla_\bM \En(u),
		\end{align*}
		where we have used \eqref{eqn-tangent-gradient} and the fact that $u\in \bM$. Thus, \eqref{eqn-Lojasiewicz-Rupp-1} becomes
		\begin{align}
			\abs{\En(u)- \En(\a)}^{1-\theta} \leq 
			C \norm{\nabla_\bM \En (u)}_{L^2(\bO)},
		\end{align}
		which completes the proof.
	\end{proof}
	
	\begin{remark}
		Note that if we wish to work with Banach spaces instead of the Hilbert spaces $\Vn$ and $\Hn$, i.e., 
		if we consider $$\U = \Vn = W^{2,q}(\bO)\cap W_0^{1,q}(\bO),\ \Y=L^{q^\ast}(\bO)\ \text{ and }\ \Y^\ast = L^q(\bO),$$ with $\frac{1}{q}+\frac{1}{q^\ast} =1$, then, the embedding $\Vn\embed C(\overline{\bO})$ holds when $q>\frac{d}{2}$ which forces the dimension to rely on the parameter $\beta$, i.e., $d<4\beta$. Precisely, we calculate $$q \in \Big(\max\Big\{2,\frac{d}{2}\Big\}, \frac{2d}{d+4-4\beta}\Big)\ \text{ and }\ q^\ast \in \Big(\frac{2d}{d-4+4\beta}, \max\Big\{2,\frac{d}{d-2}\Big\}\Big).$$
		Then, by an application of \cite[Theorem 1.4]{FR-20} (see Proposition \ref{Prop-Submanifold}), $\bM$ is locally a manifold near $\mathrm{a}$ and a \L{}ojasiewicz-Simon gradient inequality on $\bM$ is satisfied, i.e., 
		there exist $C,\sigma>0$ and $\theta \in (0,\frac{1}{2}]$ such that for any $u\in\bM$ with $\|u - \a\|_{W^{2,q}(\bO)} \leq \sigma$, from \eqref{eqn-tangent-gradient}, we deduce
		\begin{equation}
			\abs{\En(u)- \En(\a)}^{1-\theta} \leq C \norm{\En^\prime (u) - \frac{\int_\bO \En^\prime(u) g^\prime(u) dx}{\int_\bO (g^\prime(u))^2 dx} g^\prime(u)}_{L^q(\bO)} = C \norm{\nabla_\bM \En(u)}_{L^q(\bO)}.\label{eqn-LSI-W^{2,q}}
		\end{equation}
		%\end{itemize}
		
		Since our approach is based on the $L^2(\bO)$-gradient flow (see \eqref{eqn-L2-grad} below), we do not need to rely on the estimate \eqref{eqn-LSI-W^{2,q}}.
	\end{remark}		
	
	\begin{remark}
		It is essential to observe that,  for any $q>2$, the operator $$A: W^{2,q}(\bO)\cap W_0^{1,q}(\bO) \to L^2(\bO)$$ is not an isomorphism and it implies that the operator $$\En^{\prime\prime}(u) = A + K: W^{2,q}(\bO)\cap W_0^{1,q}(\bO) \to L^2(\bO)$$ is not a Fredholm of index zero. Therefore, we cannot get a \L{}ojasiewicz-Simon inequality when $$\Vn=W^{2,q}(\bO)\cap W_0^{1,q}(\bO)\ \text{ and }\ \Hn = L^2(\bO).$$ Similarly, we cannot work with $$\U = \Vn = D(A^\beta)\text{ and }\ \Y=L^{2}(\bO) \cong \Y^\ast,$$ in that case also, the operator $A: D(A^\beta) \to L^2(\bO)$ is not an isomorphism.
	\end{remark}			
	
	We now focus on the convergence analysis results, where we establish the principal result concerning the long-time behavior of the unique global strong solution to problem \eqref{eqn-main-heat-recall}.

	\subsection{Convergence in time}\label{Subsec-Asymp}
	For any fixed $p$ as assumed in \eqref{eqn-H_0^1-in-L^p-asp}, we recall the energy functional:
	\begin{align}
		\En: D(A) \ni u \mapsto \En(u) \in \R,
	\end{align}
	defined as 
	\begin{equation}\label{Def-En-constrained}
		\En(u):= \int_{\bO} \bigg[\frac{1}{2}\abs{\nabla u}^2 + \frac{1}{p}|u|^p \bigg]dx.  
	\end{equation}
	
	\begin{lemma}\label{Lem-En-dissipative}
		Suppose $u\in D(A)\cap\bM$. Then, the energy functional $\En$, defined in \eqref{Def-En-constrained}, is dissipative in time and, for all $t\geq 0$, it satisfies the following energy equality:
		\begin{align}\label{eqn-bound-u_t-En_0}
			\En(u(t)) + \int_0^t \norm{\frac{\partial u(s)}{\partial s}}_{L^2(\bO)}^2ds = \En(u_0).
		\end{align}
	\end{lemma}
	
	\begin{proof}
		Let us choose and fix $u\in D(A)\cap\bM$. Then, the proof follows from Subsection \ref{Subsec-Gradient-flow}.
	\end{proof}

	\begin{lemma}\label{Lem-stat-iff-critical}
		A function $v \in D(A)$ is a stationary solution of the problem
		if and only if $v$ is a critical point of $\En|_\bM$, where $\En$ is defined in \eqref{Def-En-constrained}.
	\end{lemma}
	
	\begin{proof}
		Let us suppose $v\in D(A)$ is a stationary solution of the problem \eqref{eqn-stationary}. If $h\in L^2(\bO)$ is fixed, then from \eqref{eqn-tangent-gradient}, we have
		\begin{align}
			(\nabla_\bM \En(v),h) 
			& =  \left(-\Delta v + \abs{v}^{p-2}v - ( \norm{\nabla v}_{L^2(\bO)}^2 +  \norm{v}_{L^{p}(\bO)}^{p} ) v, h\right)= 0.\label{eqn-gardEn}
		\end{align}
		Thus, $v$ is a critical point of $\En|_\bM$. On the other hand, if we assume $v\in D(A)
		%W^{2,p}(\bO)
		$ is a critical point of $\En|_\bM$, then
		\begin{align*}
			(\nabla_\bM \En(v),h) &= \left(-\Delta v + \abs{v}^{p-2}v - ( \norm{\nabla v}_{L^2(\bO)}^2 + \norm{v}_{L^{p}(\bO)}^{p} ) v, h\right)\\
			&  = 0, \ \text{ for any }\ h \in L^2(\bO).
		\end{align*}
		Since $h$ is arbitrary, it follows that
		\begin{align*}
			\Delta v - \abs{v}^{p-2}v + ( \norm{\nabla v}_{L^2(\bO)}^2 + \norm{v}_{L^{p}(\bO)}^{p} ) v = 0\ \text{ in }\ L^2(\bO).
		\end{align*}
		Hence $v$ is a stationary solution to the problem \eqref{eqn-stationary}.
	\end{proof}
	
	\begin{lemma}\label{Lem-omega-limit-Stat}
		Let us choose and fix $w\in \omega(u)$. Then, $\En(w)$ is a constant and all points in $\omega(u)$ are critical points of $\En|_\bM$, where
		\begin{align}\label{eqn-omega-limit}
			\omega(u) :=\{w\in W^{2,q}(\bO)\cap W_0^{1,q}(\bO): \exists\ t_n \to \infty\  \text{ such that } \ u(t_n) \to w\}.
		\end{align}
	\end{lemma}
	
	\begin{proof}
		First note from the \cite[Proposition 5.12]{AB+ZB+MTM-25+}  that, each $w\in \omega(u)$ solves the stationary equation \eqref{eqn-stationary} in $L^{\frac{p}{p-1}}(\bO)+H^{-1}(\bO)$ and $w\in \bM$. 
		From Lemma \ref{Lem-stat-iff-critical}, it follows that each $w$ is also a critical point of $\En|_\bM$.
		
		Our next aim is to show that $\En(w)$ is a fixed constant, for every $w\in \omega(u)$, i.e.,  $$\En(w_1) = \En(w_2),\ \text{ for every }\ w_i\in \omega(u),\ i=1,2.$$
		By the definition of $\omega(u)$, there exist $\{t_n^i\}_{n\in\N}$, $i=1,2$,
		such that
		\begin{align*}
			\En(u(t_n^i)) \to \En(w_i),\ \text{ as }\ n\to \infty,\ \text{ for }\ i=1,2.
		\end{align*}
		Using \eqref{eqn-bound-u_t-En_0}, for each $i=1,2$, we also have
		\begin{align*}
			\En(u(t_n^1)) = \En(u_0) - \int_{0}^{t_n^1} \norm{\frac{\partial u(t)}{\partial t}}_{L^2(\bO)}^2dt,
		\end{align*}
		and
		\begin{align*}
			\En(u(t_n^2)) = \En(u_0) - \int_{0}^{t_n^2} \norm{\frac{\partial u(t)}{\partial t}}_{L^2(\bO)}^2dt.
		\end{align*}
		Passing to the limit as $n\to \infty$ in the above two equations, and then subtracting, we find 
		\begin{align}
			\En(w_1) - \En(w_2)
			& =\lim_{n\to\infty} \En(u(t_n^1))-\lim_{n\to\infty}\En(u(t_n^2))\\
			& = \int_{0}^{\infty} \norm{\frac{\partial u(t)}{\partial t}}_{L^2(\bO)}^2dt - \int_{0}^{\infty} \norm{\frac{\partial u(t)}{\partial t}}_{L^2(\bO)}^2dt = 0,
		\end{align}
		where we have utilized the fact that $\frac{\partial u}{\partial t} \in L^2(0,\infty; L^2(\bO))$ (see \eqref{eqn-bound-u_t-En_0}).
		Hence $\En(w)$ is a fixed constant, for every $w\in \omega(u)$.
	\end{proof}

	Let us state the main result of this section, motivated from \cite[Theorem 3.2]{PR-06}.%\cite[Section 3]{LS-83} (or \cite[Theorem 3.1]{PR+KHH-98} or \cite[Theorem 3.2]{PR-06}).
	
	\begin{theorem}\label{Thm-u(t)-cgs-u^infty}
		Let  $u_0 \in D(A)\cap\bM$ and suppose that 
		\begin{align*}
			p\in [2,\infty),  \text{ if }\  1\le d \le 3,
			\ \text{ and }\ q\in \left[2, \frac{2d}{d+4-4\beta}\right), \ \beta \in \left(1,\frac{3}{2}\right),
		\end{align*}
		are fixed.
		Assume that $u$ is the unique global strong solution to problem \eqref{eqn-main-heat}, whose existence is guaranteed by Theorem \ref{Thm-Main-Exis}.  Then, there exists $u^\infty\in W^{2,q}(\bO)\cap W^{1,q}_0(\bO)\cap\bM$ which is a stationary solution to the problem \eqref{eqn-stationary} such that 
		\begin{align}\label{eqn-convergence in W^2q}
			\norm{u(t)- u^\infty}_{W^{2,q}(\bO)} \to 0 \ \text{ as }\ t \to \infty.
		\end{align}
		In particular, it holds that
		\begin{align}\label{eqn-convergence in D(A)-1}
			\norm{u(t)- u^\infty}_{D(A)} \to 0 \ \text{ as }\ t \to \infty.
		\end{align}
	\end{theorem}
	
	\begin{remark}
		It is important to note that Theorem \ref{Thm-u(t)-cgs-u^infty} establishes convergence only when the trajectory lies in a neighborhood of $u^\infty$ with respect to the $W^{2,q}-$norm. Therefore, this result characterizes a local type of stability.
	\end{remark}

	\begin{proof}[Proof of Theorem \ref{Thm-u(t)-cgs-u^infty}]
		Let us choose and fix $u_0\in D(A)\cap \bM$. Assume $u$ is the unique  global strong solution to the problem \eqref{eqn-main-heat}, proved in Theorem \ref{Thm-Main-Exis}. 
		
		\vspace{1mm}
		\noindent
		\textbf{Step I.}  First note that, 
		for every $w\in \omega(u)$, Lemma \ref{Lem-omega-limit-Stat} asserts
		\begin{align}\label{omega-has-constant-energy}
			\En(w) = \En_0 = \text{constant},
		\end{align}
		and all the points in $\omega(u)$ are critical points of $\En|_\bM$.
		Thus, by an application of the \L{}ojasiewicz-Simon gradient inequality, see Theorem \ref{Thm-LSI-our-setting}, for each $(\a=) w\in \omega(u) \subset D(A)$, there exist $C(w),\sigma(w)>0$ and $\theta(w) \in (0,\frac{1}{2}]$ such that
		$$\text{ for all }\ v\in\bM\ \text{ with }\ \norm{v-w}_{D(A)} < \sigma(w),$$
		we have
		\begin{equation}\label{eqn-Lojasiewicz-inequality-omega}	
			\abs{\En(v)- \En(w)}^{1-\theta} \leq 
			%C \norm{\En^\prime (u)}_{L^2(\bO)} = 
			C \norm{\nabla_\bM \En (v)}_{L^2(\bO)}.
		\end{equation}
		On the other hand, for each fixed $w \in \omega(u)$, the definition of $\omega(u)$ (see \eqref{eqn-omega-limit}), continuity of $\En$ and \eqref{omega-has-constant-energy} assert that for a given $\eps>0$ ($\eps<<\sigma(w)$), there exists $N\in\N$ and some constant $C>0$ such that
		\begin{equation}
			\norm{u(t_n) - w}_{W^{2,q}(\bO)} < \frac{\eps}{2}\ \text{ and }\ \frac{C}{\theta}\abs{\En(u(t_n)) - \En_0}^\theta < \frac{\eps}{2},\ \text{ for every }\ n\geq N,\label{eqn-conv-est-u(t_n)-w}
		\end{equation}
		where $C>0$ is the same constant as appeared in \eqref{eqn-Lojasiewicz-inequality-omega}.
		We now choose $N\in\N$ large enough so that 
		\begin{align}\label{eqn-L^2-bound}
			\norm{u(t_n) - w}_{L^2(\bO)} < \frac{\eps}{2},\ \text{ for all }\ n\ge N,
		\end{align}
		which is possible due to the Sobolev embedding $W^{2,q}(\bO)\cap W^{1,q}_0(\bO)\embed L^2(\bO).$
		
		\vspace{1mm}
		\noindent
		\textbf{Step II.} Observe  that a collection of open balls, say $\{B(w,\sigma(w))\}_{w\in\omega(u)}$ in $W^{2,q}(\bO)\cap W^{1,q}_0(\bO)$ form an open cover of $\omega(u)$. 
		But from Corollary \ref{Cor-omega-compact}, we have obtained that $\omega(u)$ is a compact set, therefore, there exists a finite subcover, say $\{B(w_j,\sigma(w_j))\}_{j=1}^m$, such that
		\begin{align*}
			\omega(u) \subset \bigcup_{j=1}^m B(w_j,\sigma(w_j)) =: U.
		\end{align*}
		Let us set $\theta = \min\{\theta(w_j) : 1\leq j\leq m\}$. Then, we recall from Corollary \ref{Cor-omega-compact} that
		\begin{align*}
			\lim_{ t \to \infty}\dist (u(t), \omega(u)):= \lim_{ t \to \infty}\inf_{\{w\in \omega(u)\}}\norm{u(t) - w}_{W^{2,q}(\bO)} = 0,
		\end{align*}
		which implies there exist $t_0>0$ such that $u(t) \in U$, for all $t\ge t_0$ and still we can choose $N\in\N$ large enough so that $u(t) \in U$, for any $t\geq t_N$. We denote 
		\begin{equation}\label{eqn-def-T}
			\wtilde{T} := \sup\big\{t\geq t_N : \norm{u(s) - w}_{W^{2,q}(\bO)} <  \sigma(w)/\wtilde{C},\ \text{ for all } \ s\in [t_N, t] \big\},
		\end{equation}
		where $\wtilde{C}$ is the embedding constant of $W^{2,q}(\bO)\cap W^{1,q}_0(\bO) \embed D(A)$ and suppose that $\wtilde{T}<\infty$.  
		Since $\norm{u(s) - w}_{W^{2,q}(\bO)} <  \sigma(w)/\wtilde{C}$, it implies $\norm{u(s) - w}_{D(A)} <  \sigma(w)$, for all $s\in (t_N, \wtilde{T})$. Thus, it is straightforward from \eqref{eqn-Lojasiewicz-inequality-omega} that
		\begin{align}
			C\norm{\nabla_\bM \En(u(t))}_{L^2(\bO)} \geq \abs{\En(u(t)) - \En_0}^{1-\theta}. \label{eqn-LS}
		\end{align}
		Using the first equation of \eqref{eqn-gradient-flow} in \eqref{eqn-En'}, we get 
		\begin{align}
			\frac{d}{dt} \En(u(t)) 
			=  - \norm{\frac{\partial u(t)}{\partial t}}_{L^2(\bO)}\norm{\nabla_\bM \En(u(t))}_{L^2(\bO)}.\label{eqn-En_t-gradEn}
		\end{align}
		By utilizing \eqref{eqn-LS} and  \eqref{eqn-En_t-gradEn}, for $t\in (t_N,\wtilde{T})$, we obtain
		{\begin{align}
				-\frac{d}{dt} ( \En(u(t)) - \En_0)^\theta & = - \theta ( \En(u(t)) -  \En_0)^{\theta-1} \frac{d}{dt} \En(u(t))\\
				& = \theta \frac{\norm{\nabla_\bM \En(u(t))}_{L^2(\bO)}}{( \En(u(t)) - \En_0 )^{1-\theta}}\norm{\frac{\partial u(t)}{\partial t}}_{L^2(\bO)}\label{eqn-L2-grad}\\ 
				%& \geq \frac{\theta}{C} \norm{\frac{\partial u(t)}{\partial t}}_{L^2(\bO)}\\ 
				& \geq \frac{\theta}{C} \norm{\frac{\partial u(t)}{\partial t}}_{L^2(\bO)}.
		\end{align}}
		Integrating the above inequality in time over $(t_N,\wtilde{T})$ and then, using Lemma \ref{Lem-inequality-theta} yield
		{\begin{align}
				\int_{t_N}^{\wtilde{T}} \norm{\frac{\partial u(r)}{\partial r}}_{L^2(\bO)} dr 
				& \leq \frac{C}{\theta}\left(( \En(u(t_N) - \En_0))^\theta - ( \En(u(\wtilde{T})) - \En_0)^\theta\right)\\
				& \leq \frac{C}{\theta}\left(\En(u(t_N))- \En(u(\wtilde{T}))\right)^\theta\\
				& \leq \frac{C}{\theta} (\En(u(t_N))- \En_0)^\theta\\
				& < \frac{\eps}{2},\label{eqn-u_t-bound-t_N-T}
		\end{align}}
		where we have used the bound from \eqref{eqn-conv-est-u(t_n)-w} and the dissipativity of $\En$. In addition, by using \eqref{eqn-L^2-bound}, the Fundamental Theorem of Calculus for Lebesgue integrals and the inequality \eqref{eqn-u_t-bound-t_N-T}, we deduce
		\begin{align*}
			\|u(\wtilde{T})-w\|_{L^2(\bO)} - \frac{\eps}{2}
			& < \|u(\wtilde{T})-w\|_{L^2(\bO)} - \norm{u(t_N)-w}_{L^2(\bO)} \leq \|u(\wtilde{T})-u(t_N)\|_{L^2(\bO)}\\
			& \leq \bigg\|\int_{t_N}^{\wtilde{T}} \frac{\partial u(r)}{\partial r} dr\bigg\|_{L^2(\bO)} \leq \int_{t_N}^{\wtilde{T}} \norm{\frac{\partial u(r)}{\partial r}}_{L^2(\bO)} dr < \frac{\eps}{2}.
		\end{align*}
		It implies
		\begin{align}\label{eqn-L^2-cgs}
			\|u(\wtilde{T})-w\|_{L^2(\bO)} < \eps.
		\end{align}
		\textbf{Step III.} Due to the precompactness of the set $\{u(t) : t\geq 1\}$ in $W^{2,q}(\bO)\cap W^{1,q}_0(\bO)$, see Corollary \ref{Cor-omega-compact}, every sequence converging in $L^2(\bO)$ converges in $W^{2,q}(\bO)\cap W^{1,q}_0(\bO)$, it suffices to choose $\eps$ small enough to get
		\begin{align*}
			\|u(\wtilde{T})- w\|_{W^{2,q}(\bO)}  < \frac{\sigma(w)}{\wtilde{C}}.
		\end{align*}
		This immediately contradicts the definition of $\wtilde{T}$ (see \eqref{eqn-def-T}), so that   $\wtilde{T}= \infty$, and the inequality  \eqref{eqn-u_t-bound-t_N-T} becomes
		\begin{align*}
			\int_{t_N}^\infty \norm{\frac{\partial u(r)}{\partial r}}_{L^2(\bO)} dr < \frac{\eps}{2}.
		\end{align*}
		Thanks to Remark \ref{Rmk-C^1} for ensuring that $u \in C^1((0,\infty); L^2(\bO))$, which implies the existence of a limit
		\begin{align*}
			\lim_{t\to\infty}u(t)= u^\infty \ \text{ in }\ L^2(\bO).
		\end{align*}
		Given the uniqueness of the limit and the convergence to points in $\omega(u)$ in the $W^{2,q}(\bO) \cap W^{1,q}_0(\bO)-$norm (as established in \eqref{eqn-conv-est-u(t_n)-w}), we conclude that
		\begin{align*}
			\lim_{t\to\infty}u(t)= w = u^\infty\ \text{ in }\ W^{2,q}(\bO)\cap W^{1,q}_0(\bO),
		\end{align*}
		where $u^\infty$ also solves the stationary problem \eqref{eqn-stationary}. 
		Further, it implies that 
		\begin{align*}
			\norm{u(t)- u^\infty}_{D(A)} \to 0 \ \text{ as }\ t \to \infty.
		\end{align*}
		Hence it shows that the unique solution $u$ to the problem \eqref{eqn-main-heat-recall} converges to a stationary solution $u^\infty$ of \eqref{eqn-stationary}, as $t\to \infty$.
	\end{proof}
	
	\begin{remark}
		If $u^\infty$ is a stationary solution which solves \eqref{eqn-stationary}, i.e.,
		\begin{align*}
			\Delta u^\infty - \abs{u^\infty}^{p-2}u^\infty + \big(\norm{\nabla u^\infty}_{L^2(\bO)}^2 + \norm{u^\infty}_{L^p(\bO)}^p\big)u^\infty = 0,
		\end{align*} 
		then, it follows that
		\begin{align}
			\Delta (-u^\infty) - \abs{-u^\infty}^{p-2}(-u^\infty) + \big(\norm{\nabla (-u^\infty)}_{L^2(\bO)}^2 + \norm{(-u^\infty)}_{L^p(\bO)}^p\big)(-u^\infty) = 0,
		\end{align}
		i.e., $-u^\infty$ also solves \eqref{eqn-stationary}. 
		Hence, in the general case, stationary solutions are not unique, as at least two solutions always exist. However, if we restrict our attention to positive stationary states, then uniqueness is guaranteed by \cite[Proposition 5.12]{AB+ZB+MTM-25+}. Therefore, for integer values of $p$, the asymptotic behavior described in Theorem \ref{Thm-u(t)-cgs-u^infty} aligns with the result stated in \cite[Theorem 1.10]{AB+ZB+MTM-25+}.
	\end{remark}

	\appendix
	\renewcommand{\thesection}{\Alph{section}}
	\numberwithin{equation}{section}
	
	\section{}\label{Sec-Appendix}
	This section presents several key results, including a Representation Theorem, a particular case of the Lions–Magenes Lemma, the Spectral Theorem for self-adjoint operators, {alternative proofs of the two results demonstrated in Subsection \ref{Subsec-Yosida} by spectral measures,} and a basic auxiliary lemma essential to the proof of Theorem \ref{Thm-u(t)-cgs-u^infty}.

	We begin by recalling a well-known result, a Representation Theorem (see \cite[Theorem 2.23]{TK-95}), which was employed to deduce that $D(A^{\frac{1}{2}}) =  H_0^1(\bO)$.
	\begin{theorem}\label{Thm-Rep}
		Let $\mathfrak{f}$ be a densely defined, closed symmetric form, $\mathfrak{f} \ge 0$, let $\Fn= T_{\mathfrak{f}}$ be the associated self-adjoint operator. Then, we have
		\begin{align*}
			D(\Fn^{\frac{1}{2}}) = D(\mathfrak{f})\ \text{ and }\ \mathfrak{f}(u,v) = (\Fn^{\frac{1}{2}}u,\Fn^{\frac{1}{2}}v)\ \text{ for }\ u,v \in D(\mathfrak{f}).
		\end{align*}
	\end{theorem}
	
	Next, we establish a special version of the well-known Lions-Magenes Lemma, which was utilized in the regularity analysis presented in Section \ref{Subsec-Thm-II}. For the Hilbert spaces $D(A) = H^2(\bO)\cap H_0^1(\bO),H_0^1(\bO),$ and $L^2(\bO)$, we know that 
	$$D(A)\embed H_0^1(\bO) \embed L^2(\bO),$$ each space is dense in the following one and the injections are continuous.

	\begin{lemma}\label{AbsC}
		Let a measurable function $u \in L^{2}(0,T;D(A))$ be such that $$\frac{\partial u}{\partial t} \in L^{2}(0,T;L^2(\bO)).$$ Then, $u$ is a.e. equal to a  function continuous from $[0,T]$ into $H_0^1(\bO)$
		%, the real-valued function $$[0,T] \ni t \mapsto \|u(t)\|_{H_0^1(\bO)}^2 \in \R$$ is absolutely continuous 
		and furthermore it holds that
		\begin{align}\label{AbsC1}
			\frac{d}{dt}\fourIdx{}{}{2}{H_0^1(\bO)}{\|u(t)\|}= 2\,\left( A u(t), \frac{\partial u(t)}{\partial t} \right)
		\end{align}
		in the sense of distributions on $[0,T]$.
	\end{lemma}

	{\begin{proof}
			Suppose $u \in L^2(0,T; D(A))$ with $\frac{\partial u}{\partial t} \in L^2(0,T; L^2(\bO))$.
			\vskip 1mm
			\noindent
			\textbf{Step I.} Let $\wtilde{u}:\R\to D(A)$ be the null extension of $u$ defined as $\displaystyle \wtilde{u}(t)=\begin{cases}
				u(t),& \text{ on }[0,T],\\
				0,& \text{ otherwise.}
			\end{cases}$\\
			By regularizing the function $\wtilde{u}$, we can obtain a sequence of smooth functions $\left\{u_{m}\right\}_{m\in\N}\subset C_{0}^{\infty}((0,T);D(A))$ with the property that
			\begin{align}
				\lim_{m\to\infty} \|u_{m} - u\|^2_{L_{\loc}^2((0,T); D(A))} 
				& = 0,\label{AbsC2}\\
				\lim_{m\to\infty} \left\| \frac{\partial u_{m}}{\partial t} - \frac{\partial u}{\partial t} \right\|_{L_{\loc}^{2}((0,T);L^2(\bO))} 
				& = 0.\label{AbsC3}
			\end{align}
			Since the sequence $\left\{u_{m}\right\}_{m\in\N}$ is smooth in time, for all $m\in\N$, it follows through integration by parts formula that
			\begin{align}
				\frac{d}{dt}\fourIdx{}{}{2}{H_0^1(\bO)}{\|u_{m}(t)\|} & = \frac{d}{dt}\left( \nabla u_{m}(t), \nabla u_{m}(t)\right) = 2\left( \nabla u_{m}(t),\frac{\partial \nabla u_{m}(t)}{\partial t}\right)\\
				& = 2\, \left(Au_{m}(t), \frac{\partial u_{m}(t)}{\partial t} \right),\ \text{ for all } \ t\in(0,T). \label{AbsC4}
			\end{align}
			\vskip 1mm
			\noindent
			\textbf{Step II.} Now, let $[a,b]\subset\subset(0,T)$ be an arbitrary interval. Then, by utilizing the integration by parts formula and the self-adjointness of the operator $A$ yield
			\begin{align*}
				& \int_{a}^{b}\big|\fourIdx{}{}{2}{H_0^1(\bO)}{\|u_{m}(t)\|} - \fourIdx{}{}{2}{H_0^1(\bO)}{\|u(t)\|}\big|dt\\ &= \int_{a}^{b} \left| (Au_{m}(t) - Au(t), u_{m}(t)) +  (u(t), Au_{m}(t)) - (u(t), Au(t)) \right| dt,\\
				& = \int_{a}^{b} \left| (Au_{m}(t) - Au(t), u_{m}(t))  +  (u(t), Au_{m}(t) - Au(t)) \right|dt,\\
				& \leq \int_{a}^{b}\left( \left\| u_{m}(t)-u(t) \right\|_{D(A)} \left\|u_{m}(t)\right\|_{L^2(\bO)} + \left\|u(t)\right\|_{L^2(\bO)} \left\|u_{m}(t)-u(t)\right\|_{D(A)} \right) dt\\
				%& \leq \int_{a}^{b}\left( \left\| u_{m}-u \right\|_{D(A)} \big[\left\|u_m \right\|_{L^2(\bO)} + \left\|u \right\|_{L^2(\bO)} \big] \right) dt,\\
				& \leq \left(\int_{a}^{b} \left\| u_{m}(t)-u(t) \right\|_{D(A)}^{2} dt\right)^{\frac{1}{2}} \bigg[\left(\int_{a}^{b} \left\|u_m(t) \right\|_{L^2(\bO)}^{2} dt\right)^{\frac{1}{2}} + \left(\int_{a}^{b} \left\|u(t) \right\|_{L^2(\bO)}^{2} dt\right)^{\frac{1}{2}}\bigg].
			\end{align*}
			Letting $m\to \infty$, it follows from \eqref{AbsC2} and the assumption $u\in L^2(0,T; D(A))$ that
			\begin{equation}\label{AbsC-Norm}
				\|u_{m}\|_{H_0^1(\bO)}^{2} \to \|u\|_{H_0^1(\bO)}^{2}\ \text{ in }\ L_{\loc}^{1}\left((0,T)\right).
			\end{equation}
			Similarly, for any $[a,b]\subset\subset(0,T),$ we calculate
			\begin{align*}
				& \int_{a}^{b} \left| \left( \frac{\partial u_{m}(t)}{\partial t}, Au_{m}(t)  \right) - \left(\frac{\partial u(t)}{\partial t}, Au(t)  \right) \right| dt,\\
				& = \int_{a}^{b} \bigg| \left(\frac{\partial u_{m}(t)}{\partial t}-\frac{\partial u(t)}{\partial t}, Au_{m}(t) - Au(t) \right) + \left(\frac{\partial u_{m}(t)}{\partial t} - \frac{\partial u(t)}{\partial t}, Au(t) \right)\\
				&\qquad\qquad +  \left(\frac{\partial u(t)}{\partial t}, Au_{m}(t) - Au(t) \right) \bigg| dt,\\
				& \leq \int_{a}^{b}  \left\|\frac{\partial u_{m}(t)}{\partial t}-\frac{\partial u(t)}{\partial t}\right\|_{L^2(\bO)} \big[ \left\|u_{m}(t)-u(t)\right\|_{D(A)} + \left\|u(t)\right\|_{D(A)}\big]dt\\
				&\quad + \int_{a}^{b}\left\|\frac{\partial u(t)}{\partial t}\right\|_{L^2(\bO)} \left\|u_{m}(t)-u(t)\right\|_{D(A)} dt\\
				& \leq \bigg(\int_{a}^{b} \left\|\frac{\partial u_{m}(t)}{\partial t}-\frac{\partial u(t)}{\partial t}\right\|_{L^2(\bO)}^2 dt \bigg)^{\frac{1}{2}} \bigg[\bigg( \int_{a}^{b}\left\|u_{m}(t)-u(t)\right\|_{D(A)}^2 dt\bigg)^{\frac{1}{2}} + \bigg(\int_{a}^{b}\left\|u(t)\right\|_{D(A)}^2 dt \bigg)^{\frac{1}{2}} \bigg]\\
				& \quad + \bigg(\int_{a}^{b}\left\|\frac{\partial u(t)}{\partial t}\right\|_{L^2(\bO)}^2 dt\bigg)^{\frac{1}{2}} \bigg(\int_{a}^{b}\left\|u_{m}(t)-u(t) \right\|_{D(A)}^2\bigg)^{\frac{1}{2}},
			\end{align*}
			where the last step follows from H\"older's inequality. Passing $m\to \infty$, by using the fact that $u\in L^2(0,T; D(A))$ and $\frac{\partial u}{\partial t}\in L^2(0,T;L^2(\bO))$, and the convergences \eqref{AbsC2} and \eqref{AbsC3}, we get
			\begin{align*}
				\int_{a}^{b} \left| \bigg(\frac{\partial u_{m}(t)}{\partial t}, Au_{m}(t)\bigg) - \bigg(\frac{\partial u(t)}{\partial t}, Au(t) \bigg) \right| dt \to 0.
			\end{align*}
			Since $[a,b]$ is arbitrary interval, it implies  that
			\begin{equation}\label{AbsC-A}
				\bigg(Au_{m},\frac{\partial u_{m}}{\partial t} \bigg) \to \bigg(Au,\frac{\partial u}{\partial t}\bigg) \text{ in } L_{\loc}^{1}\left((0,T)\right), \ \text{ as }\ m\to \infty.
			\end{equation}
			\vskip 1mm
			\noindent
			\textbf{Step III.} Since the above convergences also hold in the sense of distributions, it follows that we may pass to the limit in \eqref{AbsC4} in the distributional sense. Specifically, fix $\phi \in C_0^\infty((0,T))$; then, integrating \eqref{AbsC4} over $(0,T)$ and applying the integration by parts formula provide
			\begin{align*}
				-\int_{0}^{T}\fourIdx{}{}{2}{H_0^1(\bO)}{\|u_{m}(t)\|}\phi'(t)dt & = 2 \int_{0}^{T} \bigg(Au_{m}(t),\frac{\partial u_{m}(t)}{\partial t}\bigg) \phi(t) dt.
			\end{align*}
			By utilizing the norm convergence from \eqref{AbsC-Norm} and \eqref{AbsC-A} in the above identity, as $m \to \infty$, we deduce
			\begin{align*}
				-\int_{0}^{T}\fourIdx{}{}{2}{H_0^1(\bO)}{\|u(t)\|}\phi'(t)dt & = 2 \int_{0}^{T} \bigg(Au(t),\frac{\partial u(t)}{\partial t}\bigg) \phi(t) dt.
			\end{align*}
			Thus, the following equality holds in the distributional sense in $(0,T)$:
			\begin{align*}
				\frac{d}{dt}\fourIdx{}{}{2}{H_0^1(\bO)}{\|u(t)\|}  = 2\, \bigg(Au(t),\frac{\partial u(t)}{\partial t}\bigg).
			\end{align*}
			We are left to show that $u$ is a.e. equal to a continuous function from $[0,T]$ into $H_0^1(\bO)$.
			\vskip 1mm
			\noindent
			\textbf{Step IV.} Since $u\in L^{2}(0,T;D(A))$ and $\frac{\partial u}{\partial t} \in L^{2}(0,T;L^2(\bO))$, it implies by an application of H\"older's inequality in time and space that
			\begin{align*}
				\int_{0}^{T}\left| \bigg(Au(t), \frac{\partial u(t)}{\partial t} \bigg) \right| dt \leq \left\|u\right\|_{L^{2}(0,T;D(A))} \left\|\frac{\partial u}{\partial t}\right\|_{L^{2}(0,T;L^2(\bO))}< \infty.
			\end{align*}
			That is, $( Au,\frac{\partial u}{\partial t}) \in L^{1}((0,T)).$ 
			%		Therefore,  \cite[Theorem 1.1, p. 169]{RT-01} implies that $$[0,T] \ni t \mapsto \fourIdx{}{}{2}{H_0^1(\bO)}{\|u(t)\|} \in \R \ \text{ is a continuous function.}$$
			Further, by integrating \eqref{AbsC1} from $0$ to $t$, for any $t\in(0,T)$, we get
			\begin{align*}
				\fourIdx{}{}{2}{H_0^1(\bO)}{\|u(t)\|} & = 2\int_{0}^{t} \bigg(Au(s),\frac{\partial u(s)}{\partial s} \bigg) ds + \fourIdx{}{}{2}{H_0^1(\bO)}{\|u(0)\|}.
				%& \leq  2\int_{0}^{T}\left| \bigg(Au(t),\frac{\partial u}{\partial t}(t) \bigg) \right| dt + \fourIdx{}{}{2}{H_0^1(\bO)}{\|u(0)\|}.
			\end{align*}
			Taking supremum over $t \in [0,T]$ and by using the fact $(Au,\frac{\partial u}{\partial t}) \in L^{1}((0,T))$, we obtain
			\begin{align*}
				\sup_{t\in[0,T]} \fourIdx{}{}{2}{H_0^1(\bO)}{\|u(t)\|} \leq  2\int_{0}^{T}\left| \bigg(Au(t),\frac{\partial u(t)}{\partial t} \bigg) \right| dt + \fourIdx{}{}{2}{H_0^1(\bO)}{\|u(0)\|} < \infty.
			\end{align*}
			This implies that the function $\nabla u \in L^{\infty}(0,T;L^2(\bO))$. On the hand, from the assumption $\nabla u\in W^{1,2}(0,T;H^{-1}(\bO))$, we also have $u\in C([0,T]; H^{-1}(\bO)) \embed C_w([0,T]; H^{-1}(\bO))$. Consequently, we obtain
			\begin{align*}
				\nabla	u \in C_w([0,T]; H^{-1}(\bO)) \cap L^{\infty}(0,T;L^2(\bO)).
			\end{align*}
			Therefore, by \cite[Theorem 2.1]{WAS-66}, $\nabla u$ is a weakly continuous function from $[0,T]$ to $L^2(\bO)$. That is, $\nabla u\in C_w([0,T]; L^2(\bO))$ and   for all $v\in L^2(\bO)$ the mapping
			\[ [0,T] \ni t \mapsto \left(\nabla u(t),v\right) \in \R \ \text{ is continuous.}\]
			\vskip 1mm
			\noindent
			\textbf{Step V.} Note that, for all $t_{0},t\in[0,T]$
			\begin{align}\label{AbsC5}
				\fourIdx{}{}{2}{H_0^1(\bO)}{\|u(t)-u(t_{0})\|} = \fourIdx{}{}{2}{H_0^1(\bO)}{\|u(t)\|}  +\fourIdx{}{}{2}{H_0^1(\bO)}{\|u(t_0)\|} - 2\left( \nabla u(t), \nabla u(t_{0})\right).
			\end{align}
			Integrating \eqref{AbsC1} from $t_{0}$ to $t$, for any $t_{0},t\in[0,T]$, and using the absolute continuity of the Lebesgue integrals, we deduce 
			\begin{equation}\label{AbsC6}
				\fourIdx{}{}{2}{H_0^1(\bO)}{\|u(t)\|} = 2\int_{t_{0}}^{t} \bigg(Au(s),\frac{\partial u(s)}{\partial t}\bigg) ds + \fourIdx{}{}{2}{H_0^1(\bO)}{\|u(t_0)\|} \to  \fourIdx{}{}{2}{H_0^1(\bO)}{\|u(t_0)\|} \text{ as }\  t\to t_{0}.
			\end{equation}
			Letting $t \to t_0$ in \eqref{AbsC5}, by using the continuity of the map $t\mapsto\left(\nabla u(t), \cdot\right)$ and the convergence \eqref{AbsC6}, we finally have
			\begin{align*}
				\lim_{t\to t_{0}} \fourIdx{}{}{2}{H_0^1(\bO)}{\|u(t)-u(t_{0})\|} &= \lim_{t\to t_{0}} \fourIdx{}{}{2}{H_0^1(\bO)}{\|u(t)\|}   + \fourIdx{}{}{2}{H_0^1(\bO)}{\|u(t_0)\|} - 2\lim_{t\to t_{0}} \left( \nabla u(t),\nabla u(t_{0})\right)= 0.
			\end{align*}
			Hence $u$ is a continuous function from $[0,T]$ to $H_0^1(\bO)$, which completes the proof.
		\end{proof}
	}
	
	%Finally, we close this section by recalling one of the famous results from the Functional Calculus called Spectral Theorem for self-adjoint operators, motivated from \cite[p. 151]{WOA-09} {or cf. \cite[Theorem 5.7]{KS-12}}. 
	Now, we recall one of the famous results from the Functional Calculus called Spectral Theorem for self-adjoint operators, motivated from \cite[p. 151]{WOA-09} {or cf. \cite[Theorem 5.7]{KS-12}}. 
	
	\begin{theorem}\label{Thm-spectral}
		Let $\bA$ be a self-adjoint operator on a Hilbert space $\Hn$. Then, there exists a unique spectral measure $E$, associated with the spectral family $\{E_\lambda\}_{\lambda\in\R}$, on Borel $\sigma-$algebra $\Bb(\R)$ such that
		\begin{align*}
			\bA = \int_\R \lambda E(d\lambda),
		\end{align*}
		where 
		\begin{align*}
			D(\bA) = \left\{h \in \Hn : \int_\R \lambda^2 m_h(d\lambda) = \int_\R \lambda^2 (h, E(d\lambda)h) < \infty \right\}.
		\end{align*}
	\end{theorem}
	
	The following remark is needed for the completeness of the alternative proofs already proven Propositions in Section \ref{Subsec-Thm-II}.
	
	{\begin{remark}
			\noindent
			\begin{itemize}
				\item[$(i)$] By an application of the Spectral Theorem, see Theorem \ref{Thm-spectral}, for the Dirichlet-Laplacian $A$ on $L^2(\bO)$, there exists a unique spectral measure $E$ associated to the spectral family $\{E_\lambda\}_{\lambda \ge 0}$ such that 
				\begin{align}\label{eqn-A-spec-meas}
					A = \int_0^\infty \lambda E(d\lambda).
				\end{align}
				\item[$(ii)$] From \cite[Proposition 4.10 (ii)]{WOA-09}, for each $u\in L^2(\bO)$ and a continuous function $\varphi: [0,\infty) \to \R$, we know that
				\begin{align}\label{eqn-phi-E-iden}
					\int_0^\infty \abs{\varphi(\lambda)}^2 m_u (d\lambda) = \norm{\Big[\int_0^\infty \varphi(\lambda) E(d\lambda) \Big] u }_{L^2(\bO)}^2,
				\end{align}
				where $m_u: \Bb([0, \infty)) \to [0,1]$ is a Stieltjes measure, see \cite[Theorem 3.3]{KS-23} or \cite[p. 134]{WOA-09}, defined as 
				\begin{align*}
					m_u (\Vs) = \norm{E(\Vs) u}_{L^2(\bO)}^2 = (u, E(\Vs)u),\ \text{ for any }\ \Vs \in \Bb([0,\infty)).
				\end{align*}
			\end{itemize}
	\end{remark}}

	Next, we present an alternative and more general approach to establish a convergence result and a regularity result.
	
	{\begin{proof}[Alternative proof of Proposition \ref{Lem-mu->u-cgs}]
			Using the functional calculus (\cite[Subsection 5.3]{KS-12}) and \eqref{eqn-A-spec-meas}, we have 
			\begin{align}
				\sup_{t\in[0,T]} \norm{Au_\mu(t) - Au(t)}_{L^2(\bO)}^2 
				& = \sup_{t\in[0,T]} \norm{[\mu A(\mu I + A)^{-1}  - A]u(t)}_{L^2(\bO)}^2\\
				& = \sup_{t\in[0,T]} \norm{ \bigg[\int_0^\infty \Big[\frac{\mu\lambda}{\mu + \lambda} - \lambda \Big]E(d\lambda)\bigg]u(t)}_{L^2(\bO)}^2 \\
				& = \sup_{t\in[0,T]} \norm{ \bigg[\int_0^\infty \frac{-\lambda^2}{\mu + \lambda}E(d\lambda)\bigg]u(t)}_{L^2(\bO)}^2. \label{eqn-L^inf-D(A)-mu-1}
			\end{align}
			Choosing $\varphi(\lambda) = \frac{-\lambda^2}{\mu + \lambda}$ in \eqref{eqn-phi-E-iden} gives
			\begin{align*}
				\norm{\bigg[\int_0^\infty \frac{-\lambda^2}{\mu + \lambda} E(d\lambda) \bigg] u }_{L^2(\bO)}^2 &= \int_0^\infty \abs{\frac{-\lambda^2}{\mu + \lambda}}^2 m_u (d\lambda)\leq  \int_0^\infty \lambda^2 m_{u} (d\lambda) \nonumber\\&=\norm{\Big[\int_0^\infty \lambda^2 E(d\lambda) \Big] u }_{L^2(\bO)}^2,
			\end{align*}
			where we have used the fact that $\lambda \le \mu + \lambda$.
			Using the above identity in \eqref{eqn-L^inf-D(A)-mu-1} yields
			\begin{align}
				\sup_{t\in[0,T]} \norm{Au_\mu(t) - Au(t)}_{L^2(\bO)}^2&  \leq  \sup_{t\in[0,T]} \norm{\Big[\int_0^\infty \lambda^2 E(d\lambda) \Big] u(t) }_{L^2(\bO)}^2\\
				&= \sup_{t\in[0,T]}\norm{Au(t)}_{L^2(\bO)}^2 < \infty,\label{eqn-est-D(A)-mu-1}
			\end{align}
			since $u\in L^\infty(0,T;D(A))$. 
			Letting $\mu \to \infty$ in \eqref{eqn-L^inf-D(A)-mu-1}, we deduce
			\begin{align*}
				\lim_{\mu\to \infty} \sup_{t\in[0,T]} \norm{Au_\mu(t) - Au(t)}_{L^2(\bO)} = \lim_{\mu\to \infty}  \norm{u_\mu - u}_{L^\infty(0,T;D(A))} = 0.
			\end{align*}
			It is suffices to close the proof.
		\end{proof}
	}
	
	In a similar spirit of the above proof, we can prove Proposition \ref{Prop-A^{3/2}} as follows:
	
	{\begin{proof}[Alternative proof of Proposition \ref{Prop-A^{3/2}}]
			By using \eqref{eqn-phi-E-iden} for $\varphi(\lambda) = \frac{\lambda^{\frac{3}{2}} \mu}{\mu + \lambda}$ and functional calculus, we obtain
			\begin{align}
				\norm{A^{\frac{3}{2}} u_\mu (t)}_{L^2(\bO)}^2  
				& = \norm{A^{\frac{3}{2}} \mu (\mu I + A)^{-1}u (t)}_{L^2(\bO)}^2 =  \norm{ \bigg[\int_0^\infty \frac{\lambda^{\frac{3}{2}} \mu}{\mu + \lambda} E(d\lambda)\bigg]u(t) }_{L^2(\bO)}^2\\
				& =\int_0^\infty \frac{\lambda^3\mu^2}{(\mu + \lambda)^2} m_{u(t)}(d\lambda) \le \mu  \int_0^\infty \lambda^2 m_{u(t)} (d\lambda) \\
				& = \mu \norm{ \bigg[\int_0^\infty \lambda E(d\lambda)\bigg] u(t) }_{L^2(\bO)}^2 = \mu\norm{Au(t)}_{L^2(\bO)}^2 < \infty,\label{eqn-A3/2-1}
			\end{align}
			where we have utilized \eqref{eqn-A-spec-meas}  and the fact that $\mu, \lambda \le \mu + \lambda$.  For each $\mu>0$, the above estimate implies that $u_\mu \in L^2(0,T; D(A^{\frac{3}{2}}))$. In a similar way, for each $\mu>0$, using \eqref{eqn-A-spec-meas}, we also have
			\begin{align*}
				\int_0^T\norm{A^{\frac{1}{2}} \frac{\partial u_\mu (t)}{\partial t} }_{L^2(\bO)}^2 dt
				& = \int_0^T \norm{A^{\frac{1}{2}} \mu (\mu I + A)^{-1}\frac{\partial u(t)}{\partial t} }_{L^2(\bO)}^2 dt \\
				& = \int_0^T \norm{\bigg[\int_0^\infty \frac{\lambda^{\frac{1}{2}} \mu}{\mu + \lambda} E(d\lambda)\bigg] \frac{\partial u}{\partial t} (t)}_{L^2(\bO)}^2 dt \\
				& \le \int_0^T  \int_0^\infty \frac{\lambda \mu^2}{(\mu + \lambda)^2} m_{\frac{\partial u(t)}{\partial t} }(d\lambda) dt\\
				& \leq \int_0^T \int_0^\infty \mu\, m_{\frac{\partial u(t)}{\partial t} }(d\lambda)dt  =  \mu\int_0^T m_{\frac{\partial u(t)}{\partial t} }([0,\infty)) dt\\
				& = \mu\int_0^T\norm{\frac{\partial u(t)}{\partial t}}_{L^2(\bO)}^2 dt  \le \mu \norm{\frac{\partial u}{\partial t}}_{L^2(0,T;L^2(\bO))}^2  < \infty,
			\end{align*}
			where we have used \eqref{eqn-phi-E-iden} for $\varphi(\lambda) = \frac{\lambda^{\frac{1}{2}}}{\mu + \lambda}$, the fact that $\mu, \lambda \le \mu + \lambda$ and $ \frac{\partial u}{\partial t} \in L^2(0,T; L^2(\bO))$. Thus, for each $\mu>0$, we have 
			\begin{align}\label{eqn-u'_mu-H_0^1-1}
				\frac{\partial u_\mu}{\partial t} \in L^2(0,T; H_0^1(\bO)) \implies 	\frac{\partial A u_\mu}{\partial t} \in L^2(0,T; H^{-1}(\bO)). 
			\end{align}
			The rest of the proof follows identical to the proof of Proposition \ref{Prop-A^{3/2}}.
	\end{proof}}
	
	To conclude this section, we present an auxiliary result that is used to establish the asymptotic result in Section \ref{Sec-Asy-anal}.

	\begin{lemma}\label{Lem-inequality-theta}
		Let $a,b\in [0,\infty)$ be such that $a\leq b$. For $0<\theta<1$, we have the following inequality:
		\begin{align}\label{eqn-inequality-theta}
			(a+b)^\theta \leq a^\theta + b^\theta.
		\end{align}
		Moreover, it implies that
		\begin{align}\label{eqn-tri-in}
			b^\theta - a^\theta \leq (b-a)^\theta.
		\end{align}
	\end{lemma}
	
	\begin{proof}
		Let us fix $0<\theta<1$ and assume that $a,b\in [0,\infty)$ be such that $0\leq a\leq b$.
		\vskip 1mm
		\noindent
		\textbf{Case I.} If $a=0$ or $b=0$, then $(a+b)^\theta = b^\theta \leq  a^\theta + b^\theta = b^\theta$ or $(a+b)^\theta = a^\theta \leq  a^\theta + b^\theta = a^\theta$, respectively.
		\vskip 1mm
		\noindent
		\textbf{Case II.} If $a=b$, $2^\theta a^\theta \leq 2 a^\theta$, which is equivalent to have $2^\theta\leq 2$.
		\vskip 1mm
		\noindent
		\textit{Case III.} When $0<a<b$, we divide the inequality \eqref{eqn-inequality-theta} by $b^\theta$ to obtain
		\begin{align*}
			\left(\frac{a}{b} + 1\right)^\theta \leq \left(\frac{a}{b}\right)^\theta + 1.
		\end{align*}
		Suppose $\gamma = \frac{a}{b}$, then the above inequality will become
		\begin{align*}
			(\gamma + 1)^\theta \leq \gamma^\theta + 1,\ \text{ for }\ 0 < \gamma < 1,
		\end{align*}
		and we need to show that the above inequality is valid. 	Let us now define
		\begin{align*}
			f(\gamma) := (\gamma + 1)^\theta - \gamma^\theta - 1.
		\end{align*}
		Note that
		\begin{align*}
			f(0) = 0\ \text{ and }\ f(1)= 2^\theta - 2 \leq 0.
		\end{align*}
		Moreover, the function $f$ is decreasing, since
		\begin{align*}
			f'(\gamma) = \theta (\gamma+ 1)^{\theta-1} - \theta \gamma^{\theta-1} = \theta \gamma^{\theta-1}\left[\left(1 + \frac{1}{\gamma}\right)^{\theta-1} -1\right] \leq 0.
		\end{align*}
		Since $f$ is decreasing, for any $0 \leq \gamma$, we have
		\begin{align*}
			f(\gamma )& \leq f(0)\Rightarrow
			(\gamma + 1)^\theta - \gamma^\theta - 1  \leq 0\Rightarrow
			(\gamma + 1)^\theta   \leq \gamma^\theta + 1,
		\end{align*}
		which completes the proof of \eqref{eqn-inequality-theta}. In order to show \eqref{eqn-tri-in}, we use \eqref{eqn-inequality-theta} to find
		\begin{align}
			b^{\theta}=(b-a+a)^{\theta}\leq (b-a)^{\theta}+a^{\theta},
		\end{align}
		and it competes the proof.
	\end{proof}

	\appendix

	\medskip\noindent
	\textbf{Acknowledgments:} 
The first three authors would like to thank the Bernoulli Center for Fundamental Studies at EPFL (Lausanne) for its support and hospitality during the programme ``New Developments and Challenges in Stochastic Partial Differential Equations'', held from June to August 2024, where part of the research leading to this article was carried out.
	A. Bawalia gratefully acknowledges the \textit{National Board of Higher Mathematics} (NBHM) for the travel grant to support the above visit (Sanction No.: 0207/9(2)/2024-R$\&$D-II/10984), and the \textit{University Grants Commission} (UGC), Govt. of India, for financial assistance (File No.: 368/2022/211610061684). M.T. Support for M. T. Mohan's research received from the National Board of Higher Mathematics (NBHM), Department of Atomic Energy, Government of India (Project No. 02011/13/2025/NBHM(R.P)/R\&D II/1137).
	
	\medskip\noindent	\textbf{Declarations:} 
	
	\noindent 	\textbf{Ethical Approval:}   Not applicable 
	
	%\noindent  \textbf{Competing interests: } The authors declare no competing interests. 
	
	\noindent  \textbf{Conflict of interest: }On behalf of all authors, the corresponding author states that there is no conflict of interest.
	
	\noindent 	\textbf{Authors' contributions:} All authors have contributed equally. 
	
	%	\noindent 	\textbf{Funding: } DST-SERB, India, MTR/2021/000066 (M. T. Mohan). 
	
	\noindent 	\textbf{Availability of data and materials:} Not applicable. 
	
	\bibliographystyle{plain}
	\bibliography{BBMR_ref}
	
\end{document}